\documentclass[11pt]{amsart}

\usepackage{amssymb,mathrsfs,amsthm}
\usepackage{graphicx}
\usepackage{subfigure}

\usepackage{enumerate}        

\newcommand {\Z}{\mathbb{Z}}

\newcommand {\bng}{B_n\Gamma}

\newcommand {\ucng}{U\mathcal{C}^n\Gamma}

\newcommand {\udng}{U\mathcal{D}^n\Gamma}
\newcommand {\uc}[2]{U\mathcal{C}^{#1} #2}
\newcommand {\ud}[2]{U\mathcal{D}^{#1} #2}

\newcommand {\leqn}{\leq_{{}_N}\hspace{-1mm}}
\newcommand {\en}{=_{{}_N}\hspace{-1mm}}

\newtheorem{theorem}{Theorem}[section]
\newtheorem{lemma}[theorem]{Lemma}
\newtheorem{proposition}[theorem]{Proposition}
\newtheorem{corollary}[theorem]{Corollary}
\newtheorem{conjecture}[theorem]{Conjecture}
\theoremstyle{definition}
\newtheorem{definition}[theorem]{Definition}
\newtheorem{convention}[theorem]{Convention}
\newtheorem{example}[theorem]{Example}

\newtheorem*{theoremA}{Theorem A}
\newtheorem*{theoremB}{Theorem B}
\newtheorem*{theoremC}{Theorem C}

\begin{document}

\title[Rigidity of tree braid groups]{On rigidity and the 
isomorphism problem for tree braid groups}
\author[L.Sabalka]{Lucas~Sabalka}
      \address{\tt Department of Mathematical Sciences\\
               Binghamton U\\
               Binghamton, NY 13902-6000
\newline       http://math.binghamton.edu/sabalka}
      \email{sabalka at math.binghamton.edu}

\begin{abstract}

We solve the isomorphism problem for braid groups on trees with $n = 4$ 
or $5$ strands.  We do so in three main steps, each of which is 
interesting in its own right.  First, we establish some tools and 
terminology for dealing with computations using the cohomology of tree 
braid groups, couching our discussion in the language of differential 
forms.  Second, we show that, given a tree braid group $B_nT$ on $n = 4$ 
or $5$ strands, $H^*(B_nT)$ is an exterior face algebra.  Finally, we 
prove that one may reconstruct the tree $T$ from a tree braid group 
$B_nT$ for $n = 4$ or $5$.  Among other corollaries, this third step 
shows that, when $n = 4$ or $5$, tree braid groups $B_nT$ and trees $T$ 
(up to homeomorphism) are in bijective correspondence.  That such a 
bijection exists is not true for higher dimensional spaces, and is an 
artifact of the $1$-dimensionality of trees.  We end by stating the 
results for right-angled Artin groups corresponding to the main 
theorems, some of which do not yet appear in the literature.

\end{abstract}



\maketitle

\section{Introduction}\label{sec:intro}

Given a graph $\Gamma$, the \emph{unlabelled configuration space} 
$\ucng$ of $n$ points on $\Gamma$ is the space of $n$-element subsets of 
distinct points in $\Gamma$. The \emph{$n$-strand braid group of 
$\Gamma$}, denoted $\bng$, is the fundamental group of $\ucng$.  If 
$\Gamma$ is a tree, $\bng$ is a \emph{tree braid group}.

Graph braid groups are of interest because of their connections with 
classical braid groups (see, for instance, \cite{Sabalka4}) and 
right-angled Artin groups \cite{CrispWiest,Sabalka,FarleySabalka2a}, as 
well as connections in robotics and mechanical engineering.  Graph braid 
groups can, for instance, model the motions of robots moving about a 
factory floor \cite{Ghrist,Farber,Farber2}, or the motions of 
microscopic balls of liquid on a nano-scale electronic circuit 
\cite{GhristPeterson}.

Ghrist \cite{Ghrist} showed that the complexes $\uc{n}{\Gamma}$ are 
$K(B_n\Gamma,1)$ spaces. Abrams \cite{Abrams1} showed that graph braid 
groups are fundamental groups of locally CAT(0) cubical complexes, and 
so for instance have solvable word and conjugacy problem 
\cite{BridsonHaefliger}.  Crisp and Wiest \cite{CrispWiest} showed that 
any graph braid group embeds in some right-angled Artin group, so graph 
braid groups are linear, bi-orderable, and residually finite.  For more 
information on what is known about graph braid groups, see for instance 
\cite{Sabalka4}.

For any class of groups $\mathcal{G}$, it is interesting to ask whether 
or not one can algorithmically decide if two members $G$ and $G'$ in 
$\mathcal{G}$ are isomorphic as groups.  We call this question the 
\emph{isomorphism problem} for $\mathcal{G}$.

Isomorphism problems are one of the fundamental topics of study for 
combinatorial and geometric group theory.  Isomorphism problems are the 
hardest of the three classes of algorithmic problems in group theory 
formulated by Max Dehn \cite{Dehn}.  It is known that the isomorphism 
problem for finitely presented groups is undecidable in general 
\cite{Adyan,Rabin}.  However, there are solutions to the isomorphism 
problem for certain classes of groups.  A short list of such classes of 
groups includes: polycyclic-by-finite groups \cite{Segal}, finitely 
generated nilpotent groups \cite{GrunewaldSegal}, torsion-free word 
hyperbolic groups which do not split over the trivial or infinite cyclic 
group \cite{Sela}, and finitely generated fully residually free groups 
\cite{BKM}.

The purpose of this paper is to implement an algorithm to solve the 
isomorphism problem for tree braid groups in some cases.  We prove:

\begin{theoremA}[cf. Theorem \ref{thm:isoproblem})(The Isomorphism Problem] 

Let $G$ and $G'$ be two groups be given by finite presentations, and 
assume that $G \cong B_nT$ and $G' \cong B_nT'$ for some positive 
integer $n$ and finite trees $T$ and $T'$.  If either:

  \begin{itemize} 
  \item $n = 4$ or $5$ or 
  \item at least one of $G$ or $G'$ is free, 
  \end{itemize} 
then there exists an algorithm which decides whether $G$ and $G'$ are 
isomorphic.  The trees $T$ and $T'$ need not be specified.  If one of 
$T$ and $T'$ has at least 3 essential vertices, then $n$ need not be 
specified.

\end{theoremA}

To prove Theorem A, the main ingredient is a bijection between trees and 
tree braid groups.  This bijection allows us to algorithmically 
reconstruct the defining tree $T$ from the tree braid group $B_nT$.
This reduces the isomorphism problem for tree braid groups to the 
isomorphism problem for trees, which has a brute force algorithmic 
solution.  

The bijection between trees and tree braid groups is the strongest and 
most difficult result of this paper, and is interesting in its own 
right:

\begin{theoremB}[cf. Theorem \ref{thm:45rigidity})(Rigidity for $4$ and 
$5$ Strand Tree Braid Groups] 

Let $T$ and $T'$ be two finite trees, and let $n = 4$ or $5$.  
The tree braid groups $B_nT$ and $B_nT'$ are isomorphic as groups if and 
only if the trees $T$ and $T'$ are homeomorphic as trees.

\end{theoremB}

The idea of the proof of Theorem B is to use cohomology to reconstruct 
the tree $T$ from the tree braid group $B_nT$.  This reconstruction 
involves careful combinatorial bookkeeping in a finite simplicial 
complex $\Delta$ associated to $B_nT$.  The complex $\Delta$ is the 
defining complex for an exterior face algebra structure on the 
cohomology ring of $B_nT$ (see Section \ref{sec:extfacealgs} for 
definitions).  The main technical used to prove Theorem B is that 
$\Delta$ exists and is unique:

\begin{theoremC}[cf. Theorem \ref{thm:cohom})(Exterior Face Algebra 
Structure on Cohomology]

Let $T$ be a finite tree.  For $n = 4$ or $5$, 
$H^*(B_nT;\mathbb{Z}/2\mathbb{Z})$ is an exterior face algebra.  The 
simplicial complex $\Delta$ defining the exterior face algebra structure 
is unique and at most $1$-dimensional.

\end{theoremC}

Theorem C and other related results lead to an almost complete 
characterization of when the cohomology of a tree braid group is 
an exterior face algebra (see Conjecture \ref{conj:extcohom}).

Throughout the proofs of Theorems B and C, we rely extensively on 
results due to discrete Morse theory, many of which were presented in 
previous papers: \cite{FarleySabalka1,Farley,FarleySabalka2a,Farley2}.  
We also develop tools to discuss many properties of cohomology rings of 
tree braid groups, using the language of differential forms.

The remainder of this paper is organized as follows.  In Section 
\ref{sec:terminology}, we introduce terminology about trees that we will 
need to prove the main results, and define exterior face algebras.  In 
Section \ref{sec:prevresults}, we survey results from other sources 
needed in our proofs.  We pay particular attention to the structure of 
the cohomology ring for tree braid groups, as detailed in 
\cite{FarleySabalka2a}.  In Section \ref{sec:cobound}, we develop the 
notation and terminology to talk about the cohomology rings of graph 
braid groups in the language of differential forms.  In Sections 
\ref{sec:extcohom} and \ref{sec:rigidity}, we prove a number of results 
which lead to Theorems \ref{thm:cohom} and \ref{thm:45rigidity}, 
respectively.  The solution to the isomorphism problem for tree braid 
groups on $4$ or $5$ strands then follows in Section 
\ref{sec:isoproblem}.  Finally, in Section \ref{sec:RAAGs}, we end with 
theorems for right-angled Artin groups in a similar vein to Theorems A, 
B, and C.

The author would like to thank the following people for their help in 
writing this paper: his postdoctoral advisor, Misha Kapovich; his 
doctoral advisor, Ilya Kapovich; Daniel Farley, for numerous helpful 
discussions on this matter; and Go Fujita for initially proposing this 
problem.

\section{Terminology}\label{sec:terminology}

\subsection{Trees}

We begin with terminology for trees.  Throughout this subsection, see 
Figure \ref{fig:terminology} for explicit examples of some of the many 
of the concepts we define.

  \begin{figure}[!hb]
  \begin{center}
  \input{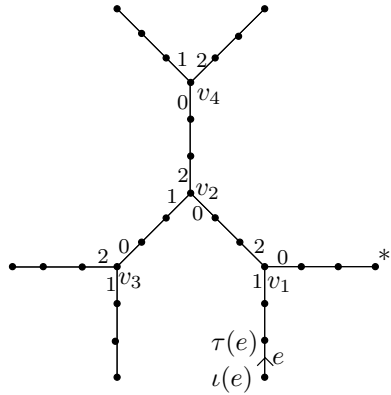}
  \end{center}

  \caption{This figure shows an embedding of a tree in the plane.  The 
  tree is $T_{min}$, the minimal nonlinear tree - i.e. the (unique up to 
  homeomorphism) nonlinear tree with the fewest number of essential 
  vertices and the smallest degrees of essential vertices.  With the 
  choice of basepoint $\ast$, we have a Morse $T_{min}$-embedding.  The 
  essential vertices of $T_{min}$ are labelled $v_1$, $v_2$, $v_3$, and 
  $v_4$.  The vertices $v_1$, $v_3$, and $v_4$ are extremal, and each 
  are adjacent only to $v_2$.  Directions from each essential vertex are 
  labelled.  For instance, $v_2$ is in direction $2$ from $v_1$, and 
  $v_4$ is in direction $0$ from $v_3$.  The edge $e$ has its endpoints 
  $\iota(e)$ and $\tau(e)$ labelled.  Here, $T_{min}$ is sufficiently 
  subdivided for $n = 4$.}

  \label{fig:terminology}

  \end{figure}

Let $T$ be a tree.  We call a vertex $v$ of $T$ \emph{essential} if $v$ 
has degree 3 or more.  Two essential vertices are considered 
\emph{adjacent} if they are connected by a path which crosses no other 
essential vertices.  An essential vertex $v$ is \emph{extremal} it is 
adjacent to exactly one other essential vertex.  A tree is \emph{linear} 
if there exists an embedded line segment which contains every essential 
vertex; equivalently, if it has at most two extremal vertices.

\begin{definition}[Morse $T$-embedding]\label{def:Morseembedding}

Let $T$ be a tree.  Embed $T$ into the plane. Let $\ast$ denote a degree 
$1$ vertex of $T$, called the \emph{basepoint} of $T$.  The information 
of the tree $T$, the embedding of $T$ into the plane, and the choice of 
$\ast$ is called a \emph{Morse $T$-embedding}.

\end{definition}

Let $T$ be a tree with a Morse $T$-embedding.  Let $e$ be an edge in 
$T$.  We call the endpoint of $e$ closer to $\ast$ in $T$ the 
\emph{terminal} vertex of $e$, denoted $\tau(e)$.  This convention gives 
us an orientation on edges in $T$. Similarly, the endpoint of $e$ 
further from $\ast$ in $T$ is the \emph{initial} vertex of $e$, denoted 
$\iota(e)$.  This convention gives us an orientation on edges in $T$.

Let $S$ be a collection of vertices and edges of $T$.  We denote by $T - 
S$ the largest closed subgraph of $T$ which does not contain an element 
of $S$.  In other words, $T - S$ is formed by deleting every edge in 
$S$, as well as any vertex in $S$ and any edge with an endpoint in $S$.

Let $v$ be a vertex in $T$.  Given a Morse $T$-embedding, the edges 
adjacent to $v$ may be numbered $0, \dots, deg(v)-1$ in the order 
encountered by a clockwise traversal of $T$ from $\ast$, and where 
$deg(v)$ denotes the degree of $v$.  If $v = \ast$, the unique edge 
adjacent to $\ast$ is numbered $1$.  A \emph{direction} from $v$ is a 
choice of one of these edge labels.  A vertex is said to \emph{lie in} 
direction $d$ from $v$ if $d$ labels the first edge of the unique simple 
path (that is, a path with no self intersections) from $v$ to the 
vertex.  By convention, $v$ \emph{lies in} direction $0$ from itself.  
An edge \emph{lies in} direction $d$ from $v$ if, for one of the 
endpoints of $e$, $d$ labels the first edge of the unique simple path 
from $v$ to that endpoint.

An extremal vertex has the property that every other essential vertex 
lies in a single direction from it.

Let $\Delta'$ denote the union of those open cells of $\prod^n \Gamma$ 
whose closures intersect the fat diagonal $\Delta = \{(x_1, \dots, 
x_n)|x_i = x_j \hbox{ for some } i \neq j\}$.  Let $\ud{n}{\Gamma}$ 
denote the quotient of the space $\prod^n \Gamma - \Delta'$ by the 
action of the symmetric group given by permuting coordinates. Note that 
$\ud{n}{\Gamma}$ inherits a CW complex structure from the Cartesian 
product: an open cell in $\udng$ has the form $\{y_1, \dots, y_n\}$ such 
that each $y_i$ is either a vertex or an edge and the closures of the 
$y_i$ are mutually disjoint.  The set notation is used to indicate that 
order does not matter.  We call $\udng$ the \emph{unlabelled discretized 
configuration space} of $\Gamma$.  Under most circumstances, the 
$\uc{n}{\Gamma}$ is homotopy equivalent to $\udng$.  Specifically:

\begin{theorem}[Sufficient Subdivision] \emph{(}\cite{Hu} for $n=2$; 
\cite{Abrams1} for $n > 2$\emph{)} \label{thm:Abrams}

For any $n>1$ and any graph $\Gamma$ with at least $n$ vertices, 
$\uc{n}{\Gamma}$ strong deformation retracts onto $\udng$ if

  \begin{enumerate}

  \item each path between distinct vertices of degree not equal to $2$
  passes through at least $n-1$ edges; and

  \item each path from a vertex to itself which is not null-homotopic in
  $\Gamma$ passes through at least $n+1$ edges.

  \end{enumerate}

\end{theorem}

A graph $\Gamma$ satisfying the conditions of this theorem for a given 
$n$ is called \emph{sufficiently subdivided} for $n$.  It is clear 
that, for any $n$, every graph is homeomorphic to a sufficiently 
subdivided graph for $n$.  

Throughout the rest of this paper, we assume that we are 
dealing with graphs which are sufficiently subdivided for at least $n+2$ 
strands.

We mention here that Abrams \cite{Abrams1} proved that the universal 
cover of the space $\udng$ is a CAT(0) cubical complex for any graph 
$\Gamma$. This implies that graph braid groups have solvable word and 
conjugacy problems \cite{BridsonHaefliger}.

\subsection{Exterior Face Algebras}\label{sec:extfacealgs}

We now consider exterior face algebras.  Let $K$ be a finite simplicial 
complex with vertices $\{v_1, \dots, v_k\}$.  For any field $F$ with 
identity, the \emph{exterior face algebra} $\Lambda_{F}(K)$ of $K$ over 
$F$ is the quotient of the exterior algebra $\Lambda_{F}[v_1, \dots, 
v_{k+1}]$ by the ideal generated by products of vertices of non-faces of 
$K$.  In other words, $\Lambda_F(K)$ is the $F$-vector space having the 
products $v_{i_1}v_{i_2}\dots v_{i_j}$ ($0 \leq j \leq k$, $i_1 < i_2 < 
\dots < i_j$) as a basis, and subject to the following multiplicative 
relations:

  \begin{itemize}
  \item $v_i v_j = -v_jv_i$ for $0 \leq i,j \leq n$,
  \item $v_i^2 = 0$ for $0 \leq i \leq n$, and
  \item $v_{i_1} \dots v_{i_k} = 0$ if $i_1 < \dots < i_k$ and 
  $\{v_{i_1}, \dots, v_{i_k}\}$ is not a face of $K$.
  \end{itemize}
An exterior algebra corresponds to the exterior face algebra of a 
standard simplex.  Note the shift of indices:  an $i$-cell in $K$ 
corresponds to an element of degree $i+1$ in $\Lambda(K)$.

For our purposes, $F$ will always be the field $\mathbb{Z}/2\mathbb{Z}$, 
so we suppress the subscript $F$ from now on.  For this field, 
$\Lambda(K)$ is a quotient of a polynomial ring:
  $$ \Lambda(K) = (\mathbb{Z}/2\mathbb{Z}) \left[ v_1, \dots, v_n 
  \right]/ I(K),$$
where $I(K)$ is the ideal of $(\mathbb{Z}/2\mathbb{Z}) \left[ v_1, \dots,
v_n \right]$ generated by the set
  \begin{eqnarray*}
  &\{v_1^2,&\hspace{-2mm} \dots, v_n^2\} \cup\\
  &&\{v_{i_1} \dots v_{i_k} | i_1 < \dots < i_k; \{ 
  v_{i_1}, \dots, v_{i_k}\} \hbox{ is not a face of $K$}\}.
  \end{eqnarray*}

In the $\mathbb{Z}/2\mathbb{Z}$ case, Gubeladze \cite{Gubeladze} has 
shown that exterior face algebras are in bijective correspondence with 
their defining simplicial complexes:

\begin{theorem}[Gubeladze's Theorem on Exterior Face Algebra Rigidity]
\emph{(}\cite{Gubeladze}\emph{)}\label{thm:Gubeladze} 

Let $K$ and $K'$ be two finite simplicial complexes.  Then $\Lambda(K)$ 
and $\Lambda(K')$ are isomorphic as algebras if and only if $K$ and $K'$ 
are isomorphic as simplicial complexes.

\end{theorem}

This rigidity allows us to speak of `the' simplicial complex defining an 
exterior face algebra.

As an example application of Gubeladze's Theorem, see Theorem 
\ref{thm:RAAGcohomrigidity} in Section \ref{sec:RAAGs}, on right-angled 
Artin groups.

\section{Previous Results}\label{sec:prevresults}

Now that we have introduced notation and terminology concerning trees 
and exterior face algebras, we turn our attention to recalling previous 
results - in particular important results from \cite{FarleySabalka1}, 
\cite{Farley}, and \cite{FarleySabalka2a}.  We begin with the first two 
references, on fundamental group and homology.

\subsection{Morse Theory, the Fundamental Group, and Homology}

For a tree $T$, consider a Morse $T$-embedding (Definition 
\ref{def:Morseembedding}). By \cite{FarleySabalka1}, this embedding 
induces a `Morse matching' on $\ud{n}{T}$.  For the sake of brevity, we 
do not define or detail the Morse matching here, but instead define the 
cells of $\ud{n}{T}$ which are \emph{critical} with respect to this 
matching.  More detailed expositions on the Morse matching and the 
classification of cells of $\ud{n}{T}$ can be found in \cite{Sabalka4}, 
where a Morse matching is referred to as a discrete gradient vector 
field (the exposition in \cite{Sabalka4} is based on the original work 
in \cite{FarleySabalka1}).

\begin{definition}[Blocked, Respectful, and Critical]

Begin with a Morse $T$-embedding.  Let $c$ be an open cell in 
$\ud{n}{T}$.  Consider a vertex $v \in c$.  If $v = \ast$, then $v$ is 
\emph{blocked} by $\ast$ in $c$.  If $v \neq \ast$, let $e$ be the 
unique edge in $T$ with $\iota(e) = v$.  If $e \cap x \neq \emptyset$ 
for some edge or vertex $x \in c$, $x \neq v$, then again $v$ is 
\emph{blocked} by $x$ in $c$.  If $v$ is not blocked in $c$, $v$ is 
\emph{unblocked}.

Now consider an edge $e \in c$.  The edge $e$ is \emph{disrespectful} in 
$c$ if: there exists a vertex $v \in c$ blocked by $e$ and which, in a 
clockwise traversal of the tree $T$ from $\ast$, is traversed after one 
endpoint of $e$ but before the other.  Otherwise, the edge $e$ is 
\emph{respectful} in $c$.

A cell $c$ is \emph{critical} if there are neither unblocked vertices 
nor respectful edges in $c$.  See Figure \ref{fig:crittobraid} for an 
example. 

\end{definition}

Critical cells were seen to be very useful in describing graph braid 
groups, as evidenced in \cite{FarleySabalka1,FarleySabalka2a,Farley}, 
etc.  In particular, as we will see, critical cells lend themselves well 
to describing generating sets and even relations for fundamental group, 
homology, and cohomology.

A $k$-cell in $\ud{n}{T}$ corresponds exactly to $n-k$ strands sitting 
on $T$ and $k$ strands simultaneously crossing $k$ disjoint edges. To 
see the connection between critical cells and braids, consider a 
critical $1$-cell $c = \{v_1, \dots, v_{n-1},e\}$ with unique edge $e 
\in c$.  Then $c$ corresponds to the following braid: start with $n$ 
strands sitting next to $\ast$ in $T$.  Move each strand out to a 
distinct vertex in $\{v_1, \dots, v_{n-1},\iota(e)\}$ in order so that 
the strand furthest away from $\ast$ goes to the highest-numbered 
vertex, the second-furthest goes to the second-highest, etc.  Next, move 
the strand at $\iota(e)$ to $\tau(e)$.  Finally, move each of the 
strands back to $\ast$, in order so that the strand on the 
lowest-numbered vertex in $\{v_1, \dots, v_{n-1},\tau(e)\}$ moves first 
and ends nearest $\ast$, the second-lowest ends second-nearest, etc.  
See Figure \ref{fig:crittobraid}.

  \begin{figure}[!h]
    \begin{center}
    \includegraphics{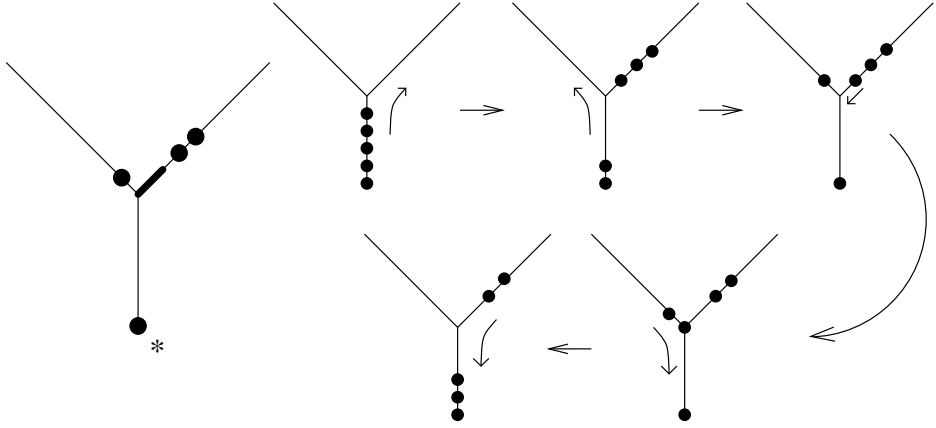}
    \end{center}

  \caption{On the left is a critical $1$-cell $c$.  The remaining 
  figures are a diagrammatic illustration of the braid that $c$ 
  represents.}
  
  \label{fig:crittobraid}
  
  \end{figure}

The power of discrete Morse theory and the classification of critical 
cells is evident in the following useful theorems:

\begin{theorem}[Morse Presentation for $B_nT$] 
\emph{(}\cite{FarleySabalka1}, corollary of Theorem 2.5\emph{)} 
\label{thm:Morse}

Let $T$ be a tree.  Fix a Morse $T$-embedding.  Then $B_nT$ has a 
presentation for which the generators may be identified with the 
set of critical $1$-cells, and the relations are determined by the 
set of critical $2$-cells.\qed

\end{theorem}

A presentation derived from a Morse $T$-embedding as in Theorem 
\ref{thm:Morse} is called a \emph{Morse} presentation, and its 
generators are Morse generators.

\begin{theorem}[Homology] 
\emph{(}\cite{Farley}\emph{)}\label{thm:homology}

Let $T$ be a tree. Fix a Morse $T$-embedding. Then $H_iB_nT$ is free 
abelian of rank equal to the cardinality of the set of critical 
$i$-cells of $T$. In particular, $H_iB_nT$ has a generating set which 
may be identified with the set of critical $i$-cells of $T$. \qed

\end{theorem}

Such a distinguished basis for homology is called a \emph{Morse basis}, 
composed of elements called \emph{Morse generators}.

We end this subsection with the following useful calculation about 
certain tree braid groups, to which we will repeatedly refer:

\begin{theorem}[Radial Rank] \emph{(}\cite{FarleySabalka1}, Corollary 
4.2\emph{)}\label{thm:radial}

If $\Gamma$ is a \emph{radial} tree - i.e. a tree with exactly one 
essential vertex $v$ - then the tree braid group $B_n\Gamma$ is free of 
rank
  $$Y_n(x)
    := \sum_{i=2}^{x-1} \left[\binom{n+x-2}{n-1} - 
    \binom{n+x-i-1}{n-1}\right],$$
where $x = deg(v)$.\qed

\end{theorem}

Note that the function $Y_n(x)$ is a monotonically increasing function 
of $x$ for $x \geq 3$ and of $n$ for $n \geq 2$.  For instance, $Y_2(x) 
= x^2/2-3x/2+1$, and $Y_3(x) = x^3/3-x^2/2-5x/6+1$.

\subsection{Cohomology}\label{sec:cohomology}

To describe the structure of cohomology for tree braid groups, we first 
need to introduce some definitions.  We begin by defining a partial 
ordering and an equivalence relation on cells of $\ud{n}{T}$.  Using 
this partial order and equivalence relation, we are able to state two of 
the main results from \cite{FarleySabalka2a}, which are computational 
and structural statements about cohomology for tree braid groups, and 
will be referred to often.  We end the section with some notation, 
suggested by one of these results, for what will be called reduced 
$1$-cells.

Let $E(c)$ denote the set of edges of the $i$-cell $c$.  We abuse the 
notation by also letting $E(c)$ denote the subset $\bigcup_{e \in E(c)} 
e$ of $T$. For two cells $c$ and $c'$, write $c \sim c'$ if

  \begin{enumerate} 

  \item $E(c) = E(c')$, and 

  \item for any connected component $C$ of $T - E(c)$,
    $$ \left|C \cap (c - E(c))\right| =
    \left|C \cap (c' - E(c'))\right|.$$

  \end{enumerate}
It is straightforward to see that $\sim$ is an equivalence relation on 
the set of open cells in $\ud{n}{T}$. Let $[c]$ denote the equivalence 
class of a cell $c$.  

If every vertex in $c$ is blocked, we say $c$ is \emph{reduced}.  If $c$ 
is reduced but there exists an edge of $c$ such that $\tau(e)$ is not 
essential, $c$ is \emph{extraneous}.  If $c$ is reduced, $c$ is not a 
$0$-cell, and for every edge $e$ of $c$ no vertex of $c$ is blocked by 
$\tau(e)$, then again $c$ is \emph{extraneous}.

We define a partial order $\leq$ on the equivalence classes based on the 
face relation $\leq$ on cells, writing $[c_0] \leq [c_1]$ if there exist 
representatives $\hat{c}_0 \in [c_0]$, $\hat{c}_1 \in [c_1]$ such that 
$\hat{c}_0 \leq \hat{c}_1$ - i.e. $\hat{c}_0$ is a face of $\hat{c}_1$.  
We record some properties of $\sim$ and $\leq$ here:

\begin{theorem}[Properties of $\leq$ and $\sim$] \emph{(}c.f. 
\cite{FarleySabalka2a}\emph{)} \label{thm:leq}

  \begin{enumerate}

  \item The relation $\leq$ is indeed a partial order.

  \item Let $c_1, \dots, c_k$ be $1$-cells from distinct equivalence 
  classes.  If the set $\{[c_1], \dots, [c_k]\}$ has an upper bound $[s]$ with 
  respect to $\leq$, then the collection has a least upper bound.  
  Furthermore, if $e_1, \dots, e_k$ are the edges of $T$ satisfying $e_i 
  \in c_i$, then the edges $e_1, \dots, e_k$ are the edges of $s$ and 
  are pairwise disjoint.

  \item For any $k$-cell $s$, there exists a unique collection 
  $\{[c_1],\dots,[c_k]\}$ of equivalence classes of $1$-cells such that 
  $[s]$ is the least upper bound of $\{[c_1],\dots,[c_k]\}$ with respect 
  to $\leq$.

  \item If $c$ is a critical cell in $\ud{n}{\Gamma}$ and $[c'] \leq 
  [c]$, then $c' \sim \hat{c}$ for some critical cell $\hat{c}$.

  \item There exists exactly one reduced cell in each $\sim$-equivalence 
  class.  Every critical cell is reduced, but not every reduced cell is 
  critical.  In particular, a critical cell is the unique critical cell 
  in its equivalence class.

  \end{enumerate} 

\end{theorem}

\begin{proof}

Most of this theorem is from \cite{FarleySabalka2a}, Lemma 4.1 and the 
preceding discussion.  The only part not proven in 
\cite{FarleySabalka2a} is Part (5).

For Part (5), that every equivalence class contains a reduced cell is 
clear.  That every critical cell is reduced follows from the 
definitions.  That there exist reduced cells which are not critical is 
clear.  That a critical cell is the unique critical cell in its 
equivalence class follows from the first two statements of Part (5), and 
was also proven in \cite{FarleySabalka2a}.

It remains to prove that two reduced cells $c_1$ and $c_2$ in the same 
equivalence class must be equal.  By \cite{FarleySabalka2a}, Lemma 
4.1(1), $E(c_1) = E(c_2)$, so let $E := E(c_1) = E(c_2)$.  Consider a 
connected component $C$ of $T-E$.  Let $\ast_C$ be the smallest vertex 
of $C$.  The idea of the proof is that $T$ is sufficiently subdivided so 
that any collection of at most $n-1$ blocked vertices inside of $C$ must 
be `stacked up' at $\ast_C$, and is in particular uniquely determined.  
The uniqueness will force $c_1$ and $c_2$ to coincide, for each such 
connected component.  We formalize this idea.

Since $\ast_C$ is the smallest vertex of $C$, either $\ast_C = \ast$ or 
$e(\ast_C)$ intersects an edge in $E$.  Define a vertex $v_C$ of $T$ as 
follows.  If $\ast_C = \ast$, let $v_C := \ast$.  If $\ast_C \neq \ast$, 
then since $\ast_C$ is the smallest vertex of $C$, there exists an edge 
$e \in E$ with $e(\ast_C) \cap e \neq \emptyset$.  In this case, define 
$\ast_C$ to be $\tau(e)$.  In the former case, $v_C$ has degree $1$ in 
$T$, and in the latter case, $v_C$ has degree at least $3$ in $T$.  
Since $T$ is sufficiently subdivided for $n+2$ strands, $v_C$ is at 
least $n+1$ edges away from any other vertex of $T$ which does not have 
degree $2$ in $T$.  Since $\ast_C$ is at most $2$ edges away from $v_C$, 
$\ast_C$ is at least $n-1$ edges away from any other vertex of $T$ which 
does not have degree $2$ in $T$.  Since each edge of $E$ has an 
essential terminal endpoint, $\ast_C$ is at least $n-2$ edges away from 
any other vertex of $C$ which does not have degree $2$ in $C$.  Thus, 
since $C$ is a (connected subset of a) tree, for any $i = 0, \dots, 
n-1$, there is a unique vertex $v_{C,i}$ in $C$ such that the unique 
path from $v_{C,i}$ to $\ast_C$ contains exactly $k$ edges.

By the definition of $\sim$, $c_1$ and $c_2$ have the same number $k$ of 
vertices in $C$.  We claim that both $c_1$ and $c_2$ contain the $k$ 
vertices $v_{C,1}, \dots, v_{C,k}$ of $C$.  For, assume otherwise.  
Without loss of generality, assume $c_1$ does not contain the vertex 
$v_{C,i}$.  Let $v$ be the smallest vertex of $c_1 \cap C$ greater than 
$v_{C,i}$.  Then $\tau(e(v))$ is in $C$, since $\ast_C$ is in the same 
direction from $v_{C,i}$ in $T$ as $\ast$ is.  But $\tau(e(v))$ does not 
intersect $c_1$, since $C$ contains only vertices of $c_1$ and $v$ was 
chosen to be as small as possible.  Thus, $v$ is unblocked in $c_1$.  
This contradicts the hypothesis that $c_1$ is reduced.  Thus, $c_1$ and 
$c_2$ have the same edge set $E$ and the same vertices in each connected 
component of $T-E$, $c_1 = c_2$.

\end{proof}

Let $C_*(\ud{n}{T})$ denote the cellular chain complex of chains of 
cells in $\ud{n}{T}$ with coefficients in $\Z/2\Z$.  Let 
$\phi_{[c]}:C_*(\ud{n}{T}) \to \mathbb{Z}/2\mathbb{Z}$ denote the 
characteristic function of the $\sim$-equivalence class of $c$: 
$\phi_{[c]}(c') = 1$ if and only if $c' \sim c$.  For $c$ a critical 
cell, let $c^*: H_*(B_nT)\to \mathbb{Z}/2\mathbb{Z}$ denote the dual of 
$c$ viewed as a basis element of cohomology, by the Universal 
Coefficient Theorem.  Note the boundary maps for chains are all $0$ 
\cite{Farley}, so $H^*(B_nT) \cong Hom(H_*(B_n,T),\Z/2\Z)$. Then, using 
the ordering $\leq$ on equivalence classes of $1$-cells, we may state 
the following theorem:

\begin{theorem}[Cohomology] \label{thm:cohomology} 
\emph{(}\cite{FarleySabalka2a}, Proposition 4.5 and preceding 
discussion\emph{)}

Let $T$ be a tree.  Fix a Morse $T$-embedding (see Definition 
\ref{def:Morseembedding}).  Then under the induced Morse matching:

  \begin{enumerate}

  \item If $c$ is a critical cell in $\ud{n}{T}$, then $c^* = 
  [\phi_{[c]}]$. A distinguished basis for $i$-dimensional cohomology is
    $$\{c^{\ast} | c \hbox{ a critical $i$-cell}\}.$$

  \item Let $s$ be a critical $i$-cell in $\ud{n}{T}$.  Let $[s]$ be the 
  least upper bound of $\{[c_1],\dots,[c_i]\}$, where the $[c_1], \dots, 
  [c_i]$ are distinct equivalence classes of $1$-cells.  Then, without 
  loss of generality, $c_1, \dots, c_i$ are critical, and
    $$c^{\ast}_1 \cup \dots \cup c^{\ast}_i = s^{\ast}.$$
  In particular, $H^*(B_nT)$ is generated as a ring by duals of critical 
  $1$-cells.

  \item If $[c_1], \dots, [c_i]$ are distinct equivalence classes of
  critical $1$-cells having the least upper bound $[s]$, then
    $$[\phi_{[c_1]}] \cup \dots \cup [\phi_{[c_i]}] = 
    \left[\phi_{[s]}\right].$$
  If $[c_1], \dots, [c_i]$ are not all pairwise distinct or have no 
  upper bound, then $[\phi_{[c_1]}] \cup \dots \cup [\phi_{[c_i]}] = 0$.

  \end{enumerate}

\end{theorem}

Note a critical cell has dimension at most $(\lfloor \frac{n}{2} 
\rfloor)$ \cite{FarleySabalka1}.  It follows from the first part of 
Theorem \ref{thm:cohomology} that cohomology is trivial in all 
dimensions greater than $\lfloor \frac{n}{2} \rfloor$.

As with homology, such a distinguished basis for cohomology is called a 
\emph{Morse basis}, composed of elements called \emph{Morse generators}.  
The power of this theorem comes from the characterization not only of 
a Morse basis, but of the cup product structure.

\subsection{Reduced Cells}\label{sec:not&term}

Throughout the remainder of this paper, we will be extensively using 
reduced cells and computations involving reduced cells.  For future 
reference, we establish here notation for many reduced $1$-cells, and 
record some properties of reduced cells.

For a vector $\vec{x} \in \mathbb{N}^{\infty}$ of nonnegative integers, 
indexed from 0 to $\infty$, we let $x_i$ denote the $i^{th}$ entry of 
$\vec{x}$. Given a tree $T$, a Morse $T$-embedding, and an essential 
vertex $v$ of $T$, the vector $\vec{x}$ is a \emph{$v$-vector} if $x_i = 
0$ for all $i \geq deg(v)$.  When writing $v$-vectors, we will omit the 
entries $0$ for indices $i \geq deg(v)$.  The \emph{length} of 
$\vec{x}$, denoted $|\vec{x}|$, is the sum $\sum_{i=0}^{\infty}x_i$ (all 
of our lengths will be finite).

Let $c$ be a reduced $1$-cell of $\ud{n}{T}$ with unique edge $e$ such 
that $e$ has an essential terminal endpoint $a = \tau(e)$.  By Theorem 
\ref{thm:leq}, $c$ is unique in its equivalence class $[c]$.  Thus, $c$ 
is uniquely determined by the number of vertices in each connected 
component of $T-e$. This means that $c$ is uniquely determined by 
specifying the endpoint $a = \tau(e)$, the direction $d$ from $a$ along 
$e$, and the number of vertices in $c$ in each direction from $a$. Let 
$\vec{x} \in \mathbb{N}^{\infty}$ be the $a$-vector such that, for $i 
\in \{0, \dots, \deg(a)-1\}$ $x_i$ is the number of vertices in $c$ in 
direction $i$ from $a$.  Note $n = |\vec{x}|$.

\begin{definition}[Reduced $1$-Cell Notation]\label{def:reducednotation} 

To encode the reduced $1$-cell $c$ with edge $e$ such that $\tau(e)$ is 
the essential vertex $a$, we write
  $$(a,d,\vec{x}) \hbox{~~or equivalently~~} (a,e,\vec{x}),$$
where $d$ and $\vec{x}$ are as above.  We say that $\vec{x}$ is 
\emph{the $a$-vector for $c$}, and that $c$ \emph{lies over} 
the vertex $a$.

\end{definition}

This notation is a slightly modified version of the notation developed 
in the paper \cite{FarleySabalka1}, and appears in 
\cite{FarleySabalka2a}.

Using this notation, a fairly immediate observation is:

\begin{lemma}\label{lem:x_0}

Let $c = (a,e,\vec{x})$ be a non-extraneous reduced $1$-cell.  Then $x_0 
\leq n - 2$.

\end{lemma}

\begin{proof}

Let $d$ be the direction from $a$ to $\iota(e)$.  Since $e$ is not 
extraneous, there exists some vertex $v$ of $c$ such that the direction 
$d'$ from $a$ to $v$ is not $0$ or $d$.  Thus, $x_d, x_{d'} \geq 1$.  
But $|\vec{x}| = n$, so $x_0 \leq n - 2$.

\end{proof}

As Theorem \ref{thm:cohomology} suggests, upper bounds of equivalence 
classes of reduced cells play an important role.  We record many 
properties of upper bounds in the following lemma.  Note that, in the 
proof of the last statement, we explicitly construct the reduced 
representative of upper bound of the equivalence classes of two 
non-extraneous reduced $1$-cells.

\begin{lemma}[Upper Bound Lemma]\label{lem:upperbound}

Let $c_1 = (a,d,\vec{x})$ and $c_2 = (b,f,\vec{y})$ be non-extraneous 
reduced $1$-cells where $a \leq b$ in the order on vertices.  Let 
$\alpha$ be the direction from $a$ to $b$.  Then:

  \begin{enumerate}

  \item the direction from $b$ to $a$ is $0$.  

  \item \label{cond:upperbound2} $\{[c_1], [c_2]\}$ has an upper bound 
  if and only if
    \begin{enumerate}

    \item $a \neq b$, and

    \item $x_\alpha + y_0 \geq n + \epsilon$, where $\epsilon = 
    \epsilon(c_1,\alpha)$ is $1$ if $d = \alpha$ and $0$ otherwise.

    \end{enumerate}

  \item If $\{[c_1],[c_2]\}$ has an upper bound, then $x_\alpha \geq 
  2 + \epsilon$.

  \item If $\{[c_1],[c_2]\}$ has an upper bound $[s]$, let $s$ be the 
  reduced representative of $[s]$.  We may explicitly describe $s$.  The 
  edge $f$ of $c_2$ is disrespectful in the reduced representative $s$ 
  of $[s]$ if and only if $f$ is disrespectful in $c_2$.  The edge $e$ 
  of $c_1$ is disrespectful in $s$ if and only if either

    \begin{enumerate}

    \item $0 < \alpha < d$ and $x_\alpha + y_0 > n$, or

    \item there exists some $i\neq \alpha$ such that $0 < i < d$ and 
    $x_i > 0$.
    
    \end{enumerate}
  
  \end{enumerate}

\end{lemma}

The $\epsilon(c_1,\alpha)$ in the lemma is called the \emph{upper bound 
constant} in direction $\alpha$ for $c_1$.

\begin{proof} ~\\

  \begin{enumerate}

  \item If $a = b$, then the direction from $b = a$ to itself is by 
  definition $0$.  If $a < b$, that the direction from $b$ to $a$ is $0$ 
  follows directly from the definition of $a < b$.

  \item If $\{[c_1], [c_2]\}$ has an upper bound, let $[s]$ be a least 
  upper bound.  By Theorem \ref{thm:leq}, $[s]$ exists, is unique, and 
  has a unique element - say, $s$ - in which all vertices are blocked.  
  Since $\tau(e)$ and $\tau(f)$ are essential, $s$ is reduced.  Also by 
  Theorem \ref{thm:leq}, the edges $e$ and $f$ are disjoint - in 
  particular, $a \neq b$.  Finally, note that $x_\alpha$ is the number 
  of elements of not just $c_1$ but also $s$ in direction $\alpha$ from 
  $a$, and similarly $y_\beta$ is the number of elements of $s$ in 
  direction $\beta$ from $b$.  Since $\beta = 0$, $n - y_\beta$ is the 
  number of elements of $s$ in all directions not equal to $\beta$ from 
  $b$.  Since $T$ is a tree, $n - y_\beta$ must be at most the number of 
  elements in $s$ in direction $\alpha$ from $a$ - that is, $x_\alpha$.  
  But if $d = \alpha$, then the edge of $c_1$ which is also an edge of 
  $s$ is in direction $\alpha$ from $a$ but direction $\beta$ from $b$.  
  Thus if $d = \alpha$, $n - y_\beta$ must be at most $x_\alpha - 1$.  
  The desired inequality then follows.

  Now assume that $c_1$ and $c_2$ satisfy the desired properties.  Let 
  $c_1'$ be the $1$-cell with $(n - x_\alpha + \epsilon)$ strands on $T$ 
  consisting of the edge of $c_1$ and all vertices of $c_1$ not in 
  direction $\alpha$ from $a$.  Similarly, let $c_2'$ be the $1$-cell 
  with $(n-y_\beta)$ strands on $T$ consisting of the edge of $c_2$ and 
  all vertices of $c_2$ not in direction $\alpha$ from $a$.  Since $a 
  \neq b$, we may define the $2$-cell $s'$ to be $c_1' \cup c_2'$ with 
  $k$ strands on $T$, where $k = (n - x_\alpha + \epsilon) + (n - 
  y_\beta)$.  Since $x_\alpha + y_\beta \geq n + \epsilon$, $k \leq n$.  
  Since $T$ is sufficiently subdivided for $n+2$, there are at least $n$ 
  vertices of $T$ between $a$ and $b$ not contained in $s'$.  Let $s$ be 
  the $2$-cell with $n$ strands on $T$ which is $s'$ plus exactly $n-k$ 
  vertices between $a$ and $b$.  Then by the definitions of $\sim$ and 
  $\leq$, $[s]$ is an upper bound for $\{[c_1],[c_2]\}$.

  \item If $\{[c_1],[c_2]\}$ has an upper bound, then $x_\alpha \geq n + 
  \epsilon - y_0 \geq 2 + \epsilon$, as $y_0 \leq n - 2$ by Lemma 
  \ref{lem:x_0}.

  \item If $\{[c_1], [c_2]\}$ has an upper bound $[s]$, then by Theorem 
  \ref{thm:leq}, the reduced representative $s$ of $[s]$ is unique.  We 
  explicitly construct $s$.

  Let $s_1$ be the reduced cell on $y_0$ strands which is $c_1$ but with 
  only $x_\alpha - (n - y_0)$ strands in direction $\alpha$ from $a$.  
  As $x_\alpha \geq n + \epsilon - y_0$, even if $e$ is in direction 
  $\alpha$ from $a$, $s_1$ still contains the edge $e$.  The cell $s_1$ 
  corresponds to deleting the $n-y_0$ largest strands with respect to 
  the order on vertices in the direction $\alpha$ from $a$. Let $s_2$ be 
  the reduced cell on $n - y_0$ strands which is $c_2$ but with the 
  vertices in direction $0$ from $b$ removed.  Finally, let $s = s_1 
  \cup s_2$.  Since each of $s_1$ and $s_2$ is reduced, so is $\alpha$.  
  Note $f$ is disrespectful in $c_2$ if and only if $f$ is disrespectful 
  in $s_2$, if and only if $f$ is disrespectful in $s$, by the 
  definition of disrespectful. Since $e$ is disrespectful in $c_1$, $e$ 
  will remain disrespectful in $s_1$ and therefore $s$ unless there was 
  only one direction $i$ such that $0 < i < d$ and $x_i > 0$ - namely, 
  $i = \alpha$ - and there are no vertices in direction $i = \alpha$ 
  from $a$ in $s_1$.  There are no vertices in direction $\alpha$ from 
  $a$ in $s_1$ if and only if $0 < \alpha < d$ and $x_\alpha + y_0 = 
  n$.  This finishes the proof of the lemma.

  \end{enumerate}

\end{proof}

\section{Cohomology in Terms of Differential Forms}\label{sec:cobound}

We will couch our further discussion of cohomology of tree braid groups 
in the terminology of de Rham cohomology and differential forms.  
Although we are dealing with CW complexes instead of manifolds, 
hopefully the similarity of the formulas, particularly in defining 
differentials, will justify our abuse of notation.

Recall that $C_*(\ud{n}{T})$ is the cellular chain complex on 
$\ud{n}{T}$.  We denote certain cochains on our space by the term 
\emph{form}, where a $k$-form will be a $k$-cochain.  The $0$-forms will 
be functions on cellular chains in $C_*(\ud{n}{T})$ which take values in 
$\Z/2\Z$.

For a vertex $v \in T$ and a direction $i$ from $v$, let $D_{v,i}$ be 
the function from cells of $\ud{n}{T}$ to $\Z$ which takes a cell $c$ 
and counts the number of vertices or edges in $c$ in direction $i$ from 
$v$.  Thus, for instance, for a reduced cell $(a,d,\vec{x})$, 
$D_{a,i}(a,d,\vec{x}) = x_i$.  Let $\overline{D}_{v,i}$ be similarly a 
function which takes a cell $c$ and counts the number of vertices or 
edges in $c$ in direction $i$ from $v$, \emph{but} where each edge of 
$v$ is considered in the direction of its terminal endpoint from $v$ 
instead of its initial endpoint.  Thus, for instance, 
$\overline{D}_{a,i}(a,d,\vec{x}) = x_i$ if $i \neq 0, d$, but 
$\overline{D}_{a,0}(a,d,\vec{x}) = x_0+1$ and 
$\overline{D}_{a,d}(a,d,\vec{x}) = x_d-1$.

  \begin{figure}
    \begin{center}
    \input{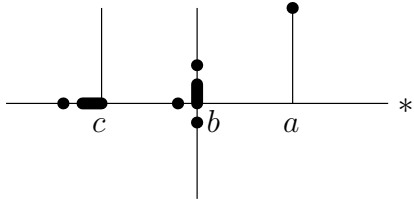}
    \end{center}

  \caption{This figure depicts the $2$-cell $s$ used in Example 
  \ref{ex:Daiexample}.}

  \label{fig:Daiexample}
  
  \end{figure}

\begin{example}\label{ex:Daiexample}

Consider the $2$-cell $s$ shown in Figure \ref{fig:Daiexample}. This 
figure depicts a $2$-cell $s$ on a tree with three essential vertices 
$a$, $b$, and $c$.  We have:

    \begin{eqnarray*}
    D_{a,i}(s) = \left\{
      \begin{array}{cl} 
      0 & \hbox{ if } i = 0\\
      6 & \hbox{ if } i = 1\\
      1 & \hbox{ if } i = 2
      \end{array}\right.,&\qquad
    D_{b,i}(s) = \left\{
      \begin{array}{cl} 
      1 & \hbox{ if } i = 0\\
      1 & \hbox{ if } i = 1\\
      3 & \hbox{ if } i = 2\\
      2 & \hbox{ if } i = 3
      \end{array}\right.,\\
    \overline{D}_{b,i}(s) = \left\{
      \begin{array}{cl} 
      2 & \hbox{ if } i = 0\\
      1 & \hbox{ if } i = 1\\
      3 & \hbox{ if } i = 2\\
      1 & \hbox{ if } i = 3
      \end{array}\right. .&
    \end{eqnarray*}

In terms of differential forms, the following is an incomplete list of 
forms which map $c$ to $1$: \\

\begin{center}
$f(a,\left[\begin{array}{c}0\\6\\1\end{array}\right])$, $\qquad$ 
$f(b,\left[\begin{array}{c}2\\1\\3\\1\end{array}\right])$, $\qquad$ 
$f(c,\left[\begin{array}{c}6\\1\\0\end{array}\right])$, $\qquad$ \\
$f(a,\left[\begin{array}{c}0\\6\\1\end{array}\right])d(b,2,\left[\begin{array}{c}2\\1\\3\\1\end{array}\right])$, $\qquad$ 
$f(c,\left[\begin{array}{c}5\\2\\0\end{array}\right])d(b,2,\left[\begin{array}{c}2\\1\\3\\1\end{array}\right])$, $\qquad$ \\
$f(b,\left[\begin{array}{c}2\\1\\3\\1\end{array}\right])d(c,1,\left[\begin{array}{c}5\\2\\0\end{array}\right])$, $\qquad$ 
$d(b,2,\left[\begin{array}{c}2\\1\\3\\1\end{array}\right]) 
\wedge d(c,1,\left[\begin{array}{c}5\\2\\0\end{array}\right])$.
\end{center}

\end{example}

\begin{lemma}\label{lem:Dais}

Let $c = (a,e,\vec{x})$ be a reduced $1$-cell.  For any cell $s$ such 
that $[c]\leq [s]$, $e \in s$.  Furthermore, $D_{a,i}(s) = D_{a,i}(c) = 
\vec{x}_i$ for each $i \in \{ 0, \dots, deg(a)-1\}$.

\end{lemma}

\begin{proof}

By the definition of $\leq$, $s$ must contain $e$.  The lemma then 
follows from the definition of $\sim$.

\end{proof}

For a vertex $a$ and an $a$-vector $\vec{x}$ with $|\vec{x}| = n$, 
define the $0$-form $f(a,\vec{x}): C_*(\ud{n}{T}) \to \Z/2\Z$ by:

  $$c \mapsto \left\{\begin{array}{cl} 
  1 & \hbox{ if } D_{a,i}(c) = x_i \hbox{ for all } i \in \{0, \dots, deg(a)-1\} \\
    & \hbox{ or } \overline{D}_{a,i}(c) = x_i \hbox{ for all } i \in \{0, \dots, deg(a)-1\} \\
  0 & \hbox{ otherwise,}
  \end{array}\right.$$
extended linearly to all cellular chains.  The constant function $1$ is 
a $0$-form: for any vertex $a \neq \ast$ of degree $1$ in $T$, $f(a,0) = 
1$.

Now we define $k$-forms.  Let $c = (a,e,\vec{x})$ be a reduced
$1$-cell.  Define the basic $1$-form $dc = d(a,e,\vec{x}): 
C_*(\ud{n}{T}) \to \Z/2\Z$ by:
  $$c' \mapsto \left\{\begin{array}{cl}
  1 & \hbox{ if } e \in c' \hbox{ and } f(a,\vec{x})([c']) = 1 \\
  0 & \hbox{ otherwise,}
  \end{array}\right.$$
and extend linearly to all cellular chains.  In general, a basic $k$-form is

  $$f(a,\vec{x})dc_1\wedge\dots\wedge dc_k,$$
where $f(a,\vec{x})$ is a $0$-form as above and each of $c_1$, $\dots$, 
$c_k$ is a distinct reduced $1$-cell.  Here, the wedge product 
represents conjunction: for $[c] \in H_*(B_nT)$, the $k$-form sends $[c]$ 
to $1$ if and only if $[c] \mapsto 1$ under $f(a,\vec{x})$ and under 
each $dc_i$.  If not all of the $c_i$ are distinct, then the $k$-form is 
identically the $0$-function.

As an example of forms, see Figure \ref{fig:Daiexample}.

Thus defined, forms have the following interpretation for cohomology:

\begin{proposition}[Forms and Cohomology]\label{prop:forms}

Let $T$ be a tree.  Fix a Morse $T$-embedding.  Then under the induced 
Morse matching:

  \begin{enumerate}

  \item If $c$ is a reduced $1$-cell, then $\phi_{[c]} = dc$.  In 
  particular, if $c$ is a critical $1$-cell, then $c^* = [dc]$.  

  \item If $c_1, \dots, c_k$ are reduced $1$-cells, then 
    $$[\phi_{[c_1]}] \cup \dots [\phi_{[c_k]}] = [dc_1 \wedge \dots 
    \wedge dc_k].$$
  In particular, if $c_1, \dots, c_k$ are critical $1$-cells, then 
    $$c_1^* \cup \dots \cup c_k^* = [dc_1 \wedge \dots \wedge dc_k].$$

  \item If $c_1, \dots, c_k$ are reduced $1$-cells such that 
  $\{[c_1],\dots, [c_k]\}$ has no upper bound or contains a repeated 
  element, then
    $$[dc_1 \wedge \dots \wedge dc_k] = [0].$$

  \end{enumerate}

\end{proposition}

\begin{proof}

Part (3) and that $c^* = [\phi_{[c]}]$ follow from Theorem 
\ref{thm:cohomology}.  The remaining statements follow from chasing 
definitions.

\end{proof}

Now that we have defined forms, we wish to define differentials for 
these forms. We will see in Proposition \ref{prop:d=delta} that 
coboundaries and differentials coincide.  Before defining differentials, 
though, we give a motivating theorem, Theorem \ref{thm:cohompres}, for 
why coboundaries are important. To state the theorem, we need a few more 
definitions.

Construct a simplicial complex $K$ from $\ud{n}{T}$ as follows.  First 
define a simplicial complex $K''$.  The vertex set of $K''$ corresponds 
to distinct equivalence classes of $1$-cells $[c]$.  Vertices $[c_1], 
\dots, [c_k]$ span a $(k-1)$-simplex in $K''$ if and only if $\{[c_1], 
\dots, [c_k]\}$ has an upper bound $[s]$, and is labelled by $[s]$.  
Note the labels on the faces of $K''$ induce an injective map from 
equivalence classes of cells in $\ud{n}{T}$ to $\Lambda(K'')$ (see 
Theorem \ref{thm:leq}).  Let $K'$ denote the $(\lfloor 
\frac{n}{2}\rfloor-1)$-skeleton of $K''$; then $\Lambda(K')= 
\Lambda(K'')/I''$, where $I''$ is the ideal of $\Lambda(K'')$ generated 
by all simplices of dimension greater than $\lfloor \frac{n}{2}\rfloor - 
1$.  Finally, define the complex $K$ to be the subcomplex of $K'$ where 
the only vertices left correspond to equivalence classes of $1$-cells 
$[c]$ with non-extraneous reduced representatives.  Then $\Lambda(K) = 
\Lambda(K')/I'$, where $I'$ is the ideal of $\Lambda(K')$ generated by 
equivalence classes of extraneous reduced $1$-cells (see Section 
\ref{sec:cohomology}).

\begin{definition}[Necessary Forms and Cells]

Let $\omega = f(a,\vec{x})dc_1 \wedge \dots \wedge dc_k$ be a $k$-form.  
The $k$-form $\omega$ is \emph{necessary} if:

  \begin{enumerate}

  \item $k < \lfloor n/2 \rfloor$~\footnote{The reason for this 
  dimension restriction is that the simplicial complex $K$ of Theorem 
  \ref{thm:cohompres} has no faces in dimension $\lfloor \frac{n}{2} 
  \rfloor$ or larger; see Theorem \ref{thm:cohomology}.},

  \item there exists an edge $e \in T$ such that $(a,e,\vec{x})$ is a 
  non-extraneous reduced $1$-cell,
  
  \item the set $\{[c_1], \dots, [c_k], [(a,e,\vec{x})]\}$ has an upper 
  bound $[s]$, and
  
  \item the edge $e$ is the unique respectful edge in the reduced 
  representative $s$ of $[s]$.

  \end{enumerate}
The reduced $1$-cell $(a,e,\vec{x})$ is the \emph{necessary reduced 
$1$-cell for} $\omega$, and is called \emph{necessary}.

\end{definition}

We leave it as an exercise for the reader to verify that there exists a 
unique necessary reduced $1$-cell for a given necessary form.

For an arbitary $k$-form $\omega$, the $(k+1)$-chain which is the sum of 
all equivalence classes in the support of the coboundary of $\omega$ 
will be called the \emph{coboundary support chain} for $\omega$.

\begin{theorem}[Presentation for Cohomology]
\emph{(}\cite{Farley2}, Theorem 4.5\emph{)}\label{thm:cohompres}

We have that
  $$H^*(B_nT) \cong \Lambda(K)/I,$$
where $I$ is the ideal of $\Lambda(K)$ generated by all coboundary 
support chains for necessary forms, viewed as cochains.  The isomorphism 
is induced by the injective map from equivalence classes of cells in 
$\ud{n}{T}$ to $\Lambda(K'')$.

\end{theorem}

\begin{proof}

This theorem is almost a rewording of Farley's Theorem 4.5 from 
\cite{Farley2}.  The difference is that Farley does not define the 
ideals $I$, $I'$, and $I''$, but instead views $H^*(B_nT)$ as isomorphic 
to a quotient of $\Lambda(K'')$.  A statement of the result presented 
here which more closely resembles Farley's theorem is that

  $$H^*(B_nT) \cong ((\Lambda(K'')/I'')/I')/I.$$
  
That we may quotient by the ideal $I''$ follows from Theorem 
\ref{thm:cohomology}:  cohomology is trivial in all dimensions greater 
that $\lfloor \frac{n}{2} \rfloor$.  

That we may then quotient by $I'$ is more difficult to see.  Clearly 
$I'$ is an ideal of $\Lambda(K')$.  Let $c = (a,e,\vec{x})$ be an 
extraneous $1$-cell.  If the edge $e$ of $c$ is such that $\tau(e)$ is 
essential, then this text and Farley agree: both quotient by the ideal 
generated by $\partial f(a,\vec{x})$.  Now consider if the edge $e$ of 
$c$ is such that $\tau(e)$ is not essential.  We may still talk about 
directions from $\tau(e)$, but there are at most $2$ directions.  Let 
$x_0 := D_{\tau(e),0}(c)$, and let $x_1 := D_{\tau(e),1}(c)$ if a second 
direction from $\tau(e)$ exists or $0$ otherwise.  Even though the cell 
$c$ is not necessarily uniquely determined from the information 
$(\tau(e),e,\vec{x})$, its equivalence class $[c]$ is.  For the purposes 
of this proof, we will abuse notation and write $[c] = 
[(\tau(e),e,\vec{x})]$.  For example, consider the cell $c$ in Figure 
\ref{fig:degree2} with edge $e$.  If exactly one of the non-filled-in 
vertices is an element of $c$, then $[c] = [(\tau(e),e, \left[ 
\begin{array}{c}1\\4\end{array} \right] )]$.  As the tuple $(\tau(e),e, 
\left[\begin{array}{c}1\\4\end{array}\right] )$ does not depend on which 
of the two non-filled-in vertices is an element of $c$, $c$ is not 
uniquely determined by it.

  \begin{figure}
    \begin{center}
    \input{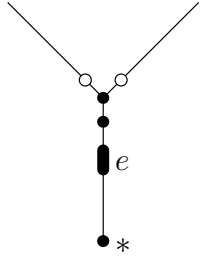}
    \end{center}

  \caption{An extraneous $1$-cell $c$ whose edge $e$ is such that 
  $\tau(e)$ is not essential.}
  
  \label{fig:degree2}
  
  \end{figure}

Consider the $\tau(e)$-vector $\vec{x}'$, where $x'_0 := x_0-1$ and 
$x'_1 := x_1 + 1$.  If $x_0 = 0$, or there are less than $x'_1$ vertices 
in direction $1$ from $\tau(e)$ in all of $T$, then there are no 
$1$-cells $c'$ satisfying $[c'] = [(\tau(e),e,\vec{x}')]$.  In this 
case, Farley includes $[c]$ in the generating set of his quotient ideal, 
as we have done here.  Otherwise, tracing Farley's definitions yields 
that the Farley generating set does not include $[c]$, but instead 
includes the chain $[c] + [(\tau(e),e,\vec{x}')]$.  In the quotient, 
$[c]$ and $[(\tau(e),e,\vec{x}')]$ are equivalent.  Inducting on the 
value $x_0$, we see that indeed $[c]$ is equivalent to $0$ in the 
quotient.  The ideal $I'$ is thus precisely the subideal of Farley's 
ideal corresponding to these chains.

Carefully tracing Farley's definitions and the definition of necessary 
show that the remaining generating cochains of Farley's ideal precisely 
coincide with coboundary support chains for necessary forms.  Note that 
in \cite{Farley2}, the convention that $\tau(e) < \iota(e)$ is switched.

\end{proof}

Let $\omega = f(a,\vec{x})dc_1 \wedge \dots \wedge dc_k$ be a necessary 
cochain.  Necessary cochains were defined so that each term in the 
coboundary support chain of $\omega$ has the form $dc_0 \wedge dc_1 
\wedge \dots \wedge dc_k$, where each $c_i$ is a non-extraneous critical 
$1$-cell for $i = 0, 1, \dots, k$.  We leave this to the reader to 
verify.  In particular, since each $c_i$ is non-extraneous and $k < 
\lfloor n/2 \rfloor$, every term in the coboundary support chain of 
$\omega$ represents a nontrivial element of $\Lambda(K)$.

We already know that cohomology is generated by duals of critical cells. 
Theorem \ref{thm:cohompres} tells us that we may effectively ignore all 
extraneous reduced $1$-cells.  \emph{From now on, all reduced cells will 
be assumed to be non-extraneous unless otherwise stated.}

Theorem \ref{thm:cohompres} also tells us that coboundaries of necessary 
forms show us how to rewrite duals of noncritical cells in terms of our 
Morse basis.  The isomorphism in the theorem takes a characteristic 
function $\phi_{[c]}$ of a cell $c$ to the element of $\Lambda(K)/I$ 
corresponding to $[c]$.

We wish to understand coboundaries of necessary forms.  To do so, we 
define differentials of forms, and show that our differential coincides 
with the operation of coboundary.

Define the \emph{differential} of a form as follows.  For a basic 
$0$-form $f(a,\vec{x})$, define

  $$df(a,\vec{x}) = \sum \left(f(a,\vec{x})(\partial c)\right) dc,$$
where the sum runs over all reduced $1$-cells $c$, and $\partial c$ is 
the CW-boundary of $c$ as a cellular chain.  We are left with an 
expression of $df(a,\vec{x})$ as a sum of basic $1$-forms, all of which 
have $1$ as the leading $0$-form.  Note that $d1 = 0$.  For a basic 
$k$-form $\omega = f(a,\vec{x})dc_1\wedge \dots\wedge dc_k$, define

  $$d\omega := d(f(a,\vec{x})\wedge dc_1\wedge \dots\wedge
  dc_k) = (df(a,\vec{x}))\wedge dc_1\wedge \dots\wedge dc_k.$$
Extend linearly to all $k$-forms. Since $d1 = 0$, it is clear that $d^2 
= 0$.

We are now ready to prove the relationship between differentials and 
coboundaries:

\begin{proposition}[Coboundary and Differential]\label{prop:d=delta}

The differential operation $d$ precisely coincides with the coboundary 
operation.

\end{proposition}

\begin{proof} 

To compute the coboundary $\delta \omega$ of an $k$-cochain $\omega = 
f(a,\vec{x})\wedge dc_1\wedge \dots \wedge dc_k$, we consider all 
$\sim$-equivalence classes of $(k+1)$-cells in $\ud{n}{T_{min}}$. Let 
$[s]$ be an equivalence class of $(k+1)$-cells, and let $s \in [s]$.  

Consider a face $s'$ of $s$.  The face $s'$ corresponds to replacing 
some edge $f$ of $s$ with one of its endpoints.  For any $i \in \{1, 
\dots, k\}$, we claim that if $dc_i(s) = 0$, then $dc_i(s') = 0$.  For, 
let $e_i$ denote the edge of $c_i$.  If $e_i \not\in s$, then $e_i 
\not\in s'$.  Consider if $e_i \in s$.  If $e_i \not\in s'$ (that is, if 
$f = e_i$), then $dc_i(s') = 0$ and there is nothing to prove.  If $e_i 
\in s'$ (that is, if $f \neq e_i$), then since $T$ is a tree both 
endpoints of $f$ are in the same direction from $\tau(e)$.  In 
particular, the number of vertices in each direction from $\tau(e)$ is 
the same for both $s$ and $s'$.  If $e_i \in s$, but $dc_i(s) = 0$, then 
by definition the number of vertices of $s$ (and thus $s'$) in some 
direction - say $j$ - from $\tau(e_i)$ is not equal to the number of 
vertices in $c_i$ in direction $j$ from $\tau(e_i)$; thus, $dc_i(s') = 
0$.

We have proven that for any $i \in \{1, \dots, k\}$, if $dc_i(s) = 0$, 
then for any face $s'$ of $s$, we will have $dc_i(s') = 0$.  By the 
definition of coboundary, $\delta \omega (s) = 1$ if and only if the sum 
of $\omega$ evaluated on the $2^{k+1}$ $k$-faces of $s$ is $1$ in 
$\Z/2\Z$.  Thus, if there exists an $i \in \{1, \dots, k\}$ such that 
$dc_i(s) = 0$, then $\delta\omega(s) = 0 = d\omega(s)$.

Assume that $dc_i(s) = 1$ for each $i \in \{1, \dots, k\}$.  Then $s$ 
contains the edge $e_i$ of the cell $c_i$ for each $i$.  If $e_i = e_j$ 
for some $i \neq j$, then since $dc_i(s) = dc_j(s) = 1$, it must be that 
$c_i = c_j$.  But then $dc_i\wedge dc_j = dc_i^2 = 0$, and again both 
$\delta\omega(s) = 0$ and $d\omega(s) = 0$.  Now assume the edges $e_i$ 
are all distinct.  Let $e$ denote the edge of $s$ which is not one of 
the $e_i$.  Consider again a face $s'$ of $s$ corresponding to replacing 
some edge $f$ of $s$ with one of its endpoints. If $f \neq e$, then $f = 
e_i$ for some $i$, so $dc_i$ and therefore $\omega$ evaluated on the 
given face will be $0$.  Let $s'$ and $s''$ be the two faces of $s$ 
corresponding to replacing $e$ with one of its endpoints. 
Thus, since $dc_i(s) = dc_i(s') = dc_i(s'') = 1$ for each $i \in \{1, 
\dots, k\}$,
  $$\delta\omega(s) = \omega(s') + \omega(s'') = f(a,\vec{x})(s') + 
  f(a,\vec{x})(s'').$$
Also, by the definition of differential,
  $$d\omega(s) = f(a,\vec{x})(\partial s)ds(s) = f(a,\vec{x})(s') + 
  f(a,\vec{x})(s'').$$
In this final case, we have again shown that $\delta\omega(s) = 
d\omega(s)$.  Since $[s]$ was chosen arbitrarily, $\delta\omega = 
d\omega$ as desired.

\end{proof}

The proof of Proposition \ref{prop:d=delta} actually gives us some 
computational techniques for considering differentials.  In particular, 
because of the last equation in the proof, we have the following 
scholium:

\begin{corollary}[Restricting the Differential]\label{cor:restrictingdf}

Let $f(a,\vec{x})$ be a basic $0$-form and let $s$ be a reduced 
$k$-cell.  

  \begin{enumerate}

  \item If $f(a,\vec{x})(\partial s) = 1$ then $s$ contains an edge $e$ 
  with $\tau(e) = a$.

  \item For each $i \in \{0, \dots, deg(a)-1\}$, let $x'_i := 
  D_{a,i}(s)$, and let $d$ be the direction from $a$ along $e$.  If 
  $f(a,\vec{x})(\partial s) = 1$ then either $\vec{x} = \vec{x'}$ or 
  $\vec{x'}$ differs from $\vec{x}$ by the subtraction of $1$ in the 
  $0^{th}$ entry and the addition of $1$ in the $d^{th}$ entry.\qed

  \end{enumerate}

\end{corollary}

More can be said about $d\omega$ when $\omega = f(a,\vec{x})dc_1\wedge 
\dots\wedge dc_k$ is a basic $k$-form.  Many of the $1$-forms in the sum 
used to define $d(f(a,\vec{x}))$ multiply to $0$ when wedged with 
$dc_1\wedge \dots\wedge dc_k$.  That is, many of the terms of the 
expansion of $d\omega$ with respect to the definition of $df(a,\vec{x})$ 
are trivial. We define an annihilator function to get rid of the trivial 
terms.

\begin{definition}[Annihilator]

Let $c_1, \dots, c_k$ be reduced $1$-cells.  Define an 
\emph{annihilator} function $A_{c_1,\dots,c_k}$ on $1$-forms as follows.  
Let $d(a,e,\vec{x})$ be a basic $1$-form. Then the annihilator acts 
either as the identity or the $0$ function on $d(a,e,\vec{x})$:
$A_{c_1,\dots,c_k}(d(a,e,\vec{x})) = d(a,e,\vec{x})$ if and only if 

  \begin{itemize}

  \item the equivalence classes $[c_1], \dots, [c_k],$ and 
  $[(a,e,\vec{x})]$ are all distinct, and

  \item the set of equivalence classes $\{[c_1], \dots, [c_k], 
  [(a,e,\vec{x})]\}$ has an upper bound;

  \end{itemize}
otherwise, $A_{c_1,\dots,c_k}(d(a,e,\vec{x})) = 0$.  Extend 
$A_{c_1,\dots,c_k}$ linearly to all $1$-forms.

\end{definition}

\begin{corollary}[Coboundaries as Annihilators]\label{cor:annihilator}

We have that 
  $$\delta\omega = d\omega = A_{c_1,\dots,c_k}(df(a,\vec{x}))\wedge 
  dc_1\wedge \dots\wedge dc_k.$$

\end{corollary}

\begin{proof}

This follows from the definition of the annihilator $A_{c_1,\dots,c_k}$, 
Proposition \ref{prop:d=delta}, and part (3) of Proposition 
\ref{prop:forms}.

\end{proof}

We refer the reader to the recent paper \cite{Farley2} for more 
information on coboundaries and cohomology presentations.

\section{Exterior Face Algebra Structures on Cohomology}\label{sec:extcohom}

Recall that a linear tree is one in which there exists an embedded line 
segment containing every essential vertex.  Consider the following 
theorem:

\begin{theorem}\label{thm:RAAcases}

Let $T$ be a tree.  If $T$ is linear or $n \leq 3$, then $H^*(B_nT; 
\Z/2\Z)$ is an exterior face algebra.

\end{theorem}

\begin{proof}

A tree braid group $B_nT$ is a right-angled Artin group (see Section 
\ref{sec:RAAGs} for a definition) if and only if $T$ is linear or $n 
\leq 3$ (\cite{ConnollyDoig} proves linear trees are right-angled Artin; 
\cite{Ghrist} and \cite{FarleySabalka1} prove $B_nT$ is free if $n \leq 
3$; \cite{FarleySabalka2a} proves the only if direction).  For any 
right-angled Artin group, the cohomology ring is an exterior face 
algebra \cite{CharneyDavis}.

\end{proof}

Let $T_{min}$ be the `minimal' nonlinear tree of Figure 
\ref{fig:terminology}: $T_{min}$ has exactly 4 essential vertices, each 
of degree 3, such that not all 4 essential vertices lie on an embedded 
line segment. In \cite{FarleySabalka2a}, it was shown that 
$H^*(B_4T_{min};\mathbb{Z}/2\mathbb{Z})$ is an exterior face algebra 
$\Lambda(\Delta)$; Figure \ref{fig:Deltamin} shows the complex $\Delta$. 
In this section, our goal is to expand the easy observations in Theorem 
\ref{thm:RAAcases} to show that the cohomology of a four, or even five, 
strand tree braid group is an exterior face algebra holds in general:

  \begin{figure}
    \begin{center}
    \input{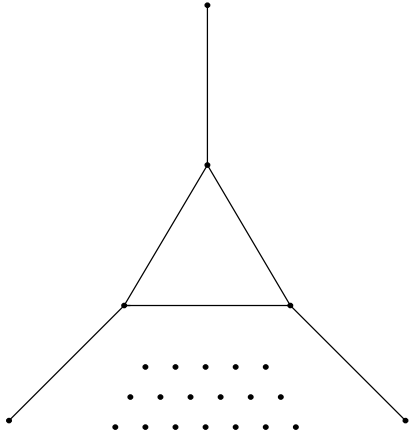}
    \end{center}

  \caption{The simplicial complex $\Delta$ giving the exterior face 
  algebra structure on $H^*(B_4T_{min}; \mathbb{Z}/2\mathbb{Z})$.}

  \label{fig:Deltamin}

  \end{figure}

\begin{theorem}[Exterior Face Algebra Structure on $4$ and $5$ Strand 
Cohomology]\label{thm:cohom}

Let $T$ be a finite tree.  For $n = 4$ or $5$, 
$H^*(B_nT;\mathbb{Z}/2\mathbb{Z})$ is an exterior face algebra.  The 
simplicial complex $\Delta$ defining the exterior face algebra structure 
is unique and at most $1$-dimensional.

\end{theorem}

We will comment on more general tree braid groups at the end of the 
section.  As our coefficients for cohomology are always in 
$\mathbb{Z}/2\mathbb{Z}$, we suppress this in the notation from now on.  
Even though a $1$-dimensional simplicial complex (for instance, 
$\Delta$) is a graph, we will continue to refer to them as simplicial 
complexes to maintain the distinction between defining complexes for 
exterior face algebras and graphs underlying graph braid groups. 

Some of the statements of Theorem \ref{thm:cohom} are immediate, and we 
prove them now.

\begin{lemma}[Uniqueness and the Dimension Bound]

Let $T$ be a finite tree.  For $n = 4$ or $5$, if $H^*(B_nT)$ is an 
exterior face algebra then the simplicial complex $\Delta$ defining the 
exterior face algebra structure is unique and at most $1$-dimensional.

\end{lemma}

\begin{proof}

By Theorem \ref{thm:cohomology} and a dimension computation for 
$\ud{n}{T}$ in \cite{FarleySabalka1}, cohomology is trivial in all 
dimensions greater than $\lfloor n/2\rfloor = 2$.  Thus, if $H^*(B_nT)$ 
is an exterior face algebra corresponding to some simplicial complex 
$\Delta$, $\Delta$ must be at most $(2-1)$-dimensional, by the shift of 
indices in the definition of exterior face algebras.  By Theorem 
\ref{thm:Gubeladze}, such a complex $\Delta$ is unique.

\end{proof}

To prove Theorem \ref{thm:cohom}, it remains to show that $H^*(B_nT)$ is 
an exterior face algebra.  We will spend much of the remainder of this 
section proving this, via a series of lemmas.  In essence, we will 
define a $1$-dimensional simplicial complex $\Delta$ and a homomorphism 
$\Psi: \Lambda(\Delta) \to H^*(B_nT)$, and then show the homomorphism is 
an isomorphism.  The bulk of our effort will be in defining $\Delta$, 
via a change of basis for $H^*(B_nT)$.

Until the end of the proof of Theorem \ref{thm:cohom}, we will assume 
that $n \in \{4, 5\}$.

\subsection{Two Computational Lemmas}\label{sec:lemmas}

In this subsection we establish two computational lemmas to be used 
throughout the remainder of this section.

\begin{lemma}[Necessary $0$-Forms]\label{lem:necessary0form}

Let $c = (a,e,\vec{x})$ be a reduced noncritical $1$-cell.  Then 
$f(a,\vec{x})$ is necessary, $c$ is necessary for the $0$-form 
$f(a,\vec{x})$, and $f(a,\vec{x})$ is the unique $0$-form for which $c$ 
is necessary.

\end{lemma}

\begin{proof}

This is a direct consequence of the definitions of necessary and 
$f(a,\vec{x})$.

\end{proof}

\begin{lemma}[Necessary $1$-Forms]\label{lem:necessary1form}

Let $c= (a,e,\vec{x})$ and $c_1 = (b,e_1,\vec{y})$ be critical 
$1$-cells.  If $\omega := f(a,\vec{x})dc_1$ is necessary and $c$ is the 
necessary reduced $1$-cell for $\omega$, then:

  \begin{enumerate}

  \item the direction from $b$ to $a$ is $0$,

  \item the direction $\alpha$ from $a$ to $b$ satisfies $0 < \alpha < 
  d$, and

  \item $\alpha$ is the unique direction between $0$ and $d$ for which 
  $x_\alpha \neq 0$.

  \end{enumerate}

\end{lemma}

\begin{proof}

Let $[s]$ be the least upper bound of $\{[c],[c_1]\}$.  By the 
definition of $\leq$, $s$ must contain the edges $e$ and $e_1$.  But the 
only edges of $s$ are $e$ and $e_1$, for if $s$ contains any other 
edges, replacing each with one of its endpoints yields a smaller upper 
bound.

Since $e$ is disrespectful in $c$, there is a vertex of $c$ between 
$\tau(e)$ and $\iota(e)$ in the order on vertices, i.e. in direction 
$d_0$ from $a$, where $0 < d_0 < d$. By Lemma \ref{lem:Dais}, $s$ still 
has a vertex or edge in direction $d_0$ from $a$.

We finish the proof by contradiction.  If the direction from $b$ to $a$ 
is nonzero, then the direction from $a$ to $b$ must be $0$ (since $T$ is 
a tree) - in particular, $e_1$ is not in direction $d_0$ from $a$. If 
$\alpha$ does not satisfy $0 < \alpha < d$, then $\alpha \neq d_0$, and 
again $e_1$ is not in direction $d_0$ from $a$.  If $c$ has two nonzero 
directions less than $d$ such that $c$ has vertices in both directions 
from $a$, then without loss of generality we may assume $d_0$ is such 
that $d_0 \neq \alpha$, and again $e_1$ is not in direction $d_0$ from 
$a$.

Thus, if any of the conclusions of this lemma do not hold, then $s$ has 
a vertex in direction $d_0$ from $a$.  Assume without loss of generality 
that $s$ is reduced, and let $v$ be the least vertex in direction $d_0$ 
from $a$.  Then $v$ must be blocked in $s$ by either $\ast$, $e_1$, or 
$e$.  Since $d_0 \neq 0$, $v$ cannot be blocked by $\ast$.  Since $d_0 
\neq \alpha$, $v$ cannot be blocked by $e_1$.  So, $v$ must be blocked 
by $e$.  This makes $e$ disrespectful in $s$, since $0 < d_0 < d$.  This 
contradicts the definition of necessary, and proves the lemma.

\end{proof}

\subsection{Towards Changing Bases}\label{sec:M_c}

In this subsection, we describe why we need to change bases to find the 
exterior face algebra structure on cohomology.  We also define many 
matrices, some associated to necessary forms and some to critical 
$1$-cells, which we will use in the next section to define our change of 
basis matrix.

Fix a Morse $T$-embedding, so that we have a classification of critical 
cells in $T$.  Consider the finite simplicial complex $\Delta'$, defined 
as follows.  Let the vertex set of $\Delta'$ be identified with the set 
$\{c^* | c \hbox{ a critical $1$-cell}\}$.  A set of vertices $\{c_1^*, 
c_2^*\}$ span a $1$-simplex, labelled $c_1^* \cup c_2^*$, if and only if 
$c_1^* \cup c_2^*$ is nontrivial.

Let $\Psi':\Lambda(\Delta') \to H^*(B_nT)$ be the map which takes an 
element of $\Lambda(\Delta')$ corresponding to a vertex of $\Delta'$ 
labelled $c^*$ and maps it to the $i$-cohomology class $c^*$.  Since 
critical $1$-cells form a free basis for $H^1(B_nT)$, the map extends to 
all of $H^1(B_nT)$.  By the definition of $\Delta'$, $\Psi'$ is 
surjective onto $H^1(B_nT)$ and moreover extends linearly to a 
surjective homomorphism, since $H^1(B_nT)$ generates all of $H^*(B_nT)$ 
(see Theorem \ref{thm:cohomology}).

If $\Psi'$ were injective, then $\Psi'$ would be the desired 
isomorphism.  Recall that a critical $2$-cell $s$ uniquely determines 
the pair $\{c_1, c_2\}$ of critical $1$-cells for which it is an upper 
bound (Theorem \ref{thm:cohomology}).  As critical $2$-cells form a 
basis for $2$-dimensional cohomology, $\Psi'$ is injective if and only 
if, for every pair $\{c_1, c_2\}$ of critical $1$-cells which has a 
least upper bound $[s]$, we may find a representative $s \in [s]$ such 
that $s$ is a critical $2$-cell.  Unfortunately, though, this is not the 
case, and $\Psi'$ is not injective.

Our goal is to modify the Morse generating set of $H^*(B_nT)$ to address 
these cases.  We will define a new simplicial complex $\Delta$ using 
this new generating set so that the corresponding map $\Psi: 
\Lambda(\Delta) \to H^*(B_nT)$ will be an isomorphism.

To modify the generating set of duals of critical $1$-cells, we need to 
specify a change of basis.  Since duals of critical $1$-cells freely 
generate $H^1(B_nT)$, we may introduce a vector-theoretic interpretation 
of $H^1(B_nT)$.  Then, using this vector-theoretic interpretation, we 
specify an invertible change of basis matrix $M$.  The matrix $M$ will 
be a product of invertible matrices, one for each critical cell, in a 
particular order, as we will see.

Define $<_r$ to be a (usually non-unique) total order on reduced 
$1$-cells $(a,d,\vec{x})$ induced by lexicographically ordering the 
triple $(a,-x_0,d)$.  Note we are ordering all reduced $1$-cells here, 
not just critical ones.  Let $rm$ be the number of reduced $1$-cells, 
$sm$ the number of critical $1$-cells, and $tm$ the number of 
non-critical reduced $1$-cells, so $rm = sm + tm$.  Define a map $ri$ on 
reduced $1$-cells, so that for a reduced $1$-cell $c$, its image $ri(c) 
\in \{1, \dots, M\}$ is its index in the total order $<_r$, so that the 
$<_r$-smallest reduced $1$-cell has index $1$, the second smallest has 
index $2$, etc.  Also define maps $si(c)$ and $ti(c)$ on critical and 
non-critical reduced $1$-cells respectively, so that for a critical 
(respectively, noncritical) $1$-cell $c$, $si(c)$ is its index among 
critical (respectively, noncritical) cells in the total order $<_r$.  
Ergo, the $<_r$-smallest critical $1$-cell maps to $1$ under $ci$, the 
second smallest to $2$, etc.

There is a bijection between reduced $1$-cells and the standard basis 
vectors for $\mathbb{F}_2^rm$, where a reduced $1$-cell $c$ corresponds 
to the vector $\vec{v}_c$ consisting of all $0$s except a $1$ in the 
$ri(c)^{th}$ row.  By Theorems \ref{thm:homology} and 
\ref{thm:cohomology}, this bijection induces a surjective homomorphism 
from $\mathbb{F}_2^{rm}$ to each of the rings $H_1(B_nT)$ and 
$H^1(B_nT)$.  The surjection is an isomorphism on the $sm$-dimensional 
subspace corresponding to the critical $1$-cells, where for a critical 
cell $c$, $\vec{v}_c$ is mapped to to $[c]$ or $c^*$, respectively.

Theorem \ref{thm:cohompres} tells us that, if we want to rewrite our 
basis for cohomology to reduce the number of relations in our 
presentation, then we need to focus on necessary $k$-forms.  So, we use 
these vector representatives of critical $1$-cells to associate to each 
necessary $k$-form $\omega$ a matrix $M_{\omega}$.

Let $\omega = f(a,\vec{x})dc_1\wedge \dots \wedge dc_k$ be a necessary 
$k$-form with $c$ the necessary reduced $1$-cell for $\omega$.  Consider 
the annihilator portion $A_{c_1,\dots,c_k}(df(a,\vec{x}))$ of $d\omega$.  
Let $\vec{u}_\omega \in \mathbb{F}_2^{rm}$ be the vector whose nonzero 
entries exactly correspond to nonzero terms of 
$A_{c_1,\dots,c_k}(df(a,\vec{x}))$.  That is, $(u_\omega)_{ri(c')} = 1$ 
if and only if $dc'$ is a nonzero term of 
$A_{c_1,\dots,c_k}(df(a,\vec{x}))$.  Define the $rm \times rm$ matrix 
$M_{\omega}$ to be the identity matrix $I_m$, but with the $ri(c)^{th}$ 
column replaced by $\vec{u}_{\omega}$.

\begin{lemma}\label{lem:lowertri}

For any necessary $k$-form $\omega = f(a,\vec{x})dc_1 \wedge \dots 
\wedge dc_k$, the necessary reduced $1$-cell $c$ for $\omega$ is the 
$<_r$-smallest reduced $1$-cell $c'$ such that $dc'$ appears as a 
nonzero term in $A_{c_1,\dots,c_k}(df(a,\vec{x}))$. 
\end{lemma}

Note this lemma holds for arbitrary $n$, not just $n = 4$ and $5$.

In Convention \ref{conv:fixingorder}, we will slightly modify the 
definition of $<_r$ in the cases $n = 4$ or $5$.  One consequence will 
be that, for a necessary $1$-form whose associated necessary critical 
$1$-cell $c$ is exceptional of Type I (to be defined soon), this lemma 
and its corollary will not hold - in fact, $M_{\omega}$ will be upper 
triangular.  The modification will not affect our applications of these 
results, though.

\begin{proof}[Proof of Lemma \ref{lem:lowertri}]

Let $c = (a,e,\vec{x})$ be the reduced $1$-cell for the  $\omega$.  Let $c' \neq c$ be any other 
reduced $1$-cell such that $dc'$ appears nontrivially in 
$df(a,\vec{x})$.  By Restricting the Differential (Corollary 
\ref{cor:restrictingdf}), $c'$ must have an edge $e'$ with $\tau(e') = 
a$.  By the definition of differential, we know that 
$f(a,\vec{x})(\partial c') = 1$.  Let $c'_\iota$ and $c'_\tau$ denote 
the faces of $c'$, corresponding to replacing $e'$ with $\iota(e')$ and 
$\tau(e')$, respectively.  Then exactly one of $f(a,\vec{x})(c'_\iota)$ 
or $f(a,\vec{x})(c'_\tau)$ is $1$.  For each direction $i$ from $a$,

  $$D_{a,i}(c'_\iota) = \overline{D}_{a,i}(c'_\iota) = D_{a,i}(c'),$$
and

  $$D_{a,i}(c'_\tau ) = \overline{D}_{a,i}(c'_\tau ) = 
  \overline{D}_{a,i}(c').$$

We need to show that $c <_r c'$.  If $f(a,\vec{x})(c'_\tau) = 1$ then 
for $i = 0$ we have $D_{a,0}(c') = \overline{D}_{a,0}(c') = x_0 - 1$, so 
$c <_r c'$. It remains to consider the case when $f(a,\vec{x})(c'_\iota) 
= 1$. Then $D_{a,i}(c') = x_i = D_{a,i}(c)$, and in particular, $c' = 
(a,e',\vec{x})$.

By the definition of necessary, $c$ is such that:

  \begin{enumerate}

  \item the set $\{[c_1], \dots, [c_k], [c]\}$ has an upper 
  bound $[s]$, and

  \item the edge $e$ is respectful in the reduced representative $s$ of 
  $[s]$.

  \end{enumerate}
Let $d > 0$ be the direction from $a$ along $e$ and let $d' > 0$ be the 
direction from $a$ along $e'$. Since $D_{a,i}(c') = x_i$ for each $i$ 
and $e' \in c'$, $D_{a,d'}(c) = D_{a,d'}(c') \geq 1$.  Unless $d = d'$ 
so that $c = c'$, $s$ must have a vertex in direction $d'$ from $a$.  
Since $s$ is reduced, that vertex must be blocked by $e$ in $s$.  Since 
$e$ is respectful in $s$, $d'> d$. Thus, $c <_r c'$.
\end{proof}

\begin{corollary}[$M_{\omega}$ is Lower Triangular]\label{cor:lowertri}

For any necessary $k$-form $\omega$, the matrix $M_{\omega}$ is lower 
triangular and invertible.

\end{corollary}

\begin{proof}

That $c <_r c'$ for any other $c'$ for which $dc'$ appears nontrivially 
in the annihilator $A_{c_1,\dots,c_k}(df(a,\vec{x}))$ makes $M_{\omega}$ 
a lower triangular matrix. By the definition of necessary $1$-cell, $dc$ 
also appears nontrivially in $A_{c_1,\dots,c_k}(df(a,\vec{x}))$.  Thus, 
$M_{\omega}$ is a lower triangular matrix, all of whose diagonal entries 
are 1.  Therefore, $M_{\omega}$ is invertible.

\end{proof}

At this point, we have associated to each necessary form $\omega$ an 
invertible matrix $M_{\omega}$.  We now use these matrices associated to 
forms to define a matrix $M_c$ for each reduced $1$-cell $c$.  

Let $c$ be a reduced $1$-cell.

If $c = (a,d,\vec{x})$ is not critical, then $c$ is necessary for the 
necessary $0$-form $\omega = f(a,\vec{x})$ (Lemma 
\ref{lem:necessary0form}).  Define $M_c$ to be the matrix $M_{\omega}$, 
but with a $0$ on the diagonal in the $ri(c)^{th}$ row.  In terms of 
multiplication by $M_c$, this will correspond to \emph{replacing $dc$ 
with a cohomologically equivalent cochain determined by the coboundary 
of $\omega$}.

If $c$ is critical, there are three types of exceptions we must make 
when defining $M_c$.  Before stating the exceptions, we state the 
general cases.  If $c$ is critical, not exceptional, and not necessary, 
define the matrix $M_c$ to be the identity matrix.  If $c$ is critical, 
not exceptional, and necessary, define $M_c$ to be the matrix 
$M_{\omega}$ for any $1$-form $\omega$ for which $c$ is necessary. This 
corresponds to \emph{rewriting $dc$ with a $1$-cochain that will 
eliminate the relation corresponding to $\omega$}.  We will prove 
momentarily that $M_c$ is well-defined.

We now state the three types of exceptions. The motivation for the 
definitions of $M_c$ for the exceptional cases is not intuitively 
apparent, but the lemmas in the remainder of this section will justify 
our choices. All of the exceptional types have $n = 5$.  Let $c = 
(a,d,\vec{x})$ be critical. Assume that there exist two directions 
$dir1$ and $dir2$ from $a$ for which $x_{dir1}, x_{dir2} \geq 2$.  
Without loss of generality, assume $dir1 < dir2$.  Let $dir3$ be: $dir1$ 
if $x_{dir1} = 3$, or $dir2$ if $x_{dir2} = 3$, or the unique index such 
that $x_{dir3} = 1$.  As $|\vec{x}| = 5$, $dir3$ is uniquely determined.

  \begin{figure}
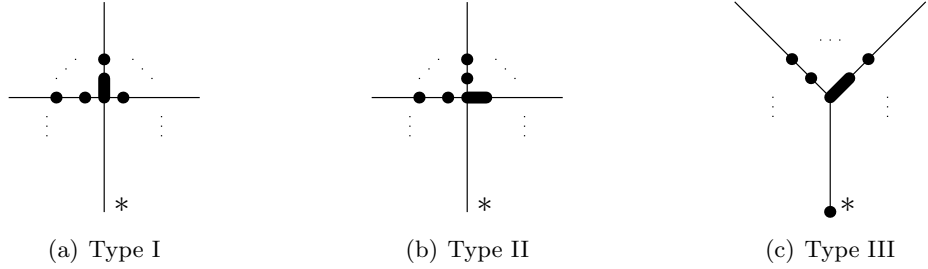

    \begin{center}
    \subfigure[Type I]{
      \input{exTypeI.pstex_t}
    }
    \hfill
    \subfigure[Type II]{
      \input{exTypeII.pstex_t}
    }
    \hfill
    \subfigure[Type III]{
      \input{exTypeIII.pstex_t}
    }
    \end{center}

  \caption{The three types of exceptional critical $1$-cells.  For each 
  type, all that is drawn are the relevant edges of the tree $T$:  $T$ 
  may have other essential vertices, and the essential vertex may have 
  degree greater than 4 (for Types I and II) or 3 (for Type III), as 
  suggested by the ellipses.  As exceptional cells are critical, the 
  diagrams depict reduced $1$-cells, even though a subdivision of $T$ is 
  not shown.}
  
  \label{fig:exceptions}
  
  \end{figure}

  \begin{enumerate}
  \item (Type I) $0 < dir1 < dir2 < dir3$, and $d = dir2$: $M_c$ is the 
  identity matrix.

  \item (Type II) $0 < dir1 < dir2 < dir3$, and $d = dir3$: Let $c'$ be 
  the critical cell $(a,dir2,\vec{x})$ - that is, replace the edge of 
  $e$ in direction $dir3$ from $a$ with an edge in direction $dir2$ from 
  $a$. Then $M_c$ is the $rm\times rm$ identity matrix, but with an 
  extra $1$ in the $ri(c)^{th}$ row and $ri(c')^{th}$ column.  Note as 
  defined $c' <_r c$, so $M_c$ is \emph{upper} triangular.  We want 
  $M_c$ to be lower triangular, so we will slightly modify the order 
  $<_r$ in Convention \ref{conv:fixingorder}, below.  The Type I cell 
  $c'$ and the Type II cell $c$ are said to \emph{correspond} to each 
  other.  Note there is a bijection between cells of Type I and cells of 
  Type II given by this correspondence.

  \item (Type III) $0 = dir3 < dir1 < dir2$, and $d = dir2$: For $i = 1, 
  \dots, deg(a)-1$, define $\vec{y_i}$ to be the $a$-vector which is 
  $\vec{x}$, but with $(\vec{y_i})_0 = x_0-1 = 0$ and $(\vec{y_i})_i = 
  x_i+1$.  Let $\vec{u}' \in \mathbb{F}_2^{rm}$ be the vector whose only 
  nonzero entries are exactly a $1$ in the $ri(c')^{th}$ row for every 
  reduced cell $c'$ such that $c' = c$ or $c' = (a,dir2,\vec{y_i})$ for 
  some $i \in \{1, \dots, deg(a)-1\}$, $i \neq d$.  Define $M_c$ to be 
  the $rm\times rm$ identity matrix $I_m$, but with the $ri(c)^{th}$ row 
  replaced by the transpose $(\vec{u}')^T$.  Note that $c <_r 
  (a,dir2,\vec{y_i})$ for each $i \in \{1, \dots, deg(a)-1\}$, so $M_c$ 
  is lower triangular.

  \end{enumerate}

\begin{convention}[A Modification to $<_r$]\label{conv:fixingorder}

We slightly modify the total order $<_r$ on reduced $1$-cells by having 
corresponding exceptional critical cells of Types I and II switch 
places.  By convention, the matrices $M_\omega$ and $M_c$ defined above 
are defined using this modified total order.  All further references to 
$<_r$, and the associated functions $ri$ and $ci$, will be to the 
modified total order. 

\end{convention}

Note that with this new convention, Lemma \ref{lem:lowertri} no longer 
holds for necessary forms $\omega$ where the associated necessary 
$1$-cell is exceptional of Type I - in fact, instead of being lower 
triangular, $M_\omega$ will be upper triangular.  This is fine, as 
$M_\omega$ was not used to define any matrix $M_c$.  It is a small 
exercise to verify that, for all other cases, Lemma \ref{lem:lowertri} 
still holds.

\begin{lemma}[$M_c$ is Well-Defined]\label{lem:Mcwelldefined}

Let $c$ be a reduced $1$-cell.  The matrix $M_c$ is well-defined and 
lower triangular.

\end{lemma}

\begin{proof}

For the three exceptional cases, the matrix $M_c$ was uniquely 
determined.  If $c$ is exceptional of Type I, $M_c$ is the identity and 
is lower triangular.  If $c$ is exceptional of Type II, then Convention 
\ref{conv:fixingorder} makex $M_c$ lower triangular.  If $c$ is 
exceptional of Type III, the argument that $M_c$ is lower triangular 
given in the definition of $M_c$ still applies, even with Convention 
\ref{conv:fixingorder}.

Aside from the three exceptional cases, if $c$ is not necessary or not 
critical, the matrix $M_c$ is uniquely specified and lower triangular by 
Corollary \ref{cor:lowertri}.  So, assume $c$ is necessary and critical.  
Corollary \ref{cor:lowertri} shows that $M_c$ will be lower triangular 
if it is well-defined.  To claim that $M_c$ is well-defined is to claim 
that, if $\omega_1$ and $\omega_2$ are any two $1$-forms for which $c$ 
is necessary, $M_{\omega_1} = M_{\omega_2}$.  Let $c = (a,d,\vec{x})$.  
By the definition of necessary, there exist critical $1$-cells $c_1 = 
(b_1,d_1,\vec{y})$ and $c_2 = (b_2, d_2, \vec{z})$ such that $\omega = 
f(a,\vec{x})\wedge dc_1$ and $\omega = f(a,\vec{x})\wedge dc_2$.  By the 
Upper Bound Lemma (Lemma \ref{lem:upperbound}), the directions 
$\alpha_1$ and $\alpha_2$ from $a$ to $b_1$ and $b_2$, respectively, are 
such that $\vec{x}_{\alpha_1} = n-y_0 \geq 2$ and $\vec{x}_{\alpha_2} = 
n - z_0 \geq 2$.  By Lemma \ref{lem:necessary1form}, $0 < \alpha_1 < d$ 
and $0 < \alpha_2 < d$, so if $\alpha_1 \neq \alpha_2$, $c$ is an 
exceptional case of Type II.  Since we are not addressing the 
exceptional cases, $\alpha_1 = \alpha_2$.

The formulas $\vec{x}_{\alpha_1} = n-y_0$ and $\vec{x}_{\alpha_2} = n - 
z_0$ imply $y_0 = z_0$.  By the Restricting the Differential (Corollary 
\ref{cor:restrictingdf}), a nonzero term of $df(a,\vec{x})$ is of the 
form $ds$, where $s$ is a reduced $1$-cell which lies over the vertex 
$a$.  Since $y_0 = z_0$ and $a$ is in direction $0$ from both of $b_1$ 
and $b_2$, by the Upper Bound Lemma (\ref{lem:upperbound}), $[s]$ has an 
upper bound with $[c_1]$ if and only if $[s]$ has an upper bound with 
$[c_2]$.  In particular, by the definition of annihilator, $A_{c_1}(ds) 
= A_{c_2}(ds)$ for each such $s$ - i.e. $A_{c_1}(df(a,\vec{x})) = 
A_{c_2}(df(a,\vec{x}))$.  The equality of the matrices $M_{\omega_1}$ 
and $M_{\omega_2}$ follows.

\end{proof}

\subsection{Changing Bases and Finishing Theorem 
\ref{thm:cohom}}\label{sec:cohom}

We now have lower triangular matrices $M_c$ for each reduced $1$-cell 
$c$.  We want to use the matrices $M_c$ to define a change of basis for 
$H^1(B_nT)$.

From now on, we think of the matrices $M_c$ as acting on reduced 
cocycles, where we identify a reduced cocycle $dc'$ and the vector in 
$\mathbb{F}_2^{rm}$ whose only nonzero entry is in the $ri(c')^{th}$ 
row.

To define our change of basis, we need to specify how to rewrite duals 
of critical cells.  We do so on the cochain level, using $1$-forms 
associated to both critical and noncritical reduced $1$-cells.  For 
critical reduced $1$-cells, we define a matrix $Ms$.  For noncritical 
reduced $1$-cells, we define a matrix $Mt$. We multiply the matrices 
$Ms$ and $Mt$ to define the desired change of basis matrix $M$, as 
follows.

\begin{definition}[The Matrices $Ms$, $Mt$, and $M$]

Define the matrix $Ms$ as:
  $$Ms := \prod_{\substack{i=1\\si(c)=i}}^{sm}M_c,$$
where the product is over all critical cells $c$ and is written so that 
the $<_r$-largest cell $c$ is such that $M_c$ is on the right, applied 
first to any target vector.
Define the matrix $Mt$ as:
  $$Mt := \prod_{\substack{i=1\\ti(c)=tm-1+i}}^{tm}M_c,$$
where the product is over all non-critical reduced cells $c$ and is 
written so that the $<_r$-smallest cell $c$ is such that $M_c$ is on the 
right, applied first to any target vector.
Define the matrix $M$, which will be the desired change of basis matrix, 
as:
  $$M := MtMs.$$

\end{definition}

The matrices $Ms$, $Mt$, and $M$ have many nice properties, some of
which we describe now.

Let $c_0$ be a reduced $1$-cell.  By definition, $M_{c}dc_0 = dc_0$ for 
every $c \neq c_0$.  Also by definition, $M_{c_0}dc_0$ consists of terms 
of the form $dc_0'$, where $c_0$ and $c_0'$ lie over the same vertex.  
Since each $M_c$ is lower triangular (Lemma \ref{lem:Mcwelldefined}),
$c_0 \leq_r c_0'$, even when $c$ is exceptional.

If $c_0$ is a critical $1$-cell, then $Mtdc_0 = dc_0$ and
  $$Msdc_0 = \prod_{\substack{i=1\\si(c)=i}}^{sm}M_cdc_0 = 
  M_{c_0}dc_0.$$
This means that $Ms$ agrees with the matrix $M_{c_0}$ in the 
$ri(c_0)^{th}$ column.  Since $M_c$ is lower triangular and invertible 
for any critical $1$-cell $c$, the matrix $Ms$ is also lower triangular 
and invertible.

If $c_0$ is a noncritical $1$-cell, $M_{c_0}dc_0$ is by definition 
cohomologous to $dc_0$.  As $M_{c_0}$ is the identity outside of column 
$ri(c_0)$, the matrix $Mt$ preserves cohomology classes.  As the matrix 
$M_c$ is lower triangular for each noncritical $1$-cell $c$, the matrix 
$Mt$ is lower triangular.  Moreover, as the matrix $M_{c_0}$ contains a 
$0$ on the diagonal in the $ri(c_0)^{th}$ place, the matrix $Mt$ has no 
nonzero entries in the $ri(c_0)^{th}$ row.  Finally, since $c_0$ is 
noncritical, $Msdc_0 = dc_0$.

These observations imply that $M = MtMs$, restricted to basic $1$-forms 
corresponding to critical $1$-cells, is a change of basis matrix on 
cohomology classes, where for a critical $1$-cell $c$, $Mc^* = [Mdc]$. 
We now have our change of basis matrix $M$ on $H^1(B_nT)$ and therefore 
on $H^*(B_nT)$.  We may thus write $M$ as acting on cohomology classes 
(like $c^*$) instead of cochains (like $dc$), and will do so freely from 
now on.  For future reference, we record some of our observations in a 
theorem:

\begin{theorem}[Change of Basis Theorem]\label{thm:changeofbasis}
~\\
  \begin{enumerate}

  \item The matrix $M$ gives a change of basis isomorphism for 
  $H^1(B_nT)$ and therefore $H^*(B_nT)$.  For a critical $1$-cell $c$, 
  $Mc^* := [Mdc]$.

  \item The matrices $M$, $Ms$, and $Mt$ are lower triangular.  The
  matrix $Ms$ is invertible.

  \item For a critical $1$-cell $c$, $Msdc = M_cdc$. The cochain $M_cdc$ 
  consists of terms of the form $dc_0'$, where $c$ and $c_0'$ lie over
  the same vertex and $c \leq_r c_0'$.

  \item The matrix $Mt$ preserves cohomology classes. \qed

  \end{enumerate}

\end{theorem}

We need to analyze the effect of $M$ on all of cohomology, not just 
$H^1(B_nT)$.  In particular, we need to know how $M$ affects cup 
products.

\begin{lemma}[Cup Products After Changing Basis]\label{lem:effectofM}

Let $c$ and $c'$ be critical $1$-cells, where $c \leq_r c'$.  If $Mc^* 
\cup M(c')^* \neq [0]$, then $[c]$ and $[c']$ have a least upper bound 
$[s]$ with reduced representative $s$ and either:

  \begin{enumerate}

  \item $s$ is critical, or

  \item $s$ is not critical and $c$ is exceptional of Type I.

  \end{enumerate}

\end{lemma}

\begin{proof}

We will repeatedly use the properties of $M$, $Ms$, and $Mt$ stated in 
Theorem \ref{thm:changeofbasis} to prove this lemma, and do so without 
further reference.

Since $M = MtMs$ and $Mt$ preserves cohomology classes, $Mc^* \cup 
M(c')^* \neq [0]$ if and only if $[Msdc]\cup [Msdc'] \neq [0]$.  We have 
that $Msdc = M_cdc$ and $Msdc' = M_{c'}dc'$.  Thus, $Mc^*\cup M(c')^* 
\neq [0]$ if and only if $[M_cdc]\cup [M_{c'}dc'] \neq [0]$.

Express $c$ as $(a,e,\vec{x})$ and $c'$ as $(a',e',\vec{x'})$.

If $a = a'$, then every summand in both $M_cdc$ and $M_{c'}dc'$ will lie 
over $a$.  By the Upper Bound Lemma (Lemma \ref{lem:upperbound}), no two 
reduced $1$-cells which lie over the same vertex have an upper bound.  
By the definition of $\wedge$, it follows that $M_cdc\wedge M_{c'}dc'$ 
is the $0$ function, so $[Msdc]\cup [Msdc'] = [0]$.  Thus, we may assume 
that $a \neq a'$.

Consider summands $dc_0$ and $dc'_0$ of $M_cdc$ and $M_{c'}dc'$, 
respectively.  Express $c_0$ as $(a,d_0,\vec{y})$ and $c'_0$ as 
$(a',d_0',\vec{y'})$. By the definition of $M_c$ and Restricting the 
Differential (Corollary \ref{cor:restrictingdf}), $\vec{y}$ differs from 
$\vec{x}$ for at most two indices, namely $0$ and some index $dir_0$.  If 
$\vec{y} \neq \vec{x}$ then $y_0 = x_0-1$ and $y_{dir_0} = x_{dir_0}+1$.  
In particular, $x_0 \geq y_0$. Furthermore, by definition, if $c$ is not 
exceptional of Type III, then $d_0' = d_0$.  Similar statements hold for 
$dc'_0$, but we will only need to use that $x'_0 \geq y'_0$.

We claim that $[M_cdc]\cup [M_{c'}dc'] \neq [0]$ if and only if $[M_cdc] 
\cup [dc'] \neq [0]$, as follows.

Consider if $[M_cdc]\cup [dc'] = [0]$ but $[M_cdc]\cup[M_{c'}dc'] \neq 
[0]$.  Then there exists some summand $dc'_0 \neq dc'$ of $M_{c'}dc'$ 
such that 
$[M_cdc]\cup[dc'_0] \neq [0]$.  Thus, there must exist a summand $dc_0$ 
of $M_cdc$ such that $dc_0\wedge dc'_0$ is not cohomologous to $0$.  In 
particular, $[c_0]$ and $[c'_0]$ have an upper bound.  By the Upper 
Bound Lemma, since $x'_0 \geq y'_0$, $[c_0]$ and $[c']$ have an upper 
bound.  Since $[M_cdc]\cup [dc'] = [0]$, there exists a sequence of 
Tietze transformations which equates $dc_0\wedge dc'$ with $(M_cdc\wedge 
dc' -dc_0\wedge dc')$.  Each Tietze transformation involving a relation 
will involve the coboundary of a necessary $1$-form.  By Restricting the 
Differential, such a $1$-form has either the form $f(a,\vec{z})dc'$ or 
$f(a',\vec{z'})dc_0$ for some $a$-vector $\vec{z}$ or $a'$-vector 
$\vec{z'}$.  By Lemma \ref{lem:necessary1form}, such a $1$-form must be 
of the form $f(a,\vec{z})dc'$, since $a$ is in direction $0$ from $a'$.  
By the definition of necessary, $f(a,\vec{z})dc'_0$ will also be a 
necessary $1$-form.  It follows that a corresponding sequence of Tietze 
transformations equates $dc_0\wedge dc'_0$ with $(M_cdc\wedge dc'_0 - 
dc_0\wedge dc'_0)$.  But then $[M_cdc \wedge dc'_0] = [0]$, a 
contradiction.

To finish the claim, it remains to prove that if $[M_cdc] \cup [dc'] 
\neq [0]$, then $[M_cdc] \cup [M_{c'}dc'] \neq [0]$.  If 
$[M_cdc]\cup[dc'] \neq [0]$ but $[M_cdc]\cup [M_{c'}dc'] = [0]$, then 
there exists some summand - say $dc'_0$ - of $M_{c'}dc'$ such that 
$dc_0\wedge dc'$ and $dc_0\wedge dc'_0$ are both in a single relation - 
that is, in the differential of a single necessary $1$-form $\omega'$.  
By Restricting the Differential, it must be that $\omega' = 
f(a',\vec{z})dc_0$ for some $a'$-vector $\vec{z}$.  This contradicts 
Lemma \ref{lem:necessary1form}, as the direction from $a'$ to $a$ is 
$0$.  This proves the claim.

To prove the lemma, we show that if $[M_cdc] \cup [dc'] \neq [0]$, then 
$[c]$ and $[c']$ have a least upper bound $[s]$ with reduced 
representative $s$, and either $s$ is critical or $s$ is not critical 
but $c$ is exceptional of Type I.

By Theorem \ref{thm:cohomology}, $[dc_0]\cup [dc_0'] \neq [0]$ only 
if $[c_0]$ and $[c_0']$ have an upper bound.  Let $\alpha > 0$ be the 
direction from $a$ to $a'$.  By the Upper Bound Lemma, $[c_0]$ and 
$[c_0']$ have a common upper bound if and only if

  \begin{equation}\label{eq:1}
  y_\alpha + y'_0 \geq n + \epsilon_0,
  \end{equation}
where $\epsilon_0 := \epsilon(c_0,\alpha)$ is the upper bound constant.  
We record the interpretation of equation \ref{eq:1} for the case that 
$c_0 = c$, as these two equations are the heart of what remains in the 
proof.  We have that $[c]$ and $[c']$ have a common upper bound if and 
only if

  \begin{equation}\label{eq:2}
  x_\alpha + x'_0 \geq n + \epsilon,
  \end{equation}
where $\epsilon := \epsilon(c,\alpha)$.

If $[c]$ and $[c']$ have a common upper bound, by the Upper Bound Lemma, 
then $x_{\alpha} \geq 2$.  Also, by definition, $x_d \geq 1$.

Consider the number $y_\alpha - \epsilon_0$.  By definition, $\vec{y}$ 
differs from $\vec{x}$ only in the $0^{th}$ and $(d_0')^{th}$ entries.  
Since $\alpha > 0$, $y_\alpha = x_\alpha$ unless $\alpha = d_0'$.  But 
unless $c$ is exceptional of Type III, $d_0' = d_0$, and when $\alpha = 
d_0'$, $\epsilon_0 = 1$.  Therefore, unless $\alpha = d_0'$ and $c$ is 
exceptional of Type III, $y_\alpha - \epsilon_0 = x_\alpha$.  Unless 
$\alpha = d$, $\epsilon = 0$, so $x_\alpha = x_\alpha - \epsilon$.  
Thus, unless $\alpha = d$, or $\alpha = d_0'$ and $c$ is exceptional of 
Type III, if $[c_0]$ and $[c']$ then $[c]$ and $[c']$ have a common 
upper bound $[s]$.

Consider the case that $c$ is exceptional of Type III.  Let $\omega = 
f(a,\vec{x})dc_1$ be a necessary $1$-form for which $c$ is necessary.  
Let $\alpha_1$ be the direction from $a$ towards the vertex over which 
$c_1$ lies.  By the Upper Bound Lemma, if $[c_0]$ and $[c']$ have an 
upper bound, then $y_\alpha - \epsilon_0 \geq 2$.  By the definition of 
$M_c$, if $\alpha \neq \alpha_1$, then $y_\alpha - \epsilon_0 \leq 1$.  
Thus, if $[c_0]$ and $[c']$ have an upper bound, then $\alpha = 
\alpha_1$.  If $[c]$ and $[c']$ also have an upper bound, then $x'_0 = 
3$.  By definition, any term $dc_0'$ of $M_cdc$ is such that $[c_0']$ 
and $[c_1]$ have an upper bound.  By the Upper Bound Lemma, we also have 
that $[c_0']$ and $[c']$ have an upper bound.  Moreover, consider the 
$1$-form $\omega' = f(a,\vec{x})dc'$.  By Restricting the Differential, 
we have that $d\omega' = M_cdc\wedge dc'$.  Thus, $[M_jdc]\cup[M_jdc'] = 
[0]$.  If $[c]$ and $[c']$ do not also have an upper bound, then $x'_0 = 
2$ and $y_\alpha = 3$.  Thus there is only one such $c_0$: when $d_0' = 
\alpha$.  By Restricting the Differential, we have that 
$d(f(a,\vec{y})dc') = dc_0\wedge dc'$, so
  $$[M_jdc]\cup [M_jdc'] = [M_cdc\wedge dc'] = [dc_0\wedge dc'] = [0].$$

Now consider the case when $c$ is not exceptional of Type III.  If 
$\alpha = d$, we have a number of subcases.

\begin{itemize}

\item If $c$ is exceptional of Type II, then $y_d = 1$.  This implies 
$x'_d \geq n - 1$, so $[c_0]$ and $[c']$ cannot have an upper bound.

\item If $c$ is not necessary and $c$ is not exceptional of Type II, 
then $M_c$ is the identity, and $c_0 = c$.  If $[c_0]$ and $[c']$ have 
an upper bound $[s]$ , then $[c] = [c_0]$ and $[c']$ have the upper 
bound $[s]$.  Let $s$ be the reduced representative of $[s]$.  If $e$ is 
not disrespectful in $s$, $c$ would be necessary for the necessary 
$1$-form $f(a,\vec{x})dc'$, contradicting that $c$ is not necessary.  
Thus, if $[c_0]$ and $[c']$ have an upper bound, $e$ must be 
disrespectful in $s$, and by the Upper Bound Lemma, $s$ is critical.

\item If $c$ is necessary and not exceptional of Type II, then let 
$\omega = f(a,\vec{x})dc_1$ be a $1$-form for which $c$ is necessary.  
Let $\alpha_1$ be the direction from $a$ towards the vertex over which 
$c_1$ lies.  By Lemma \ref{lem:necessary1form}, $0 < \alpha_1 < d$, and 
$x_{\alpha_1} \geq 2$. Since $c$ is not exceptional of Type III, $d_0' = 
d_0$.  By equations \ref{eq:1} and \ref{eq:2}, $y_\alpha - \epsilon_0 = 
x_\alpha$.  Thus, since $x'_\alpha \leq n - 2$, if $[c_0]$ and $[c']$ 
have an upper bound then $x_\alpha \geq 2 + \epsilon_0$.

  \begin{itemize}

  \item If $n = 4$, since $|\vec{x}| = n$, $x_\alpha = x_{\alpha_1} = 2$ 
  and $c$ is completely determined.  In this case, by Restricting the 
  Differential, we have $A_{c_1}(df(a,\vec{x})) = dc$, so $c_0 = c$.  But then 
  $\epsilon_0 = 1$, and $2 = x_\alpha \geq 2 + 1 = 3$.  Thus, if $n = 
  4$, then $[c_0]$ and $[c']$ cannot have an upper bound.

  \item If $n = 5$, since $|\vec{x}| = 5$ we know there is one vertex of 
  $c$ unaccounted for: there is some index $\beta$ such that $x_\beta =
  1$ when $\beta \not\in \{\alpha, \alpha_1\}$ or $x_\beta = 3$ when
  $\beta \in \{\alpha, \alpha_1\}$.  Since $c$ is not exceptional of Type
  III, $\beta \neq 0$.

    \begin{itemize}

    \item If $\beta > d$, then $c$ is of Type I, and so $c_0 = c$ by the 
    definition of $M_c$ and the lemma holds.

    \item If $\beta = d$, then the only nonzero entries of $\vec{x}$ are 
    $x_{\alpha_1} = 2$ and $x_d = 3$.  Since $\alpha \neq \alpha_1$ and 
    $x_\alpha \geq 2 + \epsilon_0$, $\alpha = d$.  By Restricting the 
    differential we have that $d\omega = dc \wedge dc_1$, so $M_cdc = dc$ 
    and $c_0 = c$.  Since $c'$ lies over a vertex in direction $\alpha =
    d$ from $a$, the lemma then follows from the Upper Bound Lemma.

    \item If $\beta < d$, then $c$ is not necessary: the cell 
    $(a,\beta,\vec{x})$ is the necessary $1$-cell for $f(a,\vec{x})dc_1$.  
    This contradicts the assumption that $c$ is necessary, so the lemma
    holds.

    \end{itemize}

  \end{itemize}

\end{itemize}

It remains to prove the lemma when $c$ is not exceptional of Type III 
and $\alpha \neq d$.  In this case, by equations \ref{eq:1} and 
\ref{eq:2}, if $[c_0]$ and $[c']$ have a common upper bound then $[c]$ 
and $[c']$ have a common upper bound $[s]$.  Let $s$ be the reduced 
representative of $[s]$.  If $e$ is disrespectful in $s$, we're done.  
Note that if $c$ is exceptional of Type II, then $e$ is disrespectful in 
$s$, by the Upper Bound Lemma.  If $e$ is respectful in $s$, then by the 
Upper Bound Lemma, $0 < \alpha < d$, $x_\alpha + y_0 = n$, and $x_i = 0$ 
for all $0 < i < d$, $i \neq \alpha$.  The lemma allows for $c$ to be 
exceptional of Type I in this case.  Consider if $c$ is not exceptional 
(as we have addressed each of the three exceptional cases).  By 
definition, any term $dc_0'$ of $M_cdc$ is such that $[c_0']$ and 
$[c_1]$ have an upper bound.  By the Upper Bound Lemma, we also have 
that $[c_0']$ and $[c']$ have an upper bound.  Moreover, consider the 
$1$-form $\omega' = f(a,\vec{x})dc'$.  By Restricting the Differential, 
we have that $d\omega' = M_cdc\wedge dc'$.  Thus, $[M_jdc]\cup[M_jdc'] = 
[0]$.  This finishes the proof.

\end{proof}

Lemma \ref{lem:effectofM} tells us that multiplication by $M$ is very 
nice in ensuring that when $Mc^*$ and $M(c')^*$ cup nontrivially, $[c]$ 
and $[c']$ have an upper bound with critical representative, 
with only one type of exception.  If $c$ is exceptional of Type I, then 
Lemma \ref{lem:effectofM} does not say much about $Mc^*$.  There is 
something to say, though:

\begin{lemma}\label{lem:multIbyM}

Let $c_1 = (a,d_1,\vec{x})$ be a critical cell of Type I and let $c_2 = 
(a,d_2,\vec{x})$ be the corresponding critical cell of Type II.  Let 
$d_0$ be the index of the remaining nonzero coordinate of $\vec{x}$.  
For any critical $1$-cell $c' = (a',e',\vec{x'})$, 

  \begin{enumerate}

  \item if $Mc_1^*\cup M(c')^* \neq [0]$ then $[c_2]$ and $[c']$ have a 
  least upper bound $[s]$, $c_2^* \cup (c')^* \neq [0]$, and $a'$ is in 
  direction $d_0$ from $a$.

  \item if $Mc_2^*\cup M(c')^* \neq [0]$ then $[c_2]$ and $[c']$ have a 
  least upper bound $[s]$, $c_2^* \cup (c')^* \neq [0]$ and $a'$ is in 
  direction $d_1$ from $a$.

  \end{enumerate}

\end{lemma}

The way to think of this is that $Mc_1^*$ is the $d_0$ `portion' of 
$c_2^*$ (and $c_1^*$), while $Mc_2^*$ is the $d_1$ `portion' of $c_2^*$ 
(and trivially $c_1^*$).

\begin{proof}

By the Upper Bound Lemma (Lemma \ref{lem:upperbound}), $[c_1]$ and 
$[c']$ have an upper bound if and only if $a'$ is in direction $d_0$ 
from $a$ and $[c_2]$ and $[c']$ have an upper bound.  Also, that $[c_2]$ 
and $[c']$ have an upper bound implies $a'$ is in either of the 
directions $d_0$ or $d_1$ from $a$.  By the Upper Bound Lemma, if 
$[c_2]$ and $[c']$ have an upper bound $[s]$, then the reduced 
representative $s$ of $[s]$ is critical, so $c_2^* \cup [dc'] \neq [0]$.

If $Mc_1^*\cup M[dc'] \neq [0]$, then by Lemma \ref{lem:effectofM}, 
$[c_1]$ and $[c']$ have an upper bound.  Thus, by above, if $Mc_1^*\cup 
M[dc'] \neq [0]$, then $c_2^* \cup [dc'] \neq [0]$.

If $Mc_2^*\cup M[dc'] \neq [0]$, then by Lemma \ref{lem:effectofM}, 
$[c_2]$ and $[c']$ have an upper bound, so by above, $c_2^* \cup [dc'] = 
[ds] \neq [0]$.  By above, $a'$ is in one of the directions $d_0$ or 
$d_1$ from $a$.  We prove, assuming $Mc_2^* \cup M[dc'] \neq [0]$, that 
$a'$ cannot be in direction $d_0$ from $a$.  By the definition of $M$ 
and the matrices $M_{c_1}$ and $M_{c_2}$, $Mdc_2 = dc_1 + dc_2$.  So, 
$Mdc_2 \wedge dc' = dc_1\wedge dc' + dc_2 \wedge dc'$.  But this is 
precisely the coboundary of the $1$-form $f(a,\vec{x})dc'$ by 
Restricting the Differential (Corollary \ref{cor:restrictingdf}).  Thus,
  $$[0] = [Mdc_2 \wedge dc'] = [Mdc_2] \cup [Mdc'] = Mc_2^* \cup M[dc'].$$
This contradicts the assumption that $Mc_2^* \cup M[dc'] \neq [0]$.  
We have proven that if $Mc_2^* \cup M[dc'] \neq [0]$ then $a'$ is in 
direction $d_1$ from $a$.  This finishes the proof.

\end{proof}

We are now ready to state and prove the final step of Theorem 
\ref{thm:cohom}.

\begin{definition}[The Simplicial Complex $\Delta$]\label{def:Delta}

Define the finite simplicial complex $\Delta$ as follows. The vertex set 
of $\Delta$ is identified with the set $\{Mc^* | c \hbox{ a critical 
$1$-cell}\}$.  A vertex $Mc^*$ is said to \emph{lie over} the vertex of
$T$ over which $c$ lies.  A set of vertices $\{Mc_1^*, Mc_2^*\}$ span a 
$1$-simplex, labelled $Mc_1^* \cup Mc_2^*$, if and only if $Mc_1^* \cup 
Mc_2^*$ is nontrivial.

\end{definition}

\begin{theorem}[The Exterior Face Algebra Structure]\label{thm:cohomiso}

We have that 
  $$H^*(B_nT) \cong \Lambda(\Delta).$$

\end{theorem}

\begin{proof}

Consider the map $\Psi:\Lambda(\Delta) \to H^*(B_nT)$ given by, for $c$ 
a critical $1$-cell, mapping the element in $\Lambda(\Delta)$ 
corresponding to the $0$-cell of $\Delta$ labelled $Mc^*$ to the 
$1$-cohomology class $Mc^*$.  Since critical $1$-cells form a free basis 
for $H^1(B_nT)$ and $M$ is a change of basis matrix, the map extends to 
all of $H^1(B_nT)$.  By the definition of $\Delta$, $\Psi$ is surjective 
onto $H^1(B_nT)$ and moreover extends linearly to a surjective 
homomorphism, since $H^1(B_nT)$ generates all of $H^*(B_nT)$ (see 
Theorem \ref{thm:cohomology}).

We claim that $\Psi$ is also injective.  Since $n \in \{4, 5\}$, it 
suffices to consider $1$-cells of $\Delta$.  A simple counting argument 
will finish the result.  For any $1$-cell $\{Mc_1^*, Mc_2^*\}$ of 
$\Delta$, by Lemma \ref{lem:effectofM} there is a critical $2$-cell 
$s$ such that $[s]$ is the least upper bound of $[c_1]$ and $[c_2]$, 
\emph{unless} possibly one of $c_1$ or $c_2$ is exceptional of type I.  
By Theorem \ref{thm:cohomology}, $s$ uniquely determines and is uniquely 
determined by $c_1$ and $c_2$.  By Lemma \ref{lem:multIbyM}, even when 
one of $c_1$ or $c_2$ is exceptional of Type I, there is still a 
uniquely determined critical $2$-cell $s$ corresponding to $c_1$ and 
$c_2$.  Thus, there are at most as many edges of $\Delta$ as there are 
critical $2$-cells.  Since $\Psi$ is surjective and by Theorem 
\ref{thm:cohomology} $H^2(B_nT)$ has rank equal to the number of 
critical $2$-cells, $\Psi$ must also be injective.  Thus, $\Psi$ is 
the desired isomorphism.

\end{proof}

Theorem \ref{thm:cohomiso} gives us an important strengthening of 
Lemmas \ref{lem:effectofM} and \ref{lem:multIbyM}:

\begin{corollary}[Multiplication by $M$]\label{cor:multbyM}

Let $c = (a,d,\vec{x})$ and $c'=(a',d',\vec{x'})$ be critical $1$-cells 
with $c \leq_r c'$.  Then
  $$Mc^*\cup M(c')^* \neq [0]$$ 
if and only if $[c]$ and $[c']$ have a least upper bound $[s]$ with 
reduced representative $s$ and either

  \begin{enumerate}

  \item $c$ is not exceptional of Types I or II and $s$ is critical,

  \item $c$ is exceptional of Type I and $s$ is not critical, or

  \item $c$ is exceptional of Type II, $s$ is critical, and the
  direction $\alpha$ from $a$ to $a'$ is not the smallest direction for
  which $x_\alpha \neq 0$.

  \end{enumerate}

\end{corollary}

The results of this section show that exterior face algebra structures 
do crop up for cohomology rings of tree braid groups, but restrictive 
assumptions were made on the tree braid groups in question.  We have 
shown that if $n < 6$ or $T$ is linear, then $H^*(B_nT)$ is an exterior 
face algebra.  We conjecture the converse holds:

\begin{conjecture}[Arbitrary Number of Strands]\label{conj:extcohom}

Let $T$ be a tree.  Then $H^*(B_nT)$ is an exterior face algebra if and 
only if $n < 6$ or $T$ is linear.

\end{conjecture}

Our conjecture is based on the reasoning that, for $T$ a nonlinear tree 
and $n \geq 6$, we may find a particular structure in $H^*(B_nT)$ which 
we believe prevents $H^*(B_nT)$ from being an exterior face algebra.  We 
do not go into the details here, but we claim if we apply the change of 
basis techniques of this section to $H^*(B_nT)$ for $n \geq 6$, the 
result is that $H^*(B_nT)$ is isomorphic to some particular quotient of 
an exterior face algebra.  That quotient has as a subalgebra the 
quotient pictured in Figure \ref{fig:W}.  We believe that this 
subalgebra is `poisonous': any algebra containing the given subalgebra 
in a `nice' way will not be an exterior face algebra. Of course, we do 
not know the appropriate definition of `nice', and showing that an 
algebra is \emph{not} an exterior face algebra is very difficult, even 
in this restricted setting.

  \begin{figure}
    \begin{center}
    \input{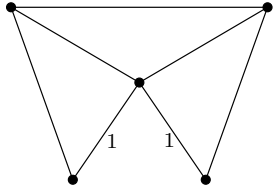}
    \end{center}

  \caption{Shown is a simplicial complex.  The $1$s labelling two 
  $1$-cells of the simplicial complex indicate that in the corresponding 
  exterior face algebra, we quotient by identifying the two edges. We 
  believe that, if this quotient is a subalgebra of a given algebra in a 
  nice way, then the given algebra cannot be an exterior face algebra.  
  We also believe that the cohomology algebra of tree braid groups for 
  nonlinear trees and at least 6 strands contain this quotient as a 
  subalgebra in a nice way.}

  \label{fig:W}

  \end{figure}

\section{Rigidity}\label{sec:rigidity}

In this section, our goal is to prove Theorem \ref{thm:45rigidity}, the 
Rigidity of $4$ and $5$ Strand Tree Braid Groups.  We do so by 
explicitly reconstructing the tree $T$ from the corresponding tree braid 
group $B_nT$.  The tree $T_{\Delta}$, defined in the first subsection, 
will be the reconstruction desired, as proven in the second subsection.

\subsection{The Tree $T_{\Delta}$}\label{sec:T_Delta}

In this subsection we try to construct a tree $T_{\Delta}$ from an 
arbitrary $1$-dimensional simplicial complex $\Delta$.  The tree 
$T_{\Delta}$ will be used in Section \ref{sec:45strands} to prove 
Theorem \ref{thm:45rigidity}.  In Theorem \ref{thm:45rigidity}, we will 
show that, if $\Delta$ is the simplicial complex giving the exterior 
face structure for $H^*(B_nT)$ for some tree $T$ and $n = 4$ or $5$, 
then $T_{\Delta}$ is homeomorphic to $T$.  By Theorem \ref{thm:cohom}, 
for $n = 4$ or $5$ there exists a unique $1$-dimensional simplicial 
complex associated to $H^*(B_nT)$ - this is why we restrict $\Delta$ to 
be $1$-dimensional.  Even with the dimension restriction, for many 
$\Delta$, $T_{\Delta}$ will not exist; it will turn out that 
$T_{\Delta}$ will exist exactly when we want it to (see Corollary 
\ref{cor:T_Delta}).

Throughout this section, refer to Example \ref{ex:treeTDelta} on page 
\pageref{ex:treeTDelta} for an example of the concepts defined.  Also 
see Example \ref{ex:5example} on page \pageref{ex:5example} for further 
reference.

Let $\Delta$ be an arbitrary 1-dimensional simplicial complex, and fix a 
constant $n \in \{4, 5\}$. For each vertex $v \in \Delta$, let $N_v$ 
denote the \emph{vertex neighborhood} of $v$ in $\Delta$: the set of all 
vertices $v' \in \Delta$ adjacent to $v$. Let $\leqn$ denote the 
preorder on vertices of $\Delta$ induced by vertex neighborhood 
inclusions, so $v \leqn v'$ if and only if $N_v \subseteq N_{v'}$.  Most 
pairs of vertices will not be comparable, but enough will be.

The preorder $\leqn$ induces an equivalence relation $\en$ on vertices 
of $\Delta$: if $v \leqn v'$ and $v' \leqn v$, we write $v \en v'$.  
Note that $\leqn$ induces a partial order on the vertex set $V(\Delta)$ 
of $\Delta$ modulo equivalence classes under $\en$.  We write the 
($\en$)-equivalence class of $v \in V(\Delta)$ as $[v]$.  For an 
($\en$)-equivalence class $[v]$, the \emph{descendants} of $[v]$ are 
those ($\en$)-equivalence classes $[v']$ such that $[v'] \leq [v]$ (i.e. 
$N_{v'} \subseteq N_v$) and $N_{v'} \neq \emptyset$.  A descendant 
$[v']$ of $[v]$ is a \emph{child} of $[v]$ and $[v]$ is a \emph{parent} 
of $[v']$ if there are no descendants of $[v]$ between $[v']$ and $[v]$ 
with respect to $\leqn$ - i.e. there does not exist a $w$ such that 
$N_{v'} \subsetneq N_w \subsetneq N_v$.

We now define a tree $T_{\Delta}$ from the simplicial complex $\Delta$.  
Let $|\Delta|$ denote the number of vertices of $\Delta$, and for any 
subset $S$ of $V(\Delta)$, let $|S|$ denote its size.

Recall the function $Y_m(x)$, defined in Theorem \ref{thm:radial}.  We 
will use $Y_m(x)$ repeatedly to define $T_{\Delta}$.

Let $v_0$ be a vertex of $\Delta$ which is maximal under $\leqn$.  Thus, 
$N_{v_0}$ is maximal under inclusion among all vertex neighborhoods of 
vertices of $\Delta$.

If $N_{v_0}$ is empty, define the tree $T_{\Delta}$ to be a radial tree 
(see Theorem \ref{thm:radial}) whose essential vertex $a$ has degree 
$deg(a)$ such that
  $$Y_n(deg(a)) = |\Delta|$$
if such a number exists.  If such a number does not exist, then 
$T_{\Delta}$ is undefined.

Now assume that $N_{v_0}$ is nonempty.  Define the finite undirected 
graph $H'$ to be the \emph{neighborhood heirarchy} of $[v_0]$, as 
follows.  Create a vertex $p_1$ in $H'$, and for each descendant $[v]$ 
of $[v_0]$ create a vertex $p_{[v]}$ in $H'$, so that the vertex set of 
$H'$ is

  $$\{p_1\} \cup \{p_{[v]} | [v]\hbox{ is a descendant of }[v_0]\}.$$
Every vertex $p_{[v]}$ is connected to $p_{[v']}$ for each child $[v']$ 
of $[v]$.  Also, there is one edge connecting $p_1$ to $p_{[v_0]}$ 
(when $\Delta$ is as in Theorem \ref{thm:cohomiso}, think of $p_1$ as 
corresponding to the cochain $1$, so that every $1$-cochain is in its 
`neighborhood').

If $n = 4$, the neighborhood heirarchy $H'$ is already almost the tree 
$T_{\Delta}$, but if $n = 5$, the $H'$ contains too many vertices.  To 
address this issue, we define a subgraph $H$ of the graph $H'$. If $n = 
4$, define the finite undirected graph $H$ to be $H'$.  If $n = 5$, 
consider if there exists a child $[v_0']$ of $[v_0]$ such that

  \begin{enumerate}

  \item $[v_0']$ has exactly half as many descendants as $[v_0]$, and

  \item if $[v]$ and $[v']$ are children of $[v_0]$ with a common 
  descendant, then one of $[v]$ or $[v']$ is a descendant of $[v_0']$.

  \end{enumerate}
If such a child $[v_0']$ exists, define $H$ to be $H'$ but with the 
descendants of $[v_0']$ removed from $H'$ - that is, $H$ is $H'$ minus 
the vertices corresponding to descendants of $[v_0']$ and any edges 
connecting them to any other vertices.  If no such child exists, $H$ is 
undefined.

We want to construct a tree $T_{\Delta}$ by `growing' $T_{\Delta}$ from 
the graph $H$.  To do so, we want to associate to each vertex $p_{[v]}$ 
of $H$ a number $pdeg_{[v]}$, which will be the desired degree of 
$p_{[v]}$.  We say $p_{[v]}$ is a \emph{leaf} of $H$ if $[v]$ has no 
children.  If $p_{[v]}$ is a \emph{leaf} of $H$, then define 
$pdeg_{[v]}$ to be the positive integer greater than 2 satisfying

  $$Y_{(n-2)}(pdeg_{[v]}) = |N_v|$$
for any $v \in [v]$; if such a number does not exist, $pdeg_{[v]}$ is 
undefined.  If $p_{[v]}$ is not a leaf of $H$, then define $pdeg_{[v]}$ 
to be the positive integer greater than $2$ satisfying

  $$Y_2(pdeg_{[v]}) = |[v']|$$
for every child $[v']$ of $[v]$; if such a number does not exist, 
$pdeg_{[v]}$ is undefined.  Also define $pdeg_1$ to be the positive 
integer greater than $2$ satisfying 

  $$Y_2(pdeg_1) = |[v_0]|;$$
if such a number does not exist, $pdeg_1$ is undefined.

If $pdeg_{[v]}$ exists, it is unique, since $Y_2$ and $Y_{(n-2)}$ are 
monotonic increasing on values greater than 2.

We may now define the $T_{\Delta}$.  The tree $T_{\Delta}$ is undefined 
if any of the following hold:

  \begin{enumerate}

  \item $H$ is undefined;

  \item $H$ is not a tree;

  \item there exists $p_{[v]}$ in $H$ such that $pdeg_{[v]}$ is 
  undefined, or $pdeg_1$ is undefined;

  \item there exists $p_{[v]}$ in $H$ which is not a leaf and $[v]$ has 
  more than $d_{[v]}-1$ children.

  \end{enumerate}
If we are not in one of the cases where $T_{\Delta}$ is undefined, then 
$T_{\Delta}$ is the tree $H$ together with some edges added for the 
purpose of increasing the degrees of the vertices $p_{[v]}$.  Each 
distinct added edge connects a vertex $p_{[v]}$ to a distinct vertex of 
degree $1$ in $T_{\Delta}$.  We add edges so that the degree of 
$p_{[v]}$ in $T_{\Delta}$ is exactly $pdeg_{[v]}$.  The degree of $p_1$ 
in $T_{\Delta}$ is $pdeg_1$.

There is one exceptional case when defining $T_{\Delta}$.  If $n = 5$ 
and $H$ has exactly three vertices, then the graph $H'$ had exactly five 
vertices.  Thus, there were exactly two children $[v_0']$ and 
$[\hat{v}_0']$ which had exactly half the number of descendants of 
$[v_0]$.  The vertices of $H'$ are then exactly $p_1$, $p_{[v_0]}$, 
$p_{[v_0']}$, $p_{[\hat{v}_0']}$, and some vertex $p_{[w]}$.  Let $a$, 
$b$, and $c$ be positive integers greater than 2 such that
  $$Y_2(a) = |[v_0]|, \qquad Y_3(b) - Y_2(b) = |[w]|, \qquad Y_2(c) = 
  |N_w|,$$
if such integers exist.  If any of $a$, $b$, or $c$ is undefined, then 
$T_{\Delta}$ is undefined.  If $a$, $b$, and $c$ are all defined, then 
define $T_{\Delta}$ to be the linear tree having no vertices of degree 
$2$ and exactly three essential vertices $A$, $B$, and $C$, where the 
essential vertices have degrees $a$, $b$, and $c$, respectively, and $A$ 
and $C$ are extremal.

We note that, according to the definition presented here, the tree 
$T_\Delta$ may depend on the choice of $\leqn$-maximal vertex $v_0$.  
As a consequence of Theorem \ref{thm:45rigidity}, we will have that,
when $\Delta$ is the unique simplicial complex associated to $H^*(B_nT)$ 
for $T$ a tree and $n = 4$ or $5$, the tree $T_\Delta$ is independent of 
$v_0$.

  \begin{figure}[!h]
    \begin{center}
    \input{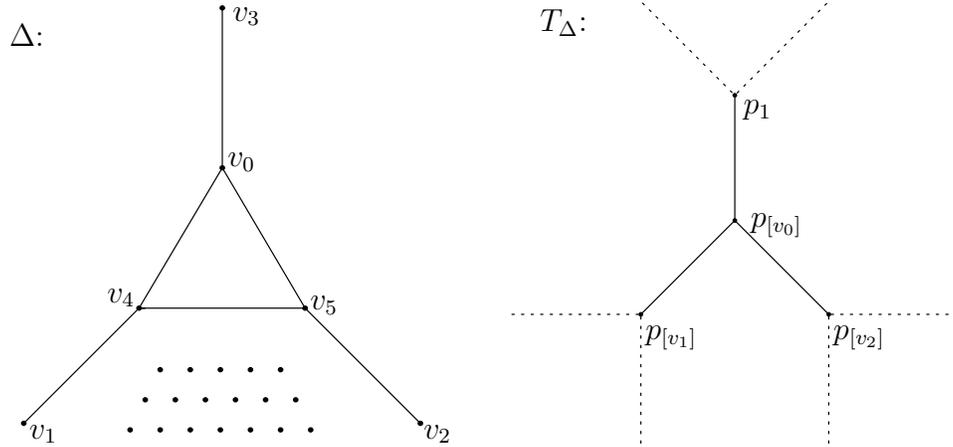}
    \end{center}

  \caption{An example calculation of the tree $T_{\Delta}$.}

  \label{fig:treeTDelta}
  \end{figure}

\begin{example}\label{ex:treeTDelta}

Consider the simplicial complex $\Delta$ shown in Figure 
\ref{fig:Deltamin}, and reproduced in Figure \ref{fig:treeTDelta}.  We 
have labelled the vertices of $\Delta$ with nontrivial neighborhoods.  
With the vertices of $\Delta$ as labelled in Figure 
\ref{fig:treeTDelta}, the vertex $v_0$ is maximal under $\leqn$ 
($N_{v_0} = \{v_1,v_4,v_5\}$).  Since $N_{v_1} = \{v_4\}$ and $N_{v_2} = 
\{v_5\}$, $[v_1]$ and $[v_2]$ are descendants of $[v_0]$.  The 
neighborhood heirarchy of $[v_0]$ is shown on the right of Figure 
\ref{fig:treeTDelta} with solid edges.  Fixing $n$ to be $4$, the tree 
$H$ constructed above is the neighborhood heirarchy.  To grow 
$T_{\Delta}$ from $H$, we compute the desired degrees in $T_{\Delta}$ of 
the vertices of $H$.  The desired degree of $p_1$ is $3$, since $Y_2(3) 
= 1 = |[v_0]|$. The desired degree of $p_{[v_0]}$ is $pdeg_{[v_0]} = 3$, 
since $p_{[v_0]}$ is not a leaf of $H$ and $Y_2(3) = 1 = |[v_1]| = 
|[v_2]|$.  The desired degree of $p_{[v_1]}$ is $pdeg_{[v_1]} = 3$, 
since $p_{[v_1]}$ is a leaf of $H$ and $Y_{(4-2)}(3) = 1 = |N_{v_1}|$. 
Finally, the desired degree of $p_{[v_2]}$ is $pdeg_{[v_2]} = 3$, since 
$p_{[v_2]}$ is a leaf of $H$ and $Y_{(4-2)}(3) = 1 = |N_{v_1}|$.  Thus, 
the tree $T_{\Delta}$ is as depicted on the right of Figure 
\ref{fig:treeTDelta}.

\end{example}

\subsection{The Four and Five Strand Case}\label{sec:45strands}

We now wish to prove the Rigidity for Tree Braid Groups in the 
restricted setting of $n = 4$ or $5$.  As usual, we need a series of 
definitions and lemmas to complete the proof.  These lemmas will 
completely describe, for a given critical cell $c$, which critical cells 
$c'$ are such that $Mc^*\cup M(c')^* \neq [0]$.  It is this information 
that we use to prove Theorem \ref{thm:45rigidity}.  All lemmas in this 
subsection will assume $n \in \{4, 5\}$.

For any critical $1$-cell $c$, define the \emph{$M$-cup neighborhood} 
$N_c$ of $c$ to be the set
  $$N_c := \{c_1 | Mc^* \cup Mc_1^* \neq [0]\}.$$
When $\Delta$ is as in Theorem \ref{thm:cohomiso} and $v = Mc^*$ is a 
vertex of $\Delta$, the vertex neighborhood $N_v$ of $v$ in $\Delta$ is 
by definition in bijective correspondence with the $M$-cup neighborhood 
of $c$.

\begin{lemma}[The Direction $d_c$]\label{lem:directiondc}

Let $c = (a,e,\vec{x})$ be a critical $1$-cell. Either the $M$-cup 
neighborhood $N_c$ of $c$ is empty or there exists a unique direction 
$d$ from $a$ such that, for every $c_1 \in N_c$, $c_1$ lies in direction 
$d_c$ from $a$.  

\end{lemma}

The direction $d_c$ is called the \emph{CUB (Critical Upper Bound) 
direction} for $c$.

\begin{proof} 

This lemma follows from repeated application of the Upper Bound Lemma, 
Lemma \ref{lem:upperbound}.  Throughout this proof, we use the Upper 
Bound Lemma, usually without reference.

Assume the converse: that $N_c$ is nonempty and there are critical 
$1$-cells $c_1, c_2 \in N_c$ such that the directions $d_1$ and $d_2$ 
from $a$ toward $c_1$ and $c_2$, respectively, are distinct.  By 
definition, we know $x_d \geq 1$.  We have that $x_{d_1}, x_{d_2} \geq 
2$.  Also, if $d_1 = d$ then $x_{d_1} \geq 3$, while if $d_2 = d$, 
$x_{d_2} \geq 3$.  If $x_{d_1} \geq 3$ and $d_2 \neq d$, then $x_{d_2} 
\leq 1$, a contradiction. If $x_{d_2} \geq 3$ while instead $d_2 = d$, 
then $x_{d_1} \leq 2$, again a contradiction.  Thus, $x_{d_1} < 3$.  
Since $2 \leq x_{d_1} < 3$, $x_{d_1} = 2$, and so $d_1 \neq d$.  
Similarly, $x_{d_2} = 2$ and $d_2 \neq d$.

Consider if $d_1 = 0$.  Then $c$ is not exceptional of Type I. Let 
$[s_2]$ be the least upper bound of $[c]$ and $[c_2]$, and let $s_2$ be 
the reduced representative of $[s_2]$.  By Corollary \ref{cor:multbyM}, 
$s_2$ exists and is critical.  But by the Upper Bound Lemma, $e$ is 
respectful in $s_2$, a contradiction.  Thus, $d_1 \neq 0$.  Similarly, 
$d_2 \neq 0$.

If $d_1 > d$, note $c$ is not exceptional of Type I.  Let $[s_2]$ be the 
least upper bound of $[c]$ and $[c_2]$, and let $s_2$ be the reduced 
representative of $[s_2]$.  By Corollary \ref{cor:multbyM}, $s_2$ exists 
and is critical.  but by the Upper Bound Lemma, $e$ is respectful in 
$s_2$, a contradiction.  Thus, $d_1 < d$.  Similarly, $d_2 < d$.

The only possibility left is that $c$ is exceptional of Type II.  The 
existence of both $c_1$ and $c_2$ contradicts Lemma \ref{lem:multIbyM}.  
This finishes the proof.

\end{proof}

Consider a critical $1$-cell $c = (a,d,\vec{x})$ with nonempty $M$-cup 
neighborhood $N_c$, and $d_c$ as in Lemma \ref{lem:directiondc}.  Just 
as the CUB direction $d_c$ helps determine $N_c$, there is a number, 
$CUB(c)$, that we may associate to $c$ to help determine $N_c$.  Before 
defining $CUB(c)$, we need the following definition. Let 
$\epsilon_c(d')$ be the \emph{cup constant} of $c$ in direction $d'$, 
associated to $c$ as follows.  If $d' = 0$ or $d' > d$, define 
$\epsilon_c(d') := 0$.  If $0 < d' < d$ and there is some index $i$ with 
$0 < i < d$ and $i \neq d'$ with $x_i > 0$, define $\epsilon_c(d') := 
0$.  If $n = 5$ and $c$ is exceptional of Type I, define $\epsilon_c(d') 
:= 0$.  Otherwise, define $\epsilon_c(d') := 1$. Note that, unless $n = 
5$ and $c$ is exceptional of Type I, if the upper bound constant 
$\epsilon(c,d')$ for $c$ in direction $d'$ is $1$, then the cup constant 
$\epsilon_c(d')$ for $c$ in direction $d'$ is also $1$. For convenience, 
we let $\epsilon_c := \epsilon_c(d_c)$. Then, the number $x_{d_c} - 
\epsilon_c$ is denoted $CUB(c)$, and is the \emph{CUB (Critical Upper 
Bound) number} for $c$.

\begin{lemma}[A Bound on $CUB(c)$]\label{lem:CUBbound}

For a critical $1$-cell $c$ with nonempty $M$-cup neighborhood, $2 \leq 
CUB(c) \leq n - 2$.

\end{lemma}

\begin{proof}

If $c$ is exceptional of Type I, then $CUB(c) = 2$ and there is nothing 
to prove.  So, assume $c$ is not exceptional of Type I. Let $c = 
(a,d,\vec{x})$.  By the Upper Bound Lemma, $x_{d_c} \geq 2 + 
\epsilon(c,d_c)$.  Since $c$ has non-empty $M$-cup neighborhood, this 
inequality is strict if there exists no index $i \neq d_c$ such that $0 
< i < d$ and $x_i > 0$, by Corollary \ref{cor:multbyM}.  But by the 
definition of $\epsilon_c$, this means $x_{d_c} \geq 2 + \epsilon_c$. 
Thus,

  $$CUB(c) = x_{d_c} - \epsilon_c \geq 2.$$

Since $c$ is critical, there exist at least two nonzero coordinates of 
$\vec{x}$, so $x_{d_c} \leq n - 1$.  If $x_{d_c} < n - 1$, there is 
nothing to prove, so consider if $x_{d_c} = n - 1$.  If $d_c = 0$ 
or $d_c > d$, then $c$ is not critical, so it must be that $0 < d_c \leq 
d$. If $d_c = d$, then $\epsilon_c = 1$.  If $0 < d_c < d$, then there 
exists no index $j\neq d_c$ with $0 < j < d$ and $x_j > 0$, so again 
$\epsilon_c = 1$.  Either way,
  $$CUB(c) = x_{d_c} - \epsilon_c \leq n-1-1 = n - 2.$$

\end{proof}

The CUB direction and number determine $N_c$ in the following sense:

\begin{lemma}[The Structure of Neighborhoods]\label{lem:Ncstructure}

Let $c = (a,e,\vec{x})$ and $c' = (a',e',\vec{x'})$ be critical 
$1$-cells.  Then $c \in N_{c'}$ and $c' \in N_c$ if and only if:

  \begin{enumerate}

  \item $a \neq a'$,

  \item $a'$ lies in direction $d_c$ from $a$,

  \item $a$ lies in direction $d_{c'}$ from $a'$, and

  \item $CUB(c) + CUB(c') \geq n$.

  \end{enumerate}

\end{lemma}

\begin{proof}

Note $c \in N_{c'}$ if and only if $c' \in N_c$ by definition.  If $c 
\in N_{c'}$ and $c' \in N_c$, then the first three conditions follow 
from Lemma \ref{lem:directiondc} and the fourth condition from the Upper 
Bound Lemma, Lemma \ref{lem:upperbound}.

Now consider if the four conditions hold.  That $[c]$ and $[c']$ have an 
upper bound $[s]$ follows from the Upper Bound Lemma.  Let $s$ be the 
reduced representative of $[s]$.  We claim that either $s$ is critical 
or the $\leqn$-smaller of $c$ and $c'$ is exceptional of Type I.  If the 
claim holds, then by Corollary \ref{cor:multbyM}, $c \in N_{c'}$ and $c' 
\in N_c$. If $s$ is not critical, then one of $e$ or $e'$ is respectful 
in $s$.  Without loss of generality, assume $e$ is respectful in $s$.  
By the Upper Bound Lemma, $d_{c'} = 0$ and $d_c$ satisfies $0 < d_c < 
dir$, where $dir$ is the direction from $a$ along $e$.  Futhermore, $x_i 
= 0$ for every $i \neq d_c$ satisfying $0 < i < dir$, and $x_{d_c} + 
x'_0 = n$.  By definition, $\epsilon_{c'}(d_{c'}) = 0$.  Since
  $$n - \epsilon_c = (x_{d_c} - \epsilon_c) + (x'_0 - \epsilon_c') = 
  CUB(c) + CUB(c') \geq n,$$
it follows that $\epsilon_c = 0$.  By definition, it must be that $c$ is 
exceptional of Type I.  This proves the claim.

\end{proof}

\begin{corollary}[Neighborhood Containment]\label{cor:subequalneighborhoods}

Let $c = (a,d,\vec{x})$ and $c' = (a',d',\vec{x'}$ be critical $1$-cells 
with nonempty $M$-cup neighborhoods.  Then $N_{c'} \subseteq N_c$ if and 
only if $CUB(c') \leq CUB(c)$ and either:

  \begin{enumerate}

  \item $a = a'$ and $d_c = d_{c'}$, or

  \item $a \neq a'$ and:

    \begin{enumerate}

    \item $a'$ is in direction $d_c$ from $a$ and

    \item $a$ is not in direction $d_{c'}$ from $a'$.

    \end{enumerate}

  \end{enumerate}

\end{corollary}

\begin{proof}

These conditions are equivalent to the statement ``if $c_0 \in N_{c'}$ 
then $c_0 \in N_c$", by Lemma \ref{lem:Ncstructure}.

\end{proof}

\begin{corollary}[Neighborhood Equality]\label{cor:equalneighborhoods}

Let $c = (a,d,\vec{x})$ and $c' = (a',d', \vec{x'})$ be critical 
$1$-cells with nonempty $M$-cup neighborhoods.  Then $N_c = N_{c'}$ if 
and only if 

  \begin{enumerate}

  \item $a = a'$,

  \item $d_c = d_{c'}$, and

  \item $CUB(c) = CUB(c')$.\qed

  \end{enumerate}

\end{corollary}

\begin{proof}

That $N_c = N_{c'}$ is equivalent to $N_c \subseteq N_{c'}$ and $N_{c'} 
\subseteq N_c$.  The result then follows from applying Corollary 
\ref{cor:subequalneighborhoods} twice.

\end{proof}

This lemma and its two corollaries almost completely describe the 
structure of the $M$-cup neighborhoods as well as their behavior under 
inclusion.  The final two ingredients for the proof of Theorem 
\ref{thm:45rigidity} are to count exactly how many critical $1$-cells 
have a given neighborhood, and describe what kinds of critical $1$-cells 
have neighborhoods which are maximal under inclusion.  We accomplish 
these tasks in the following two results.

\begin{lemma}[Counting]\label{lem:counting}

For any essential vertex $a$, there exist exactly $Y_n(deg(a))$ critical 
$1$-cells lying over $a$.  Let $d$ be a direction from $a$ such that 
there exists some essential vertex besides $a$ itself in direction $d$ 
from $a$.  Let $k$ be an integer satisfying $2 \leq k \leq n-2$.  Then 
there exists exactly $Y_{n-k}(deg(a))$ critical $1$-cells lying over $a$ 
with CUB direction $d$ and CUB number at least $k$.

\end{lemma}

\begin{proof}

By Definition \ref{def:reducednotation} and the preceding discussion, a 
critical $1$-cell $c$ is determined by: an essential vertex $a$, an edge 
$e$ with terminal endpoint $a$, an $a$-vector $\vec{x}$, and the 
condition that at least one vertex of $c$ be blocked by $e$ at $a$.  All 
of these properties are local properties, and do not depend on $T$ or on 
the existence or nonexistence of any other essential vertices in $T$.  
Thus, there is a bijection between critical $1$-cells lying over $a$ and 
critical $1$-cells in a radial tree whose essential vertex has degree 
$deg(a)$.  The first statement of the lemma then follows from the Radial 
Rank Theorem, Theorem \ref{thm:radial}.

We now prove the second statement of the lemma.  If $c = 
(a,dir,\vec{x})$ is a critical $1$-cell lying over $a$ with $d_c = d$ 
and $CUB(c) \geq k$, then $x_d \geq k + \epsilon_c$.  Thus, $c$ has at 
least $k$ vertices (not including the edge $e$) in direction $d$ from 
$a$.  Let $c'$ be the reduced $1$-cell corresponding to deleting $k$ 
vertices from $c$ in direction $d$ from $a$.  That is, $c' := 
(a,dir,\vec{x'})$, where $\vec{x'}$ equals $\vec{x}$ in all but the 
$d^{th}$ coordinate, and $x'_d = x_d - k$.  We make one exception: if 
$c$ is exceptional of Type I, define $c' := (a,dir',\vec{x'})$, where 
$dir' > dir$ is the direction towards the edge of the exceptional 
$1$-cell of Type II associated to $c$.

We claim $c'$ is critical.  For, if $c'$ is not critical, then it must 
be that $0 < d < dir$, $x_d = k$, and $x_i = 0$ for all $0 < i < dir$, 
$i \neq d$.  Since $x_d \geq k + \epsilon_c$, it must be that 
$\epsilon_c = 0$.  By the definition of $\epsilon_c$, it follows that 
$c$ is exceptional of Type I.  But when $c$ is exceptional of Type I, we 
changed $c'$ to be critical.

Now consider a critical $1$-cell $c_0'$ lying over $a$ on $n-k$ strands.  
Let $c_0$ be the reduced $1$-cell on $n$ strands corresponding to adding 
$k$ vertices to $c_0'$ in the direction $d$ from $a$.  As $c_0'$ is 
critical, so is $c_0$.  It is straightforward to construct a critical 
$1$-cell $c_1$ lying over a vertex in direction $d$ from $a$ such that 
$[c_0]$ and $[c_1]$ have a least upper bound $[s]$ whose reduced 
representative $s$ is critical.  We leave this to the reader.  By Lemma 
\ref{lem:directiondc} and Corollary \ref{cor:multbyM}, unless $c_0$ is 
exceptional of Type II and $d$ is the smallest direction from $a$ in 
which $c_0$ has strands, $d_{c_0} = d$.  It then follows that $CUB(c_0) 
\geq k$.  Furthermore, since $c_0'$ is critical, $c_0$ cannot be of Type 
I.  If $c_0$ is exceptional of Type II and $d$ is the smallest direction 
from $a$ in which $c_0$ has strands, we make an exception and replace 
$c_0$ with the corresponding exceptional cell of Type I.

These two constructions (constructing $c'$ from $c$ and $c_0$ from 
$c_0'$) are inverses of each other (we leave it to the reader to check 
this when $c$ is of Type I or $c_0'$ is of Type II).  Thus, they give a 
set bijection between `critical $1$-cells $c$ with $d_c = d$ and $CUB(c) 
\geq k$' and `critical $1$-cells lying over $a$ with $n-k$ strands'.  
In particular, these two sets have the same number of elements.  The 
latter set has $Y_{n-k}(deg(a))$ elements, by Theorem \ref{thm:radial}.  
This proves the lemma.

\end{proof}

When $\Delta$ is as in Theorem \ref{thm:cohomiso} and $v = Mc^*$ is a 
vertex of $\Delta$ where $[v]$ is maximal under $\leqn$, then both the 
vertex neighborhood $N_v$ and the $M$-cup neighborhood $N_c$ will be 
maximal under inclusion.  Furthermore:

\begin{corollary}[Maximal Neighborhoods]\label{cor:extremal}

Let $c$ be a critical $1$-cell.  If the $M$-cup neighborhood $N_c$ of 
$c$ is maximal under inclusion among all $M$-cup neighborhoods, then $c$ 
lies over an extremal vertex $a$ of $T$, and $CUB(c) = n - 2$.

\end{corollary}

\begin{proof}

Let $a$ be the essential vertex over which $c$ lies.  If $T$ is radial, 
then the only essential vertex is $a$, so $a$ is extremal.  If $T$ is 
not radial, then there exist non-empty $M$-cup neighborhoods, so the 
empty set is not maximal under inclusion and therefore $N_c$ is 
nonempty.  Assume that $a$ is not extremal, and let $d_c$ be the 
direction prescribed by Lem \ref{lem:directiondc}.  Since $a$ is not 
extremal, there exist extremal vertices in at least 2 directions from 
$a$.  Thus there exists an extremal vertex $a'$ in a direction which is 
not $d_c$ from $a$.  Let $d'$ be the direction from $a'$ to $a$.  By 
Lemma \ref{lem:counting}, there exists a critical $1$-cell $c'$ lying 
over $a'$ with $d_{c'} = d'$ and $CUB(c') = CUB(c)$.  By Corollaries 
\ref{cor:subequalneighborhoods} and \ref{cor:equalneighborhoods}, $N_c 
\subsetneq N_{c'}$, contradicting the maximality of $N_c$.  Thus $a$ is 
extremal.

That $CUB(c) = n-2$ is similar: assume $CUB(c) < n-2$.  By Lemma 
\ref{lem:counting}, there exists a critical $1$-cell $c'$ lying over $a$ 
with $d_{c'} = d_c$ and $CUB(c') = n - 2$.  By Corollaries 
\ref{cor:subequalneighborhoods} and \ref{cor:equalneighborhoods}, $N_c 
\subsetneq N_{c'}$, contradicting the maximality of $N_c$.  Thus $CUB(c) 
= n - 2$. 

\end{proof}

\medskip

We are now ready to prove the following rigidity result for $n = 4$ and 
$5$ strand tree braid groups:

\begin{theorem}[Rigidity]\label{thm:45rigidity}

Let $T$ and $T'$ be two finite trees, and let $n = 4$ or $5$.  
The tree braid groups $B_nT$ and $B_nT'$ are isomorphic as groups if and 
only if the trees $T$ and $T'$ are homeomorphic as trees.

\end{theorem}

There are many equivalent formulations of Theorem \ref{thm:45rigidity}.  
Some of these restatements include:

\begin{itemize}

\item there is a bijection between finite trees up to homeomorphism and 
$n$-strand tree braid groups up to isomorphism, $n = 4$ or $5$;

\item there is a bijection between $n$-strand tree braid groups up to 
isomorphism, $n = 4$ or $5$, and the set of cohomology rings of $n$ 
strand tree braid groups (that this is a reformulation follows from 
Theorem \ref{thm:cohom});

\item Given an $n$-strand tree braid group $G$ with $n = 4$ or $5$, 
there exists a unique tree $T$ (up to homeomorphism) such that $G = 
B_nT$.

\end{itemize} 
To prove Theorem \ref{thm:45rigidity}, we actually show the last 
equivalent statement.

Note that, in each homeomorphism class of trees, there exists exactly 
one tree that has no vertices of degree two.  For this tree, the 
vertices are exactly all of the nonmanifold points.  The following proof 
involves construction of a tree with no vertices of degree two, and 
graph braid groups on this tree.  As the fundamental group of the 
associated configuration spaces is invariant under homeomorphism, this 
does not affect the associated tree braid group.  The reader is 
cautioned, however, that in order to utilize critical cells via 
discretized configuration spaces and Theorem \ref{thm:Abrams}, further 
subdivision of this tree is required.

For applications of the proof of Theorem \ref{thm:45rigidity}, we refer 
the reader to Examples \ref{ex:exTmin} and \ref{ex:5example}, 
immediately following the proof.

\begin{proof}

Let $Trees$ denote the set of all trees up to homeomorphism.  Let $B_n$ 
be the map taking a tree $T$ to its $n$ strand tree braid group $B_nT$. 
By definition, $B_n$ is well-defined.  We prove $B_n: Trees \to 
B_n(Trees)$ is bijective for $n = 4, 5$ by constructing an inverse map.  
Let $\Psi_n: B_n(Trees) \to Trees$ be given by $\Psi_n(B_nT) = 
T_{\Delta}$, where $\Delta$ is the simplicial complex uniquely 
determined by $H^*(B_nT)$, as in Theorem \ref{thm:cohom}, and 
$T_{\Delta}$ is as in Section \ref{sec:T_Delta}.  We claim $\Psi_n$ is 
well-defined and that $\Psi_n \circ B_n$ is the identity map for $n = 4, 
5$.  Our goal is to analyze $\Psi_n$ so that both properties will be 
proven simultaneously. This would prove the desired result, as $B_4$ and 
$B_5$ would thus be injective.  See Figure \ref{fig:45rigidity}.

  \begin{figure}[!h]
    \begin{center}
    \input{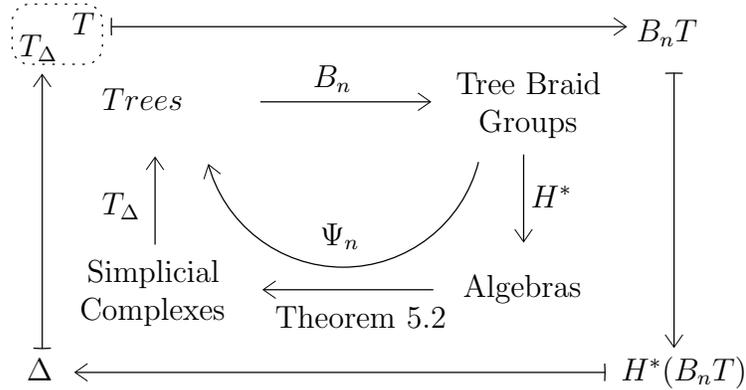}
    \end{center}

  \caption{A diagram showing the maps used in the proof of Theorem 
  \ref{thm:45rigidity}.  The proof of Theorem \ref{thm:45rigidity} shows 
  that, when $n = 4$ or $5$, $T_{\Delta}$ is homeomorphic to $T$.}
  
  \label{fig:45rigidity}
  \end{figure}

Let $T$ be a tree.  If $T$ is radial with central vertex $a$, then by 
\cite{FarleySabalka2a} we know that $H^*(B_nT)$ is free of rank 
$Y_n(deg(a))$ (see Theorem \ref{thm:radial}).  Then $\Delta$ is a 
collection of $Y_n(deg(a))$ vertices.  By definition, $T_{\Delta}$ is a 
tree with one essential vertex of degree $deg(a)$.  Thus, $T$ is 
homeomorphic to $T_{\Delta}$, as desired.

Assume that $T$ is not radial (and is sufficiently subdivided for $n$).  
Let $\Delta$ be the unique $1$-dimensional simplicial complex giving the 
exterior face algebra structure on $H^*(B_nT)$ given by Theorem 
\ref{thm:cohom}.  Fix a Morse $T$-embedding, so that the vertices of 
$\Delta$ may be labelled $Mc^*$, where each $c$ is a critical $1$-cell, 
as in Definition \ref{def:Delta}.

Since $T$ is not radial and $n \geq 4$, there exist critical $2$-cells 
(this is a consequence of Lemmas \ref{lem:Ncstructure} and 
\ref{lem:counting}, and is an exercise for the reader).  By Theorem 
\ref{thm:cohomology}, $H^2(B_nT)$ is nontrivial, so $\Delta$ contains 
some vertex with nonempty vertex neighborhood.  Let $v_0 = Mc_0^*$ be a 
vertex of $\Delta$ which is maximal under $\leqn$.  We will prove that 
the tree $T_\Delta$ defined by this choice of $\leqn$-maximal vertex is 
homeomorphic to $T$.

By the definition of $\leqn$, both the vertex neighborhood $N_{v_0}$ of 
$v_0$ and the $M$-cup neighborhood $N_{c_0}$ of $c_0$ are nonempty and 
maximal under inclusion.  Let $q_1$ denote the essential vertex in $T$ 
over which $c_0$ lies.  By Corollary \ref{cor:extremal}, $q_1$ is 
extremal and $CUB(c) = n-2$.

Consider a descendant $[v]$ of $[v_0]$.  Pick $v \in [v]$, and let $c$ 
be the critical $1$-cell such that $v = Mc^*$.  Let $q$ be the vertex 
over which $c$ lies.  By Lemma \ref{lem:CUBbound}, we have that $2 \leq 
CUB(c) \leq n -2$.  By Corollary \ref{cor:subequalneighborhoods}, either

  \begin{enumerate}

  \item $n = 5$, $q = q_1$, $d_c = d_{c_0}$, and $CUB(c) = CUB(c_0)-1 = 
  2$, or

  \item $q \neq q_1$, $q$ is in direction $d_{c_0}$ from $q_1$, and 
  $q_1$ is not in direction $d_c$ from $q$.

  \end{enumerate}
Consider if $CUB(c) = CUB(c_0) = n - 2$.  Then we are in the second case 
above.  Label the unique essential vertex of $T$ which is adjacent to 
$q$ and is in direction $d_c$ from $q$ by $q_{[v]}$.  By Corollary 
\ref{cor:equalneighborhoods}, the vertex $q$, the CUB direction $d_c$, 
and the CUB number $CUB(c)$ are independent of the choice of $v \in 
[v]$, so $q_{[v]}$ is well-defined.  By Lemma \ref{lem:counting}, every 
essential vertex of $T$ except $q_1$ will have a label of the form 
$q_{[v]}$.  As $T$ is a tree, the labels are unique.

We claim that the map $q_{[Mc^*]} \mapsto p_{[Mc^*]}$ from essential 
vertices of $T$ to vertices of $T_{\Delta}$ shows that $T_{\Delta}$ is 
defined and that this map induces the desired homeomorphism 
$\Psi_n$ between the two trees.  To prove our claim, we need to show:

  \begin{enumerate}[(i)]

  \item There exists a vertex labelled $q_{[Mc^*]}$ in $T$ if and only 
  if there exists a vertex labelled $p_{[Mc^*]}$ in $T_{\Delta}$,

  \item $q_{[M(c')^*]}$ is adjacent to $q_{[Mc^*]}$ if and only if
  $p_{[M(c')^*]}$ is adjacent to $p_{[Mc^*]}$, and

  \item $deg(q_{[Mc^*]}) = deg(p_{[Mc^*]})$.

  \end{enumerate}

(i) If $n = 4$, this is clear.  If $n = 5$, let $[v_0']$ be a child of 
$[v_0]$, and let $c_0'$ be a critical $1$-cell such that $v_0' = 
M^*c_0'$.  Consider if $CUB(c_0') = CUB(c_0) - 1 = 2$.  For $[v_0']$ to 
be a child of $[v_0]$, it follows from Corollary 
\ref{cor:subequalneighborhoods} that $c_0'$ lies over $q_1$ and 
$d_{c_0'} = d_{c_0}$.  By Corollary \ref{cor:equalneighborhoods}, 
$[v_0']$ is thus uniquely determined by the choice of $CUB(c_0') = 2$.  
Either a descendant $[Mc^*]$ of $[v_0]$ is also a descendant of $[v_0']$ 
or not.  By Corollary \ref{cor:subequalneighborhoods}, $[Mc^*]$ is a 
descendant of $[v_0']$ if and only if $CUB(c) = 2$, and is not a 
descendant of $[v_0']$ if and only if $CUB(c) = 3$.  Thus, by Lemma 
\ref{lem:counting}, $[v_0']$ has exactly half as many descendants as 
$[v_0]$.  If $[Mc^*]$ and $[M(c')^*]$ are distinct children of $[v_0]$ 
such that neither $[Mc^*]$ nor $[M(c')^*]$ is $[v_0']$, then $CUB(c) = 
CUB(c') = CUB(c_0) = 3$, as $[v_0']$ is uniquely determined by the 
choice $CUB(c_0') = 2$.  By Corollary \ref{cor:subequalneighborhoods}, 
it follows that $c$ and $c'$ both lie over $q_{[v_0]}$.  By Corollary 
\ref{cor:equalneighborhoods}, since $[Mc^*] \neq [M(c')^*]$, $d_c \neq 
d_c'$.  It follows from Corollary \ref{cor:subequalneighborhoods} that 
$[Mc^*]$ and $[M(c')^*]$ have no common descendants, let alone children.

Thus, $[v_0']$ is as in the definition of $T_{\Delta}$.  Note this child 
$[v_0']$ always exists, by Lemma \ref{lem:counting}. There exists a 
vertex labelled $q_{[Mc^*]}$ in $T$ if and only if $[Mc^*]$ is a 
descendant of $[v_0]$ and $CUB(c) = 3$, if and only if $[Mc^*]$ is a 
descendant of $[v_0]$ but not $[v_0']$, if and only if there exists a 
vertex labelled $p_{[Mc^*]}$ in $T_{\Delta}$.  This proves (i) in the 
case that $CUB(c_0') = 2$.

If $CUB(c_0') = 3$, let $[M(c')^*]$ denote the child of $[v_0]$ such 
that $CUB(c') = 2$.  If $[v_0]$ has any other child $[Mc^*]$ distinct 
from $[v_0']$ and $[M(c')^*]$, then $[M(c')^*]$ and $[Mc^*]$ have a 
common descendant (namely, $[M(c'')^*]$ where $CUB(c'') = 2$, $d_{c''} = 
d_c$, and $c''$ and $c$ lie over the same vertex in $T$).  Thus, unless 
$[v_0]$ has no other children, $[v_0']$ is not as in the definition of 
$T_{\Delta}$.  If $[v_0]$ has no other children, then by Corollary 
\ref{cor:subequalneighborhoods}, every descendant of $[v_0]$ except 
$[v_0]$ and $[M(c')^*]$ are also descendants of $[v_0']$.  Thus, unless 
$[v_0]$ has exactly $4$ descendants, $[v_0']$ is again not as in the 
definition of $T_{\Delta}$.

If $CUB(c_0') = 3$, $[v_0]$ has no other children, and exactly $4$ 
descendants, we are in the exceptional case in the definition of 
$T_{\Delta}$. We repeatedly apply Corollary 
\ref{cor:subequalneighborhoods} to analyze this case.  As $[v_0]$ has 
only one child $[v_0'] = [M(c_0')^*]$ with $CUB(c_0') = 3$, the vertex 
$q_{[v_0]}$ is adjacent to exactly two essential vertices in $T$: $q_1$ 
and $q_{[v_0']}$.  As $[v_0']$ has only one descendant $[Mc^*]$, 
$q_{[v_0']}$ must be adjacent to only $q_{[v_0]}$.  Let $x$, $y$, and 
$z$ denote the degees of vertices $q_1$, $q_{[v_0]}$, and $q_{[v_0']}$, 
respectively.  Then by Lemma \ref{lem:counting}, it follows that:
  $$Y_2(x) = |[v_0]|, \qquad Y_3(y) - Y_2(y) = |[Mc^*]|, \qquad 
  Y_2(z) = |N_{c}|.$$
Thus, in this exceptional case, we already have that $T$ is homeomorphic 
to $T_{\Delta}$.  This finishes the proof of (i), as either we must be 
in this exceptional case or $CUB(c_0') = 2$.

(ii) Let $v := Mc^*$ and $v' := M(c')^*$ be such that $q_{[v]}$ and 
$q_{[v']}$ are essential vertices in $T$.  Without loss of generality, 
assume $q_{[v]}$ is closer to $q_1$ than $q_{[v']}$.  Since $CUB(c) = 
CUB(c') = n-2$, by Corollary \ref{cor:subequalneighborhoods}, $q_{[v]}$ 
and $q_{[v']}$ are adjacent if and only if there are no critical 
$1$-cells $c''$ such that $N_{c'} \subsetneq N_{c''} \subsetneq N_{c}$.  
By definition, this holds if and only if $[v']$ is a child of $[v]$, if 
and only if $p_{[v]}$ is adjacent to $p_{[v']}$.

(iii)Let $c$ be a critical $1$-cell such that $q_{[Mc^*]}$ labels a 
vertex of $T$. Let $a' := q_{[Mc^*]}$, and let $a$ denote the vertex 
over which $c$ lies.  Then $a'$ is adjacent to $a$ and in direction 
$d_c$ from $a$.  If $a'$ is extremal, then by Lemma 
\ref{lem:Ncstructure}, $c' \in N_c$ if and only if $c'$ lies over $a'$ 
and $d_{c'}$ is the direction from $a'$ to $a$.  By Lemma 
\ref{lem:counting}, we have
  $$Y_{(n-2)}(deg(a')) = |N_c|.$$
If $a'$ is not extremal, then if follows from Corollary 
\ref{cor:subequalneighborhoods} that any child $[M(c')^*]$ of $[Mc^*]$ 
will satisfy either:

  \begin{enumerate}

  \item $n = 5$, $c'$ lies over $a$, $d_{c'} = d_c$, and $CUB(c') = 
  CUB(c)-1 = 2$, or

  \item $c'$ lies over $a'$, $CUB(c') = CUB(c)$, and $a$ is not in 
   direction $d_{c'}$ from $a'$.

  \end{enumerate}
In the former case, there is no vertex labelled $q_{[M(c')^*]}$ in $T$.  
In the latter case, by Lemma \ref{lem:counting}, we have
  $$Y_2(deg(a')) = |N_{c'}|.$$
Thus, regardless of whether $a'$ is extremal or not, 
  $$deg(q_{[Mc^*]}) = deg(p_{[Mc^*]}).$$
This finishes the proof of (iii), and of the lemma.

\end{proof}

  \begin{figure}[!h]
    \begin{center}
    \input{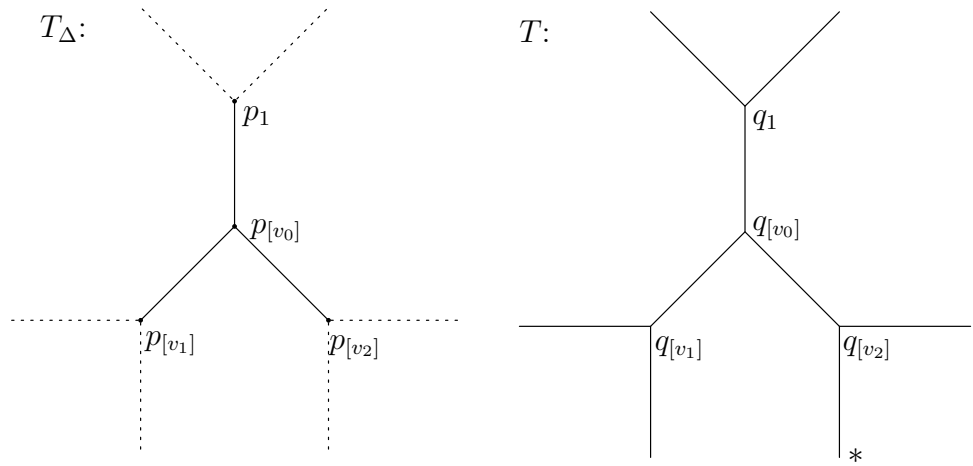}
    \end{center}

  \caption{The tree $T_{\Delta}$ and the isomorphism of Theorem 
  \ref{thm:45rigidity} for the tree $T_{min}$.}
  
  \label{fig:exTmin}

  \end{figure}

\begin{example}\label{ex:exTmin}

Consider the tree $T_{min}$, shown on the right of Figure 
\ref{fig:exTmin}.  On the left of Figure \ref{fig:exTmin} is the tree 
$T_{\Delta}$ generated in Example \ref{ex:treeTDelta}. The complex 
$\Delta$ used is in fact the simplicial complex underlying the exterior 
face algebra structure on $H^*(B_4T_{min})$.  The proof of Theorem 
\ref{thm:45rigidity} gives the isomorphism between $T_{\Delta}$ and $T$, 
induced by the map $q_{x} \mapsto p_{x}$. Here, $c_0 = 
(q_1,2,\left[\begin{array}{c}2\\1\\1\end{array}\right])$, $c_1 = 
(q_{[v_0]},2,\left[\begin{array}{c}0\\3\\1\end{array}\right])$, and $c_1 
= (q_{[v_0]},2,\left[\begin{array}{c}2\\1\\1\end{array}\right])$.

\end{example}

  \begin{figure}[!h]
    \begin{center}
    \input{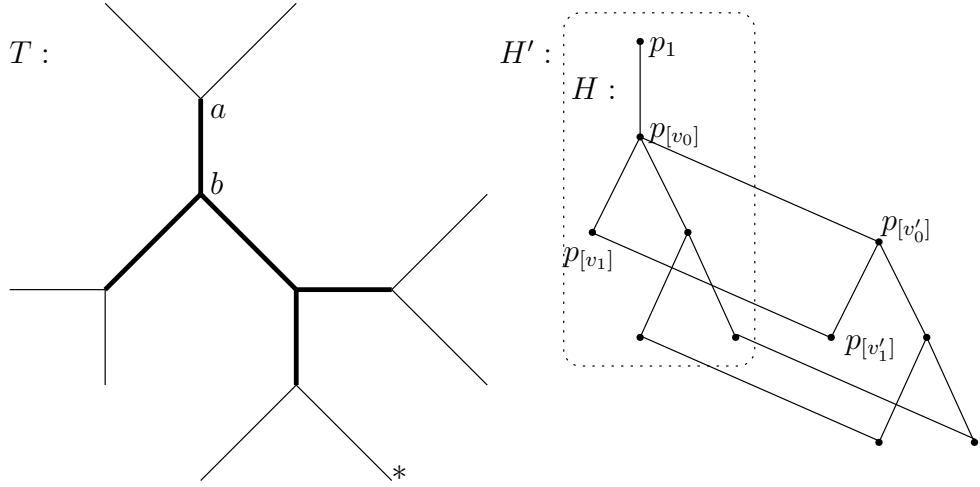}
    \end{center}

  \caption{An example with $n = 5$.}

  \label{fig:5example}
  
  \end{figure}

\begin{example}\label{ex:5example}

We wish to show an example of a partial computation for an $n=5$ strand 
tree braid group.  Consider the tree $T$ shown on the left of Figure 
\ref{fig:5example}, depicted with a Morse $T$-embedding.  Let $\Delta$ 
be the simplicial complex giving the exterior face algebra structure on 
$H^*(B_5T)$.  On the right of Figure \ref{fig:5example} is the graph 
$H'$ is the neighborhood heirarchy $H'$ for a vertex $v_0$ in the 
simplicial complex $\Delta$.  Since we have a Morse $T$-embedding, we 
can give $v_0$ a label: it is $Mc_0^*$, where $c_0$ is the critical 
$1$-cell $(a,2,\left[\begin{array}{c}3\\1\\1\end{array}\right])$.  
Similarly, $v_0' = M(c_0')^*$, $v_1 = Mc_1^*$, and $v_1' = M(c_1')^*$, 
where $c_0':= (a,2,\left[\begin{array}{c}2\\1\\2\end{array}\right])$, 
$c_1 := (b,2,\left[\begin{array}{c}0\\4\\1\end{array}\right])$, and 
$c_1':= (b,2,\left[\begin{array}{c}0\\3\\2\end{array}\right])$. The 
subtree $H$ of $H'$, circled in Figure \ref{fig:5example}, is the 
subtree from which $T_{\Delta}$ is grown.  The subtree $H$ is isomorphic 
to a subtree of $T$, shown in bold.

\end{example}

We may expand the proof of Theorem \ref{thm:45rigidity} to give us the 
following useful corollary:

\begin{corollary}[Determining $n$]\label{cor:45candeterminen}

If $G = B_nT$ is a tree braid group on either $n = 4$ or $5$ strands 
for which $T$ has at least 3 essential vertices, then $n$ may be 
determined from $G$.

\end{corollary}

\begin{proof}

Let $\Delta$ be the unique simplicial complex associated to $H^*(G)$, by 
Theorem \ref{thm:cohom}.  If $T$ has at least 3 essential vertices, $T$ 
is not radial, so $G$ is not free.  If $G$ is not free, there exists a 
vertex $v_0$ in the associated complex $\Delta$ with nonempty 
neighborhood and which is maximal under $\leqn$. If $n = 5$, then 
$[v_0]$ has a child $[v_0']$ as in the definition of $\Delta$. As $v_0$ 
is maximal uner $\leqn$, $v_0$ lies over an extremal vertex $q_1$ of 
$T$.  Let $q_{[v_0]}$ denote the unique essential vertex of $T$ adjacent 
to $q_1$.  As $T$ has at least $3$ essential vertices, there exists a 
third essential vertex $q \neq q_1$ adjacent to $q_{[v_0]}$.  Pick a 
Morse $T$-embedding, and let $c$ and $c'$ be critical $1$-cells that lie 
over $q_{[v_0]}$ and such that $d_{c'} = d_c$ is the direction from 
$q_{[v_0]}$ to $q$, $CUB(c) = 3$, and $CUB(c') = 2$ (by Lemma 
\ref{lem:counting}, such $c$ and $c'$ exist).  By Corollary 
\ref{cor:subequalneighborhoods}, it follows that $[Mc^*]$ is a child of 
$[v_0]$ but not of $[v_0']$, and $[M(c')^*]$ is a child of both $[Mc^*]$ 
and $[v_0']$.  Thus, if $n = 5$, there exist two children of $[v_0]$ 
with a common child.  If $n = 4$, then the neighborhood heirarchy of 
$[v_0]$ must be a tree, by Theorem \ref{thm:45rigidity}.  Thus, $n = 5$ if 
and only if there exist two children of $[v_0]$ with a common child.  
Note this characterization does not actually depend on fixing a Morse 
$T$-embedding. 

\end{proof}

It seems reasonable that some combinatorial argument can address the 
case when $T$ has exactly two essential vertices.  If this is done, the 
only exception to Corollary \ref{cor:45candeterminen} would be when $T$ 
has exactly one essential vertex.  But if $T$ has only one essential 
vertex, $T$ is radial, so $G$ is free.  Thus, modulo a combinatorial 
argument about when $T$ has exactly two essential vertices, Corollary 
\ref{cor:45candeterminen} would read: If $G$ is a tree braid group on 
either $n = 4$ or $5$ strands which is not free, then $n$ may be 
determined by $G$.

Theorem \ref{thm:45rigidity} also proves:

\begin{corollary} \label{cor:T_Delta}

The tree $T_{\Delta}$ constructed in Section \ref{sec:T_Delta} is 
defined if and only if $\Delta$ is a simplicial complex giving the 
exterior face algebra structure of $H^*(B_nT)$ for some tree $T$ and 
some $n = 4$ or $5$.\qed

\end{corollary}

The fact that one may reconstruct any tree $T$ up to homeomorphism given 
a ($4$ or $5$ strand) tree braid group $B_nT$ on $T$ is an artifact of 
the $1$-dimensionality of trees.  For instance, let $k > 1$ and consider 
the braid groups on a $k$-dimensional ball, and on a $k$-dimensional 
ball joined at a single point with a line segment.  On any number of 
strands, the two corresponding braid groups are isomorphic.

\subsection{More Strands}\label{sec:nstrands}

It would be nice to generalize Theorem \ref{thm:45rigidity} to say that, 
given a tree braid group on $n$ strands \emph{for any $n \geq 4$}, one 
may reconstruct the underlying tree.  As noted in Conjecture 
\ref{conj:extcohom}, we believe that tree braid groups on more than 5 
strands do not have an exterior face algebra structure on cohomology in 
general.  Thus the techniques used for $n = 4$ and $5$ probably do not 
apply for $n \geq 6$.  However, we still believe a generalization is 
possible:

\begin{conjecture}[Rigidity for an Arbtirary Number of 
Strands]\label{conj:nrigidity}

Let $G$ be an $n$ strand tree braid group where $n \geq 4$.  Then there 
exists a unique tree $T$ (up to homeomorphism) such that $G = B_nT$.

\end{conjecture}

Our intuition for conjecturing this generalization comes from the 
following situation. Let $T$ be a finite tree and let $n \geq 4$. Fix a 
Morse $T$-embedding.  As we have a Morse $T$-embedding, we may talk 
about critical cells.  We may still apply the results of Sections 
\ref{sec:prevresults} and \ref{sec:cobound} to $T$ and its 
critical cells. Consider a critical $1$-cell $c_0 = (a_0,d_0,\vec{x})$.  
Assume that $a_0$ is an extremal vertex.  Let $d_{c_0}$ denote the 
direction from $a_0$ towards every other essential vertex of $T$.  
Further assume that $x_{d_{c_0}} = n-2 + \epsilon_{c_0}(d_{c_0})$, where 
$\epsilon_{c_0}(d_{c_0})$ is the cup constant associated to $c_0$ in 
direction $d_{c_0}$.  We call a critical cell of this form 
\emph{extremal}. Under these assumptions, we claim that:

\begin{lemma}[Motivation for More Strands]\label{lem:Bn-2}

Let $c_0 = (a_0,d_0,\vec{x})$ be extremal.  Then
  $$H^*(B_{n-2}(T-\{a\})) \cong H^*(B_nT)\cup c_0^*.$$

\end{lemma}

We will prove this lemma momentarily.  Assuming Lemma \ref{lem:Bn-2} 
holds, we should be able to reconstruct $T$ by induction: reconstruct 
$T-\{a\}$ from $B_{n-2}(T-\{a\})$, and then simply `reattach' $a$ to 
$T$.

Of course, in general, when we are given a tree braid group we are not 
given any information about a Morse $T$-embedding or a classification of 
critical cells.  The heart of answering Conjecture \ref{conj:nrigidity} 
is in finding cohomology classes which behave like $c_0^*$ in the sense 
of Lemma \ref{lem:Bn-2}.

\begin{proof}[Proof of Lemma \ref{lem:Bn-2}]

We wish to construct the isomorphism on cohomology.  To do so, we define 
functions on several types of objects to build up to the desired map: on 
$b$-vectors for essential vertices $b \neq a$; on reduced $1$-cells; on 
$k$-forms; and finally on cohomology classes.

For any essential vertex $b$ of $T$ which is not $a$, define a function 
$at_b$ on $b$-vectors, which takes a $b$-vector and adds two (hence the 
name $at$) to the $(d_b)^{th}$ coordinate, where $d_b$ is the direction 
from $b$ to $a$.  By definition, $at_b$ is injective.  Note that $at_b$ 
is defined regardless of the length of the $b$-vectors.

Now define a function $at$ on reduced $1$-cells: 
$at(b,f,\vec{y}) := (b,f,at_b(\vec{y}))$.  This takes a reduced 
$1$-cell and adds two (again, hence the name) strands in the direction 
towards $a$.  We must require that $b \neq a$.  As with $at_b$, the map 
$at$ is injective.  Note $at$ is defined regardless of the length of 
$\vec{y}$.  For a reduced $k$-cell $s$, by Theorem \ref{thm:leq} there 
exists a unique collection $\{[c_1],\dots,[c_k]\}$ of equivalence 
classes of $1$-cells such that $[s]$ is the least upper bound of 
$\{[c_1],\dots,[c_k]\}$.  Define $at(s)$ to be the reduced 
representative of the least upper bound of 
$\{[at(c_1)],\dots,[at(c_k)]\}$.  We must require that no $c_i$ lies 
over $a$.  This corresponds to adding two strands to $s$ in the 
direction towards $a$ from the essential vertex $b$ closest to $a$ over 
which one of the $c_i$ lies.  Think of $at$ as placing two strands at 
the vertex $a$ in $T$.  Again, as the collection $\{[c_1],\dots,[c_k]\}$ 
is unique, $at$ is injective.

We generalize the map $at$ to $k$-forms, as follows.  For a basic 
$k$-form $\omega = f(b,\vec{y})dc_1\wedge\dots\wedge dc_k$, define 
$at(\omega)$ to be
  $$f(b,at_b(\vec{y}))d(at(c_1))\wedge \dots \wedge d(at(c_k)).$$
Here we require that $b \neq a$ and that each $c_i$ does not lie over 
$a$.  Again, $at$ is injective and defined regardless of the number of 
strands.

We claim that the desired isomorphism on cohomology is the map $AT$, where
$AT: H^*(B_{n-2}(T-\{a\})) \to H^*(B_nT)\cup c_0^*$ is given by
  $$[\omega] \mapsto [at(\omega)\wedge dc_0],$$
where $\omega$ is a cochain for $n-2$ strands on the tree $T-\{a\}$, 
thought of as a cochain for $n-2$ strands on the tree $T$.  
We prove that $AT$ is well-defined with a well-defined inverse map 
$AT^{-1}$, where
  $$AT^{-1}([\omega \wedge dc_0]) := [at^{-1}(\omega)]$$
if $at^{-1}(\omega)$ is defined; otherwise, $AT^{-1}([\omega \wedge 
dc_0]) := [0]$. If $AT$ and $AT^{-1}$ are well-defined, then $AT$ is a 
ring homomorphism, by definition and Proposition \ref{prop:d=delta}.  
Then, clearly from the definitions of $AT$ and $AT^{-1}$, they will be 
indeed inverse homomorphisms whose composition in either order is the 
identity, giving the desired isomorphism.

First, consider if $AT$ is well-defined.  Let $\omega = 
f(b,\vec{y})dc_1\wedge\dots\wedge dc_k \in [0]$ be a necessary $k$-form, 
and let $c = (b,f,\vec{y})$ denote the necessary $1$-cell for $\omega$.  
By the definition of necessary, $\{[c_1],\dots, [c_k], [c]\}$ has an 
upper bound $[s]$, and $f$ is the unique respectful edge in the reduced 
representative $s$ of $[s]$.  Consider $at(\omega)$.  If $at(\omega)$ is 
necessary, then $at(\omega)$ is cohomologous to $0$, and 
$[at(\omega)\wedge dc_0] = [0]$.  Consider if $at(\omega)$ is not 
necessary.  Clearly $f$ is such that $(b,f,at_b(\vec{y})) = at(c)$ is a 
reduced $1$-cell.  Moreover, $\{[at(c_1)],\dots, [at(c_k)], [at(c)]\}$ 
has an upper bound $[at(s)]$, by the definition of $at$.  Thus, if 
$at(\omega)$ is not necessary, $f$ must be disrespectful in $at(s)$, by 
the definition of necessary.  As $f$ is respectful in $s$ but 
disrespectful in $at(s)$, it must be that the two strands added to $s$ 
in $at(s)$ are the strands which make $f$ disrespectful.  A 
generalization of the Upper Bound Lemma, which we leave as an exercise 
for the sake of brevity, shows that 
$\{[at(c_1)],\dots,[at(c_k)],[at(c)],[c_0]\}$ has an upper bound $[s']$, 
and that moreover $f$ is the unique respectful edge in the reduced 
representative $s'$ of $[s']$.  Thus, $at(c)$ is necessary for the 
necessary $(k+1)$-form $at(\omega)\wedge dc_0$.  Thus, regardless of 
whether $at(\omega)$ is necessary or not, $[at(\omega) \wedge dc_0] = 
[0]$.  This shows $AT$ is well-defined.

Now, consider if $AT^{-1}$ is well-defined.  Let $\omega = 
f(b,\vec{y})dc_1'\wedge\dots\wedge dc_k'$ be a $k$-form such that 
$\omega \wedge dc_0 \in [0]$ is necessary, and let $c' = (b,f,\vec{y})$ 
denote the necessary $1$-cell for $\omega \wedge dc_0$.  By the 
definition of necessary, $\{[c_1'],\dots, [c_k'],[c_0], [c']\}$ has an 
upper bound $[s']$, and $f$ is the unique respectful edge in the reduced 
representative $s'$ of $[s']$.  Either $at^{-1}(\omega)$ is defined or 
not.

If $at^{-1}(\omega)$ is defined and $at^{-1}(\omega)$ is necessary, then 
$at^{-1}(\omega)$ is cohomologous to $0$, and there is nothing to prove. 
Consider if $at^{-1}(\omega)$ is defined but not necessary.  Clearly $f$ 
is such that $(b,f,at^{-1}_b(\vec{y})) = at^{-1}(c')$ is a reduced 
$1$-cell.  Moreover, the set of preimages $\{[at^{-1}(c_1)], \dots,  
[at^{-1}(c_k)], [at^{-1}(c)]\}$ has an upper bound $[at^{-1}(s)]$, by 
the definition of $at$.  Thus, if $at^{-1}(\omega)$ is not necessary, 
$f$ must be disrespectful in $at^{-1}(s)$, by the definition of 
necessary.  But if $f$ were disrespectful in $at^{-1}(s)$, then $f$ 
would be disrespectful in $at(s)$, a contradiction.  Thus, whether or 
not $at^{-1}(\omega)$ is necessary, if $at^{-1}(\omega)$ is defined then 
$AT^{-1}(\omega)$ is well defined.

If $at^{-1}(\omega)$ is not defined, then by definition either 
$at_b^{-1}(\vec{y})$ is not defined or $at^{-1}(c_i')$ is not defined 
for some $i \in \{1, \dots, k\}$.  In the former case, let $c'' := c'$ 
and in the latter case let $c'' := c_i'$.  In either case, express $c''$ 
as $(a'',d'',\vec{z})$, and let $d_0''$ denote the direction from $a''$ 
to $a_0$.  Note $[c'']$ and $[c_0]$ have a least upper bound $[s'']$ 
whose reduced representative $s''$ is $\sim$-equivalent to a face of 
$s$, since $c''$ and $c_0$ are both $\sim$-equivalent to faces of $s$. 
As $at^{-1}(c'')$ is not defined but $[s'']$ is, by the Upper Bound 
Lemma it follows that the edge $e_0$ of $c_0$ is respectful in $s''$.  
But then $e_0$ is respectful in $s$.  This contradicts the assumption 
that $\omega \wedge dc_0$ was necessary.  Thus, $at^{-1}(\omega)$ must 
be defined and cohomologous to $0$.  This proves that $AT^{-1}$ is 
well-defined as a homomorphism.

This finishes the proof.

\end{proof}

To conclude this section, we state a theorem which clarifies the 
behavior of free tree braid groups - that is, tree braid groups on 
radial trees or on fewer than $4$ strands:

\begin{theorem}[The Free Case]\label{thm:freecase}

Let $G$ be a free tree braid group.  If the number of strands $n$ for 
which $G = B_nT$ for some tree $T$ is known and $n \geq 4$, $T$ may be 
reconstructed up to homeomorphism.  Otherwise, $T$ may not be uniquely 
determined up to homeomorphism.

\end{theorem}

\begin{proof}

If $n < 4$, then not much can be said about $T$.  If $n \geq 4$, then 
$T$ is radial.  If $n$ is known, Theorem \ref{thm:radial} gives an 
explicit equation to solve for the degree of the unique essential 
vertex.  The solution is unique since all of the $Y_n$ functions are 
monotone increasing.  If $n$ is not known, then it is possible for free 
groups of ranks given by $r = Y_{n}(x)$ and $r' = Y_{n'}(x')$ to be such 
that $r = r'$, $n \neq n'$, and $x \neq x'$.  For instance: for $n = 4$, 
$x = 4$, $n' = 3$, and $x' = 5$, then $r = r' = 26$; for $n = 5$, $x = 
5$, $n' = 4$, and $x' = 6$, then $r = r = 155$.  In general, for $x' = 
n+1$ and $n' = x-1$, $r = r'$.  We note without proof, though, that 
empirically it appears this is the \emph{only} situation in which $r = 
r'$.

\end{proof}

\section{The Isomorphism Problem and Generalizations}\label{sec:isoproblem}

Theorems \ref{thm:45rigidity} and \ref{thm:freecase} allow us to solve 
the isomorphism problem for tree braid groups when we can reconstruct 
the defining trees by enumeration:

\begin{theorem}[The Isomorphism Problem]\label{thm:isoproblem}

Let $G$ and $G'$ be two groups be given by finite presentations, and 
assume that $G \cong B_nT$ and $G' \cong B_nT'$ for some positive 
integer $n$ and finite trees $T$ and $T'$.  If either:

  \begin{itemize} 
  \item $n = 4$ or $5$ or 
  \item at least one of $G$ or $G'$ is free, 
  \end{itemize} 
then there exists an algorithm which decides whether $G$ and $G'$ are 
isomorphic.  The trees $T$ and $T'$ need not be specified.  If one of 
$T$ and $T'$ has at least 3 essential vertices, then $n$ need not be 
specified.

\end{theorem}

\begin{proof}

If both $B_nT$ and $B_nT'$ are free, then the problem reduces to the 
isomorphism problem for free groups, which has a solution.  If exactly 
one of $B_nT$ or $B_nT'$ is free, then the groups cannot be isomorphic.  
If both of $B_nT$ and $B_nT'$ are not free, then we reconstruct the 
trees $T$ and $T'$ (and the value $n$ when one of $T$ or $T'$ has at 
least 3 essential vertices) from the corresponding groups, which can be 
done, according to Theorem \ref{thm:45rigidity}.  We need to show that 
we can reconstruct the trees algorithmically.  We begin by extrapolating 
$T$ from $B_nT$.

Let the \emph{degree} of a tree denote the sum of degrees of all 
essential vertices of the tree.  Note that the number of trees up to 
homeomorphism of any given degree is finite.  Let the \emph{length} of a 
presentation for a group denote the sum of lengths of the defining 
relators in the presentation plus the number of generators.  Note that 
the number of distinct presentations for a group of any given length is 
finite.  For any given homomorphism between two groups and any given 
presentations of both, let the \emph{length} of the homomorphism with 
respect to the presentations be the sum of the lengths of the images of 
the generators of the first group under the homomorphism.  Note that, 
for fixed groups and presentations, the number of homomorphisms of a 
given length is finite.

Enumerate all trees up to homeomorphism in some order from lesser degree 
to greater. For a given tree and a given $n$, enumerate all (reduced) 
presentations of the corresponding $n$ strand tree braid groups in some 
order from lesser length to greater.  For a given tree, a given $n$, and 
a given presentation of the corresponding $n$ strand tree braid group, 
enumerate all homomorphisms to $B_nT$ in some order from lesser length 
to greater.  Finally, enumerate all homomorphisms from $n$ strand tree 
braid groups to $B_nT$ by diagonalization.

As we know that $B_nT$ is a tree braid group, eventually in this 
enumeration there will be an isomorphism.  Define an algorithm to 
extrapolate the tree $T$ (up to homeomorphism) and the value $n$ from 
the group $B_nT$ by running through the above enumeration and finding 
the first isomorphism, and outputting the corresponding tree $T$ and 
value $n$.  Similarly, define an algorithm to extract the tree $T'$ (up 
to homeomorphism) from $B_nT'$ (which can be shortened, since we now 
know $n$).

By Theorems \ref{thm:45rigidity} and \ref{thm:freecase}, the trees $T$ 
and $T'$ are uniquely determined up to isomorphism for the cases of the 
theorem.  Note Corollary \ref{cor:45candeterminen} implies that $n$ is 
uniquely determined up to isomorphism if either $T$ or $T'$ contain at 
least 3 essential vertices.

Now, the isomorphism problem for $n$ strand tree braid groups has been 
reduced to solving the homeomorphism problem for trees.  But there 
exists an algorithm for solving this problem, so we are done.

\end{proof}

\section{Right-angled Artin groups}\label{sec:RAAGs}

Let $\Delta$ be a finite graph.  For the purposes of this paper, we will 
always assume $\Delta$ is \emph{simple}: it contains no nontrivial 
embedded edge loops having less than 3 edges.  The \emph{right-angled 
Artin group} $G(\Delta)$ associated to $\Delta$ is defined as follows. 
The group $G(\Delta)$ is generated by the vertices of $\Delta$, and the 
only relations are that two generators commute if and only if the 
corresponding vertices are connected by an edge.  For more information 
on right-angled Artin groups, see \cite{Charney}.

The purpose of this section is to state analogues of the main theorems 
of this paper, but for right-angled Artin groups.  Graph braid groups 
are closely related to right-angled Artin groups, so it is not 
surprising that similar theorems hold for each.  The differences between 
the two classes of groups is highlighted, however, by the comparative 
difficulty of proofs for graph braid groups.

For a finite simplicial graph $\Delta$, the \emph{flagification} of 
$\Delta$, denoted $\overline{\Delta}$, is the simplicial complex whose 
vertices are the vertices of $\Delta$, and for which $k$ vertices span a 
$k$-simplex in $\overline{\Delta}$ if and only if the $k$ vertices span 
a complete subgraph of $\Delta$.  We say $\Delta$ is \emph{flag} if 
$\Delta = \overline{\Delta}$.  The analogue of Theorem \ref{thm:cohom} 
is already known, and was proven by Charney and Davis:

\begin{theorem}[RAAG cohomology]\cite{CharneyDavis}\label{thm:flag}

Let $\Delta$ be a finite simple graph.  Then $H^*(G(\Gamma)) \cong 
\Lambda(\overline{\Delta})$.

\end{theorem}

We remark that in \cite{FarleySabalka2a} it was proven that a tree braid 
group $B_nT$ is right-angled Artin if and only if $n \leq 3$ or $T$ is 
linear.  The method of proof essentially involved proving, in the 
non-right-angled Artin cases, that even if the cohomology ring of a tree 
braid group is an exterior face algebra, it cannot be an exterior face 
algebra over a flag simplicial complex.

There are two analagues of Theorem \ref{thm:45rigidity}.  One is an 
important rigidity result and appears in the literature, due to Droms.  
The other does not yet appear in the literature, but follows from 
Droms's theorem and Gubeladze's theorem:

\begin{theorem}[RAAG defining graph 
rigidity]\cite{Droms}\label{thm:RAAGrigidity}

Let $\Delta$ and $\Delta'$ be finite simple graphs.  Then $G(\Delta) 
\cong G(\Delta')$ if and only if $\Delta \cong \Delta'$.

\end{theorem}

\begin{theorem}[RAAG cohomology rigidity]\label{thm:RAAGcohomrigidity}

There is a setwise bijection between right-angled Artin groups up to 
isomorphism and their cohomology rings up to isomorphism.

\end{theorem}

\begin{proof}

By Theorem \ref{thm:flag}, cohomology rings of right-angled Artin groups 
are always exterior face algebras over flag simplicial complexes. 
Moreover, every flag simplicial complex corresponds to the cohomology 
ring of some right-angled Artin group, as there is a bijective 
correspondence between finite simplicial graphs and flag simplicial 
complexes.  To get from graphs to simplicial complexes, add a 
$k$-simplex whenever its $1$-skeleton is present; to get from flag 
simplicial complexes to graphs, take the $1$-skeleton.  Gubeladze's 
Theorem (Theorem \ref{thm:Gubeladze}) shows that there is thus a 
bijective correspondence between cohomology rings of right-angled Artin 
groups and flag simplicial complexes.  As there are bijections between 
right-angled Artin groups and finite simplicial graphs, finite 
simplicial graphs and finite flag simplicial complexes, and finite flag 
simplicial complexes and cohomology rings of right-angled Artin groups, 
this proves the theorem.

\end{proof}

We end with the analogue of Theorem \ref{thm:isoproblem} for 
right-angled Artin groups.  Theorem \ref{thm:RAAGrigidity} suggests that 
the isomorphism for right-angled Artin groups may be solved by 
reconstructing $\Delta$ given $G(\Delta)$.  One simply needs an 
algorithm for the reconstruction.

\begin{theorem}[RAAG isomorphism problem]

Let $G$ and $G'$ be two groups be given by finite presentations, and 
assume that $G \cong G(\Delta)$ and $G' \cong G(\Delta')$ for some 
finite simple graphs $\Delta$ and $\Delta'$.  Then there exists an 
algorithm which decides whether $G$ and $G'$ are isomorphic.  The graphs 
$\Delta$ and $\Delta'$ need not be specified.

\end{theorem}

The proof of this theorem is similar to the proof of Theorem 
\ref{thm:isoproblem}, and is a standard diagonalization argument.  The 
theorem follows directly from Theorem \ref{thm:RAAGrigidity}, and is, 
more or less, generally known.

\begin{proof}

Let the \emph{degree} of a graph denote the sum of degrees of all 
essential vertices.  Note that the number of graphs up to homeomorphism 
of any given degree is finite.  Let the \emph{length} of a presentation 
for a group denote the sum of lengths of the defining relators in the 
presentation plus the number of generators.  Note that the number of 
distinct presentations for a group of any given length is finite.  For 
any given homomorphism between two groups and any given presentations of 
both, let the \emph{length} of the homomorphism with respect to the 
presentations be the sum of the lengths of the images of the generators 
of the first group under the homomorphism.  Note that, for fixed groups 
and presentations, the number of homomorphisms of a given length is 
finite.

Enumerate all graphs in some order from lesser degree to greater. For a 
given graph, enumerate all (reduced) presentations of the corresponding 
right-angled Artin group in some order from lesser length to greater.  
For a given graph and a given presentation of the corresponding 
right-angled Artin group, enumerate all homomorphisms to $G$ in some 
order from lesser length to greater.  Finally, enumerate all 
homomorphisms from right-angled Artin groups to $G$ by diagonalization.

As we know that $G$ is a right-angled Artin group, eventually in this 
enumeration there will be an isomorphism.  Define an algorithm to 
extrapolate the graph $\Delta$ from $G$.  Similarly, extrapolate 
$\Delta'$ from $G'$.  By \ref{thm:RAAGrigidity}, $\Delta$ and $\Delta'$ 
are uniquely determined.  Thus, the isomorphism problem for right-angled 
Artin groups reduces to the isomorphism problem for graphs.

\end{proof}

\bibliography{refs-S3}
\bibliographystyle{plain}

\end{document}